\documentclass[12pt,reqno,oneside]{amsart}
\usepackage{fullpage}
\usepackage{etex}
\usepackage{amsmath,amssymb, amsthm, amsfonts}
\usepackage[english]{babel}

\usepackage
[plainpages=false,
pdfpagelabels
]{hyperref}
\hypersetup{ hidelinks = true, }

\usepackage{microtype}

\hypersetup{linktocpage}

\usepackage{enumerate}
\usepackage[capitalize]{cleveref}

\usepackage[all]{xy}
\usepackage[numbers,sort,compress]{natbib}

\usepackage{ifthen}
\usepackage{xspace}
\usepackage{multirow}
\usepackage[table,xcdraw]{xcolor}
\usepackage{pgfplotstable}
\usepackage{tikz}
\usetikzlibrary{calc,matrix,arrows,positioning,calc,decorations.markings,decorations.pathreplacing,patterns}
\usepackage{tikz-cd}
\usepackage[textwidth=.8in,colorinlistoftodos,textsize=scriptsize]{todonotes}
\usepackage[font=small,labelfont=bf]{caption}
\usepackage{fixltx2e}
\usepackage{etex}
\usepackage[labelfont=rm]{subcaption}
\usepackage{mdwlist}
\usepackage{bbm}
\usepackage[section]{placeins}
\usepackage{hhline}
\newcommand{\bparagraph}[1]{\indent\paragraph{\textit{#1}}}

\newcommand\Perp{\protect\mathpalette{\protect\independenT}{\perp}}
\def\independenT#1#2{\mathrel{\rlap{$#1#2$}\mkern2mu{#1#2}}}

\newcommand{\au}{\textup{a}}
\newcommand{\bu}{\textup{b}}
\newcommand{\cu}{\textup{c}}
\newcommand{\du}{\textup{d}}
\newcommand{\eu}{\textup{e}}
\newcommand{\fu}{\textup{f}}
\newcommand{\forest}{F^G}
\newcommand{\gu}{\textup{g}}
\newcommand{\hu}{\textup{h}}
\newcommand{\iu}{\textup{i}}
\newcommand{\ju}{\textup{j}}
\newcommand{\ku}{\textup{k}}
\newcommand{\lu}{\textup{l}}
\newcommand{\Au}{\mathbf{a}}
\newcommand{\Bu}{\mathbf{b}}
\newcommand{\Cu}{\mathbf{c}}
\newcommand{\Du}{\mathbf{d}}

\newcommand{\Gu}{\textup{G}}
\newcommand{\Hu}{\textup{H}}
\newcommand{\Iu}{\textup{I}}
\newcommand{\Ju}{\textup{J}}
\newcommand{\Ku}{\textup{K}}
\newcommand{\Lu}{\textup{L}}
\newcommand{\muu}{\textup{m}}
\newcommand{\met}[1]{\operatorname{met} #1}
\newcommand{\nuu}{\textup{n}}
\newcommand{\ou}{\textup{o}}

\newcommand{\A}{\mathcal{A}}
\newcommand{\B}{\mathcal{B}}
\newcommand{\C}{\mathcal{V}}
\newcommand{\Cr}{\mathrm{\Gamma}}
\newcommand{\id}{\mathrm{id}}
\newcommand{\eh}{\mathcal E}
\newcommand{\F}{\mathcal{F}}
\newcommand{\field}{K}
\newcommand{\vfa}{X}
\newcommand{\vfb}{Y}
\newcommand{\G}{\mathcal{G}}
\newcommand{\I}{\mathcal{I}}
\newcommand{\J}{\mathcal{J}}
\newcommand{\pfd}{p.f.d.\@\xspace}
\newcommand{\Pleft}{P_{\mathrm{left}}}
\newcommand{\Q}{\mathbb{Q}}
\newcommand{\R}{\mathbb{R}}
\newcommand{\T}{\mathcal{T}}
\newcommand{\Z}{\mathbb{Z}}

\newcommand{\Set}{\mathbf{Set}}
\newcommand{\Rips}{\mathcal{V}}

\newcommand{\asplit}{\mathrm{asp}}
\newcommand{\spl}{\mathrm{sp}}
\newcommand{\jn}{\mathrm{jn}}
\newcommand{\total}{\mathrm{tot}}

\DeclareMathOperator{\diam}{diameter}

\DeclareMathOperator{\Gall}{Gall}
\DeclareMathOperator{\len}{length}
\DeclareMathOperator{\lens}{lengths}
\DeclareMathOperator{\nde}{nd}

\DeclareMathOperator{\grSum}{\vee}
\DeclareMathOperator{\Trim}{Trim}

\newtheorem{theorem}{Theorem}[section]
\newtheorem*{utheorem}{Theorem}
\newtheorem{proposition}[theorem]{Proposition}
\newtheorem{lemma}[theorem]{Lemma} 
\newtheorem{corollary}[theorem]{Corollary}

\theoremstyle{definition}
\newtheorem{definition}[theorem]{Definition}
\newtheorem{example}[theorem]{Example}
\newtheorem{remark}[theorem]{Remark}

\begin{document}

\title{Quantifying Genetic Innovation:\\ Mathematical Foundations for \\ the Topological Study of Reticulate Evolution}

\author{Michael Lesnick}
\address{Lesnick: SUNY Albany, NY, USA}
\email{mlesnick@albany.edu}
\author{Ra\'ul Rabad\'an}
\address{Rabad\'an: Columbia University, New York, NY, USA}
\email{rr2579@cumc.columbia.edu}
\author{Daniel I. S. Rosenbloom}
\address{Rosenbloom: Merck Research Laboratories, Rahway, NJ, USA}
\email{daniel.rosenbloom@merck.com}

\begin{abstract}
A topological approach to the study of genetic recombination, based on persistent homology, was introduced by Chan, Carlsson, and Rabad\'an in 2013. This associates a sequence of signatures called \emph{barcodes} to genomic data sampled from an evolutionary history. In this paper, we develop theoretical foundations for this approach.

First, we present a novel formulation of the underlying inference problem. Specifically, we introduce and study the \emph{novelty profile}, a simple, stable statistic of an evolutionary history which not only counts recombination events but also quantifies how recombination creates genetic diversity. We propose that the (hitherto implicit) goal of the topological approach to recombination is the estimation of novelty profiles.

We then study the problem of obtaining a lower bound on the novelty profile using barcodes. We focus on a low-recombination regime, where the evolutionary history can be described by a directed acyclic graph called a \emph{galled tree}, which differs from a tree only by isolated topological defects. We show that in this regime, under a complete sampling assumption, the $1^\mathrm{st}$ barcode yields a lower bound on the novelty profile, and hence on the number of recombination events. For $i>1$, the $i^{\mathrm{th}}$ barcode is empty.  In addition, we use a stability principle to strengthen these results to ones which hold for any subsample of an arbitrary evolutionary history.  To establish these results, we describe the topology of the Vietoris--Rips filtrations arising from evolutionary histories indexed by galled trees.

As a step towards a probabilistic theory, we also show that for a random history indexed by a fixed galled tree and satisfying biologically reasonable conditions, the intervals of the $1^{\mathrm{st}}$ barcode are independent random variables. Using simulations, we explore the sensitivity of these intervals to recombination.
\end{abstract}

\maketitle

\tableofcontents

\section{Introduction}

\subsection{Recombination}
Recombination is a process by which the genomes of two parental organisms combine to form a new genome.  Like genetic mutation, recombination gives rise to genetic diversity in evolving populations. But unlike mutation, recombination can unite advantageous traits which have arisen in separate lineages, or rescue an advantageous trait from an otherwise disadvantageous genetic background.  In these ways, recombination hastens the pace at which adaptive genetic novelty arises.  

Evolving populations can be studied by observing genetic sequences obtained from a sample of organisms.  Several methods exist to estimate or bound the number of recombination events that have occurred in the ancestry of a sample and to identify the genomic locations where recombination may have occurred \cite{Hudson:1985wh, Myers:2003tv, Song:2005cg}. Yet these methods do not reveal how recombination generates genetic diversity: Recombination between two very distinct parents may create a genetically very novel offspring, contributing substantial diversity to the population, but recombination between genetically similar parents can only create genetically similar offspring, contributing little diversity.  

\subsection{Novelty Profiles}
In this work, we introduce a simple, stable statistic of an evolving population, the \emph{novelty profile}, which quantifies how recombination contributes to genetic diversity.  To define the novelty profile, we first need to select a formal model of an evolving population.  We call the model we consider in this paper an \emph{evolutionary history}.  An evolutionary history $E$ is a directed acyclic graph $G$, together with a set $E_v$ at each vertex $v$ of $G$, satisfying certain conditions.  We call $G$ a \emph{phylogenetic graph}, and say that $G$ \emph{indexes} $E$.  Each vertex of $G$ represents an organism, each edge of $G$ represents a parental relationship, and each $E_v$ specifies the genome of the organism $v$.  See \cref{Sec:Phylogenetic_graphs_and_populations} for the formal definition of an evolutionary history and an illustration.

The novelty profile of an evolutionary history is simply a list of $k$ monotonically decreasing numbers, where $k$ is the number of recombination events in the history.  Roughly, each number measures the contribution to genetic diversity of one recombinant.  We introduce two versions of this statistic, the \emph{temporal} and \emph{topological} novelty profiles.  The definition of the temporal novelty profile is very elementary and intuitive, but depends on a specification of the time at which each organism is born.  Moreover, the temporal novelty profile, while stable to perturbations (i.e., small changes) of the genomes,  is unstable to perturbations of the birth times.  In contrast, the topological novelty profile is defined in a way that does not depend on birth times.  It is also stable to perturbations of the genomes.  The topological novelty profile is a lower bound for the temporal novelty profile, in the sense that the $i^{\mathrm{th}}$ element of the topological novelty profile is less than or equal to the $i^{\mathrm{th}}$ element of the temporal novelty profile for all $i$.  

\subsection{Prior Topological Work on Recombination}

The broader idea of quantifying the scale of recombination events, in addition to their number, is already present in earlier topological work on recombination \cite{chan2013topology, Emmett:2014us, Emmett:2014um, Camara:2016eh, Camara:2016kl,parida2015topological}; the recent textbook \cite{rabadan2019topological} provides an detailed introduction.  Our definitions of novelty profiles are inspired by some of this previous work, and one of our main objectives here is to use novelty profiles to develop mathematical foundations for that work.

In the previous work, a popular topological data analysis method called \emph{persistent homology} is used to associate a sequence $\B{}_0(S)$, $\B{}_1(S)$, $\B{}_2(S),\ldots$ of objects called \emph{barcodes} to an arbitrary sample $S$ of an evolving population.  Each barcode is a collection of intervals $[a,b)$ on the real line.   In \cite{chan2013topology}, it is shown that, under a standard \emph{infinite sites} assumption ruling out multiple mutations at the same genetic site, if no recombination occurs in the population's history, then $\B{}_i(S)$ is empty for all $i\geq 1$; see \cref{Sec:Barcodes_From_Trees}.  Hence a non-empty barcode $\B{}_i(S)$ for any $i\geq 1$ certifies that recombination has occurred at some point in the history.  Within simulations of evolving populations, the number of intervals in the first barcode $\B{}_1(S)$ has been observed to increase with the simulated recombination rate \cite{Camara:2016eh,chan2013topology}.  
Moreover, it has been observed empirically that the endpoints of the intervals in the barcode $\B{}_1(S)$ depend on the genetic scale at which recombination events occur \cite{Emmett:2014us, Camara:2016kl,parida2015topological}.  
For instance, in studies of population admixture (i.e., interbreeding between distantly related subpopulations), intervals in $\B{}_1(S)$ with large values for the endpoints have been observed to appear in the barcode only in the presence of admixture \cite{parida2015topological,Camara:2016kl}.

While these findings together suggest that the barcode encodes information about both the number of recombination events and the contribution of recombination to genetic diversity,  the precise statistical nature of the relationship between barcodes and recombination has not been made clear.  In this paper, we make progress towards understanding this relationship.

\subsection{Barcodes as Lower Bounds of Novelty Profiles}
We propose that the central inference problem implicit in the previous topological work on evolution is the estimation of novelty profiles.  Given this, we are led to ask how the barcodes $\B{}_i(S)$ of genomic data studied in the previous work perform as estimators of the novelty profile.  It is known that these barcodes can fail to detect recombination events,  even in the simplest and most favorable circumstances \cite{chan2013topology}, so one expects the barcodes to encode only partial information about the novelty profile.  

In this paper, we study barcodes as lower bounds on the novelty profile.  For context, we note that computable lower bounds on numbers of recombination events play a key role in the study of recombination \cite{Hudson:1985wh, Myers:2003tv,Song:2005cg}.  Our lower bounds are in a similar spirit.  Similarly, in topological data analysis, the idea of using barcodes to formulate lower bounds (e.g., on the Gromov-Hausdorff distance between compact metric spaces) is fundamental---it lies at the heart of the well-known stability theory for persistence \cite{bauer2015induced,chazal2012structure,chazal2014persistence}; see \cref{Sec:Topological_Preliminaries}.

To formulate our bounds, we first restrict attention to a low-recombination regime, where the evolutionary histories are indexed by \emph{galled trees}. Galled trees are directed acyclic graphs that are almost trees, in a sense: They may have cycles, but these cycles are topologically separated from one another; see \cref{Sec:Galled_Trees} for the precise definition. 
Galled trees have received considerable attention in the phylogenetics literature as computationally convenient models of evolution with infrequent recombination \cite{huson2010phylogenetic}, \cite{gusfield2014recombinatorics}.  They have been of interest primarily because certain phylogenetic network reconstruction problems that are computationally hard in general admit polynomial-time solutions when restricted to galled trees.
To clarify the biological relevance of galled tree models of evolution, in \cref{Sec:Appendix} we study the probability $P$ that a galled tree correctly models an evolutionary history.  We work with a coalescent model of evolution, a standard model in population genetics.  We show that for this model, $P$ can be computed by solving a linear system of equations, and we observe that for a fixed population size, $P$ tends to 1 as the recombination rate tends to 0; see also \cref{Rem:Galled_Low_Recomb}.

We observe that for evolutionary histories indexed by galled trees, the temporal and topological novelty profiles are equal (\cref{Equality_Of_Nov_Profiles_For _Galled_Trees}).  Our main result relating barcodes to  recombination in the galled tree setting is the following (see \cref{Thm:Lower_Bound} and the preceding definitions for the precise formulation):

\begin{utheorem}
Let $\eh$ be an evolutionary history indexed by a galled tree.
\begin{enumerate}
\item[(i)] The set of lengths of intervals in the barcode $\B_1(\eh)$ is a lower bound on the novelty profile.  In particular, the number of intervals in $\B_1(\eh)$ is a lower bound on the number of recombination events in $\eh$.  
\item[(ii)] $\B_i(\eh)$ is empty for $i\geq 2$.
\end{enumerate}
\end{utheorem}

Part (i) of the theorem does not hold for barcodes $\B_1(S)$ of arbitrary samples $S\subseteq \eh$ (\cref{Ex:Subsample_Counterexample}).  However, using a well-known stability property of persistent homology, we observe that the theorem extends to an approximate version which holds for an arbitrary sample $S$, even in the presence of noise (\cref{Cor:GalledTreeDetInference_withnoise}).  The quality of the approximation depends on the similarity of the geometries of $S$ and $\eh$, as measured by the \emph{Gromov-Hausdorff distance}.  Along similar lines, the theorem further extends to an approximate version for histories indexed by arbitrary phylogenetic graphs (\cref{Cor:Inference_Arbitrary_G}); here, the quality of approximation is controlled by the number of mutations which must be ignored to obtain a history indexed by a galled tree.

These results are deterministic; in cases where the history is sampled at random from a known distribution, one hopes to be able to obtain stronger probabilistic results.  As a first step towards such results, we show in \cref{Sec:Independence} that for a random history indexed by a fixed galled tree $G$ and satisfying a biologically reasonable independence condition, 
the intervals of the $1^\mathrm{st}$ barcode are independent random variables indexed by the recombinants of $G$.

We then study the distributions of these random variables via simulation, for one class of random models of genetic sequence evolution.  Our simulation results indicate that even when we have sampled all individuals in the evolutionary history, the barcode is usually a rather loose lower bound on the novelty profile.  For example, in the most favorable circumstances, a recombinant of high novelty is detected in our simulations about a third of the time.  Nevertheless, the barcodes provide partial information about the novelty profile.  Notably, we observe in our simulations that when a recombination event of novelty $n$ is detected by the barcode, the average length of the corresponding interval is approximately $c+d\sqrt{n}$ for constants $c$ and $d$.

\subsection{Remarks on the Practical Inference of Novelty Profiles}
The primary motivation for our results relating barcodes to novelty profiles is to further our understanding of the prior topological work on recombination.  Our results clarify what the barcodes of genomic  data do and do not tell us about recombination, and they highlight the mathematical challenges involved in understanding the connection between barcodes and recombination more fully.

That said, we are hopeful that novelty profiles can find practical use in biology applications.  For this, we need to be able to infer (statistics of) novelty profiles from real-world genomic data. It is thus natural to ask whether our bounds relating barcodes to novelty profiles can be applied in practice to such inference.  In \cref{Sec:Assumptions}, we consider this question in detail.  While we find that such inference may be possible in some circumstances, for typical genomic data the assumptions underlying our bounds seem too strong to apply in a useful way.  Thus, we see our bounds as a first step towards a more applicable theory for inferring information about novelty profiles from topological statistics of genomic data.  Some possible directions forward along these lines are discussed in \cref{Sec:Discussion}.  

While we are indeed hopeful that our bounds can be strengthened to obtain results that are more readily applied in practice, we expect that in the near term, it may be more fruitful for practical applications to pursue alternative approaches to the inference of (statistics of) novelty profiles.  For smaller samples, it may be effective to estimate the novelty profile by first estimating the full evolutionary history of the sample; there is well-developed technology for this \cite{gusfield2014recombinatorics,Rasmussen:2014cq}.  For larger samples, where estimation of a full evolutionary history is computationally infeasible, a machine learning approach may be the most practical way forward; for this, it may be possible to adapt ideas from recent work on the learning of recombination rates from topological features \cite{humphreys2019fast}.  While a full development of these ideas is beyond the scope of this paper, we discuss them briefly in \cref{Sec:Alternative_Strategies}.

\subsection{Other Theoretical Work on the Topological Approach to Recombination} Theoretical foundations for the application of persistent homology to recombination have also been studied in recent work of C\'amara et al. \cite{Camara:2016kl} and Parida et al. \cite{parida2015topological}, though from a rather different angle than ours.  \cite{Camara:2016kl} considers connections to the problem of constructing \emph{minimal ancestral recombination graphs} (ARGs) for single-breakpoint models of recombination. (An ARG roughly corresponds to what we call an evolutionary history in this paper; a minimal ARG is a history of a given set of genome sequences with as few recombination events as possible.) In contrast, we do not consider ARG reconstruction or constrain recombination to a single-breakpoint model.

The theory developed in \cite{parida2015topological} concerns population admixture.  The work models the evolution not only of individual organisms, but also of entire populations, and defines barcodes signatures both at the individual level and at the population level.  It is shown that under natural assumptions on the inter-population and intra-population genetic distances, one can deduce information about barcodes at the population level from barcodes at the individual level.  However, no direct theoretical relationship is established between the barcodes and recombination or admixture.

While our aim and technical approach differ from these previous works, we do share the common goal of understanding persistent homology as a signature of genetic recombination.

\subsection{Mathematical Contributions}
One key feature of the barcode signatures of recombination studied here is that they depend only on the \emph{metric structure} on an evolutionary history, i.e., the genetic distances between organisms---in our formalism, the Hamming distance, or monotonic transformations thereof such as the Jukes-Cantor distance.  In fact, these barcodes are given by a standard construction which associates barcodes to any finite metric space $M$.  In this construction, one first builds a 1-parameter family of simplicial complexes $\Rips(M)$ called the \emph{Vietoris--Rips filtration (VRF)}.  

The topological study of VRFs is a central theoretical problem in topological data analysis.  While some fundamental results about VRFs are well known, including a stability theorem \cite{chazal2009gromov,chazal2014persistence,blumberg2017universality}, relatively little is known about concrete computations of the topology of VRFs, outside of special cases; even for points distributed uniformly on a circle or ellipse, the problem is already non-trivial, and has been the subject of recent research \cite{adamaszek2017vietoris,adamaszek2017vietorisB}.  

The result of Chan et al. \cite{chan2013topology}, that $\B_i(S)=\emptyset$ for $i\geq 1$ when $S$ is a sample of a history with no recombination, amounts to a proof that the VRF of a \emph{tree-like} metric space is topologically trivial, up to multiplicity of connected components; see \cref{Prop:Tree_Like_MS}.  Analogously, the mathematical heart of our main results about recombination for galled trees is a topological description of $\Rips(\eh)$, for $\eh$ an evolutionary history indexed by a galled tree, regarded as a metric space: We use discrete Morse theory \cite{forman2002user} to show that each simplicial complex in $\Rips(\eh)$ is homotopy equivalent 
to a disjoint union of bouquets of circles, where each circle corresponds to a unique recombination event.  Moreover, we completely describe the topological behavior of the inclusion maps in $\Rips(\eh)$ and give bounds on the number of intervals in $\B_1(\eh)$.  For the precise statements, see \cref{Lem:Loop_Almost_Linear,Thm:ALMS,Prop:Decomposable_Populations}.

Our topological study of $\Rips(\eh)$ hinges on the study of the VRFs of \emph{almost linear metric spaces}; we say a metric space is almost linear if (up to isometry) it is obtained from a finite subspace of $\R$ by adding a single point.  In brief, almost linear metric spaces enter into our analysis in the following way: We observe in \cref{Prop:Decomposing_Evolutionary_Graphs} that, up to isometry, the metric space $\eh$ can be constructed by iteratively gluing together tree-like metric spaces and almost linear metric spaces using a coproduct construction.  (The coproducts are taken in a category of based metric spaces, allowing the basepoint to change.)  Moreover, letting $P\vee Q$ denote a coproduct of two based metric spaces $P$ and $Q$, we have that $\Rips(P\vee Q)$ is, up to homotopy, a wedge sum of $\Rips(P)$ and $\Rips(Q)$ (\cref{Prop:Barcodes_Of_1_Point_Unions}).  Since the VRFs of tree-like metric spaces are topologically trivial, it follows that to describe the topology of $\Rips(\eh)$, it suffices to describe the topology of the VRF of an almost linear metric space; \cref{Thm:ALMS} gives such a description.  

\bparagraph{Outline}
For some of the material of this paper, we must assume that the reader is familiar with elementary algebraic topology.  However, much of our material on novelty profiles does not require a background in algebraic topology, and we believe that this material may be of independent interest.  Thus, we have arranged the paper so that the material on topology appears as late as possible.

\cref{Sec:Phylogenetic_graphs_and_populations} introduces our mathematical formalism for working with evolving populations in the presence of recombination.  \cref{Sec:NoveltyProfiles} introduces novelty profiles.  \cref{Sec:Galled_Trees} reviews galled trees and establishes that in the special case of galled trees, the temporal and topological novelty profiles are equal.   \cref{Sec:Topological_Preliminaries} reviews aspects of persistent homology and discrete Morse theory needed in the remainder of our paper, and observes that the topological novelty profile is stable.  \cref{Sec:Barcodes_From_Trees}  briefly reviews the results of Chan et al. on barcodes of evolutionary histories indexed by trees.  \cref{Sec:Met_Decomposition} studies the VRFs of coproducts of based metric spaces, and \cref{Sec:ALMS} presents our topological analysis of the VRFs of almost linear metric spaces.  Using the results of  \cref{Sec:Barcodes_From_Trees,Sec:Met_Decomposition,Sec:ALMS}, \cref{Sec:Inference} establishes our main deterministic result about the barcodes of evolutionary histories indexed by galled trees.  
\cref{Sec:Extensions} applies the stability of persistent homology to extend this result to subsamples of histories indexed by arbitrary phylogenetic graphs.

\cref{Sec:Independence} establishes that for a suitably chosen random history indexed by a fixed galled tree, the intervals in the $1^\mathrm{st}$ persistence barcode are independent random variables.  With this as motivation, \cref{Sec:Sensitivity_Numerical} uses simulation to study the statistical properties of the barcode of a random history with a single recombination event.  \cref{Sec:Discussion} discusses the applicability of our results and ideas to real-world genomic data, and explores directions for future work.

Two appendices tie our results explicitly to coalescent theory.  \cref{Sec:Appendix} studies the probability that an evolutionary history generated by the coalescent model is a galled tree.  \cref{Sec:Violations} observes in simulation that, although our main result for histories indexed by galled trees does not hold exactly for arbitrary subsamples, the lower bound on the number of recombination events implied by that result is only rarely violated under subsampling.

\subsection*{Acknowledgements}
We thank Ulrich Bauer for helpful discussions about how to prove \cref{Thm:ALMS}, our main result about the Vietoris--Rips filtrations of almost linear metric spaces.  Ulrich provided valuable input about the use of the triangle inequality in that argument, and also suggested the use of the discrete gradient vector field of \cite{kahle2011random}.  We also thank Pablo C\'amara and Kevin Emmett for valuable discussions, Matthew Zaremsky for sharing the counterexample of \cref{Rem:AlmostTreeLike}, and Greg Henselman for helpful feedback on our discussion of discrete Morse theory. Finally, we thank Peter Landweber and the anonymous reviewers for suggestions which helped improve the paper.  Lesnick was partially supported by funding from the Institute for Mathematics and its Applications, NIH grants U54CA193313 and T32MH065214, and an award from the J. Insley Blair Pyne Fund. Rabad\'an and Rosenbloom were funded by NIH grants U54CA193313 and R01GM117591.

\section{Phylogenetic Graphs and Evolutionary Histories}
\label{Sec:Phylogenetic_graphs_and_populations}
We now introduce our mathematical formalism for the topological study of reticulate evolution.  The formalism is similar to that used elsewhere in the literature on reticulate evolution, though some of our terminology is non-standard; for context, see for example \cite{huson2010phylogenetic} and the references therein.

\begin{definition}[Phylogenetic Graph]
A \emph{phylogenetic graph} is a finite directed acyclic graph $G$ such that
\begin{enumerate}
\item[1.] $G$ has a unique vertex $r$, the \emph{root}, with in-degree 0, 
\item[2.] Each vertex of $G$ has in-degree at most 2.
\end{enumerate}
We call a vertex in $G$ of in-degree 1 a \emph{clone}, and a vertex of in-degree 2 a \emph{recombinant}.  If $(v,w)$ is a directed edge in $G$, we say $v$ is a \emph{parent} of $w$.  
We define a \emph{rooted tree} to be a phylogenetic graph with no recombinants. 
\end{definition}
\cref{Fig:Phylo_Graph} illustrates a simple phylogenetic graph.
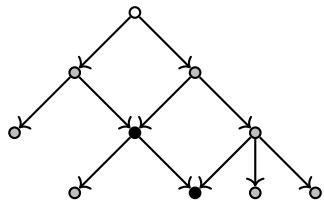
\begin{figure}[h]
\begin{center}
\begin{tikzpicture}[thick,scale=.8,->,
                   shorten >=2pt+0.5*\pgflinewidth,
                   shorten <=2pt+0.5*\pgflinewidth,
                   every node/.style={circle,
                                      draw,
                                      fill          = black!50,
                                      inner sep     = 0pt,
                                      minimum width =4 pt}]
                                      \tikzset{
    every node/.style={
        circle,
        draw,
        fill          = gray!50,
        inner sep     = 0pt,
        minimum width =4 pt
    }   
}  

       \coordinate (a) at (0,0);
       \coordinate (b) at (-1,-1);
       \coordinate (c) at (1,-1);
        \coordinate (d) at (-2,-2);
        \coordinate (e) at (0,-2);
        \coordinate (f) at (2,-2);
        \coordinate (g) at (-3,-3);
         \coordinate (h) at (-1,-3);
          \coordinate (i) at (1,-3);
          \coordinate (ih) at (2,-3);
           \coordinate (j) at (3,-3);
           
\path[draw] 

       
       node[fill=white] at (a) {} 
       node at  (b) {}
       node at  (c) {} 
       node at  (d) {} 
       node[fill=black] at  (e) {} 
       node at  (f) {}
         node at  (h) {}
          node[fill=black] at  (i) {}
          node at  (ih) {}
           node at  (j) {};
           

    \draw (a) -- (b) ;
    \draw (a) -- (c);
    \draw (b) -- (d);
    \draw (b) -- (e);
    \draw (c) -- (e);
    \draw (c) -- (f);
      \draw (f) -- (i);
      \draw (f) -- (j);
      \draw (f) -- (ih);
      \draw (e) -- (h);
      \draw (e) -- (i);
      

\end{tikzpicture}
\end{center}
\caption{A phylogenetic graph with 10 vertices: 7 clones (gray), 2 recombinants (black), and the root (white).}
\label{Fig:Phylo_Graph}
\end{figure}

For $G$ a rooted directed acyclic graph with vertex set $V$ and $S\subseteq V$, we say $v\in S$ is the \emph{minimum of $S$} if for all $s\in S$, any directed path from $r$ to $s$ in $G$ contains $v$.  $S$ may not have a minimum element, but if the minimum element exists, it is clearly unique.  

Let $\Set$ denote the collection of all finite sets.  

\begin{definition}[Evolutionary History]\label{Def:History}
For $G$ a phylogenetic graph with vertex set $V$, an \emph{(evolutionary) history indexed by $G$} is a map $\eh:V\to \Set$ with the following three properties:
\begin{enumerate}
\item[1.] If $w$ is a clone with parent $v$, then $\eh_v\subseteq \eh_w$.  
\item[2.] For each $m\in \cup_{v\in V} \eh_v$, the set $\{v\in V\mid m\in \eh_v\}$ has a minimum element.
\item[3.] If $w$ is a recombinant with parents $u$ and $v$, then \[\eh_u\cap \eh_v\subseteq \eh_w\subseteq \eh_u\cup \eh_v.\] 
\end{enumerate}
We call the elements of the sets $\eh_v$ \emph{mutations}.
\end{definition}

\cref{Fig:Population} gives an example of a history indexed by the phylogenetic graph of \cref{Fig:Phylo_Graph}. 

\begin{remark}The biological interpretation of the above definitions is this: A phylogenetic graph describes the ancestral relationships between organisms in a history, and each set $\eh_v$ specifies the genome of organism $v$, in terms of the difference between that genome and some fixed (unspecified) reference genome.  Properties 1 and 2 are standard in phylogenetics; they specify that each mutation arises only once in the history, and that each clone inherits all the mutations of its parent. Together, these two properties are often referred to as the \emph{infinite sites} assumption.  Property 3 stipulates that if both parents of a recombinant carry a mutation, then the recombinant inherits that mutation, and moreover, any mutation carried by a recombinant is inherited from a parent.  
\end{remark}

\begin{remark}
As indicated earlier, evolutionary histories are often called \emph{ancestral recombination graphs (ARGs)} in the literature \cite{gusfield2014recombinatorics, Rasmussen:2014cq}.  However, since we wish to make a clear distinction between the history and its underlying phylogenetic graph, we prefer the terminology presented here.
\end{remark}

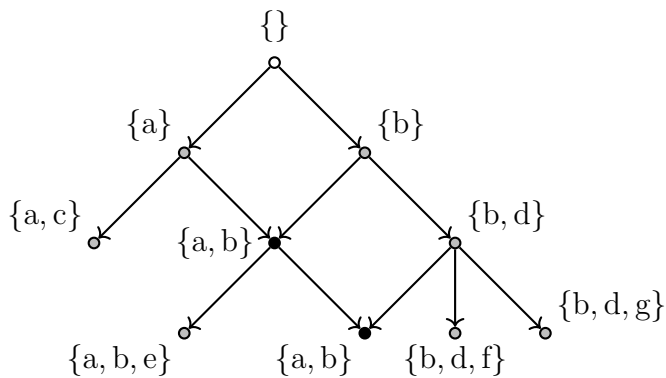
\begin{figure}[h]
\begin{center}
\begin{tikzpicture}[thick,scale=1.2,->,
                   shorten >=2pt+0.5*\pgflinewidth,
                   shorten <=2pt+0.5*\pgflinewidth] 

       \coordinate (a) at (0,0);
        \coordinate (aUp) at (0,.1);
       \coordinate (b) at (-1,-1);
       \coordinate (c) at (1,-1);
        \coordinate (d) at (-2,-2);
        \coordinate (e) at (0,-2);
         \coordinate (eLeft) at (-.1,-2);
        \coordinate (f) at (2,-2);
        \coordinate (g) at (-3,-3);
         \coordinate (h) at (-1,-3);
          \coordinate (i) at (1,-3);
               \coordinate (ih) at (2,-3);
           \coordinate (j) at (3,-3);
          \node[above] at (aUp) {$\{\}$};          
          \node[above left] at (b) {$\{\textup{a}\}$};          
               \node[above right] at (c) {$\{\textup{b}\}$};                  
          \node[above left] at (d) {$\{\au,\cu\}$};      
            \node[left] at (eLeft) {$\{\au,\bu\}$};    
              \node[above right] at (f) {$\{\bu,\du\}$};
                \node[below left] at (h) {$\{\au,\bu,\eu\}$};        
                 \node[below left] at (i) {$\{\au,\bu\}$};
                   \node[below] at (ih) {$\{\bu,\du,\fu\}$};
                      \node[above right] at (j) {$\{\bu,\du,\gu\}$};                                                                                 
                                                
                                                \tikzset{
    every node/.style={
        circle,
        draw,
        fill          = gray!50,
        inner sep     = 0pt,
        minimum width =4 pt
    }   
}
          
\path[draw] 

       
       node[fill=white] at (a) {} 
       node at  (b) {}
       node at  (c) {} 
       node at  (d) {} 
       node[fill=black] at  (e) {} 
       node at  (f) {}
         node at  (h) {}
          node at  (ih) {}
          node[fill=black] at  (i) {}     
           node at  (j) {};
           

    \draw (a) -- (b) ;
    \draw (a) -- (c);
    \draw (b) -- (d);
    \draw (b) -- (e);
    \draw (c) -- (e);
    \draw (c) -- (f);
          \draw (e) -- (h);
            \draw (e) -- (i);
      \draw (f) -- (i);
      \draw (f) -- (ih);
      \draw (f) -- (j);

\end{tikzpicture}
\end{center}
\caption{A history indexed by the phylogenetic graph of \cref{Fig:Phylo_Graph}.}
\label{Fig:Population}
\end{figure}

\begin{definition}[Symmetric Difference Metric on an Evolutionary History]\label{Def:Sym_Diff_Metric}
Define a metric $d$ on finite sets, the \emph{symmetric difference metric} by taking \[d(A,B)=|(A\cup B)\setminus (A\cap B)|\] for any finite sets $A$, $B$.  For any history $\eh$ indexed by a phylogenetic graph $G$ with vertex set $V$, this restricts to a metric on the set $\{\eh_v \mid v\in V\}$.  We denote the resulting metric space as $\met{\eh}$, or when no confusion is likely, simply as $\eh$.
\end{definition}

\begin{remark}
It is common in the phylogenetics literature to model genomes as binary vectors, and to metrize a set of genomes using the Hamming distance.  It is easy to see that the formalism we've introduced here is essentially  equivalent.  Under this equivalence, other common phylogenetic distances (e.g., Jukes-Cantor distance, Nei-Tamura distance) correspond to monotonic transformations of the symmetric difference metric $d$.  In fact, all the results of this paper formulated in terms of $d$ extend immediately to such monotonic transformations.
\end{remark}

\begin{remark}\label{Rem:Homoplasies}
In real-world evolving populations, the infinite sites assumption, described above, may not always hold. In other words, the same mutation may occur in different organisms despite being absent in their common ancestors. Such mutations, termed \emph{homoplasies}, may be observed in sampled data either if the per-site mutation rate is high (which is typical for species with short genomes, such as RNA viruses) or if the mutations confer high fitness. Homoplasies are typically rare for species with long genomes, as the probability of mutating twice at the same exact genetic site is small. If they do occur, homoplasies usually involve few sites, so that the metric space underlying the history differs only slightly from that of a history satisfying the infinite sites assumption.    
\end{remark}

\section{Novelty Profiles}\label{Sec:NoveltyProfiles}

\subsection{The Temporal Novelty Profile}
For $G$ a phylogenetic graph with vertex set $V$, define a partial order on $V$ by taking $v\leq w$ if there is a directed path in $G$ from $v$ to $w$.  We say $t:V\to \R$ is \emph{a time function} if $t(v)<t(w)$ whenever $v<w$.  We interpret  $t(v)$ as the birth time of organism $v$.

\begin{definition}[Temporal Novelty Profile]
Given a history $\eh$ indexed by $G$, a time function $t:V\to \R$, and a recombinant $r$ of $G$, we define $\mathcal N(r,t)$, the \emph{temporal novelty} of $r$, by 
\[\mathcal N(r,t):=\min\, \{d(\eh_v,\eh_r)\mid t(v)< t(r)\}.\]
We define $\mathcal N(\eh,t)$, the \emph{temporal novelty profile} of $\eh$ (with respect to $t$) to be the list of temporal novelties $\mathcal N(r,t)$, for all recombinants $r$ of $G$, sorted in decreasing order.
\end{definition}

\begin{example}\label{Ex:Temp_Novelty_Profile_Two_Time_Functions}
\cref{Fig:Novelty_Ex} illustrates novelty profiles for two histories indexed by the same simple phylogenetic graph.  For any time function on the history shown in \cref{Fig:Novelty_Ex1}, the unique recombinant has temporal novelty 1, so the temporal novelty profile of this history is the single-element list (1).  Similarly, for any time function on the history shown in \cref{Fig:Novelty_Ex2}, the temporal novelty profile is the single-element list (6).
\begin{figure}[h]
\begin{center}
\begin{subfigure}[b]{.40\linewidth}
\centering
\begin{center}
\begin{tikzpicture}[thick,scale=.8,->,
                   shorten >=2pt+0.5*\pgflinewidth,
                   shorten <=2pt+0.5*\pgflinewidth] 

       \coordinate (a) at (0,0);
        \coordinate (aUp) at (0,.1);
       \coordinate (b) at (-1,-1);
       \coordinate (c) at (1,-1);
        \coordinate (d) at (-2,-2);
        \coordinate (e) at (0,-2);
         \coordinate (eBelow) at (-.1,-2.1);
        \coordinate (f) at (2,-2);
        \coordinate (g) at (-3,-3);
         \coordinate (h) at (-1,-3);
          \coordinate (i) at (1,-3);
               \coordinate (ih) at (2,-3);
           \coordinate (j) at (3,-3);
          \node[above] at (aUp) {$\{\}$};          
          \node[above left] at (b) {$\{\au,\bu,\cu,\du,\eu,\fu\}$};          
               \node[above right] at (c) {$\{\Gu,\Hu,\Iu,\Ju,\Ku,\Lu\}$};                  
            \node[below] at (eBelow) {$\{\au,\bu,\cu,\du,\eu,\fu,\Gu\}$};    
                                                
                                                \tikzset{
    every node/.style={
        circle,
        draw,
        fill          = gray!50,
        inner sep     = 0pt,
        minimum width =4 pt
    }   
}
          
\path[draw] 

       
       node[fill=white] at (a) {} 
       node at  (b) {}
       node at  (c) {} 
       node[fill=black] at  (e) {}; 
           

    \draw (a) -- (b) ;
    \draw (a) -- (c);
    \draw (b) -- (e);
    \draw (c) -- (e);

\end{tikzpicture}
\end{center}
\caption{The temporal novelty profile of the history shown is the single-element list (1), for any time function. The topological novelty profile is the same.}
\label{Fig:Novelty_Ex1}
\end{subfigure}
\hskip30pt
\begin{subfigure}[b]{.40\linewidth}
\centering
\begin{center}
\begin{tikzpicture}[thick,scale=.8,->,
                   shorten >=2pt+0.5*\pgflinewidth,
                   shorten <=2pt+0.5*\pgflinewidth] 

       \coordinate (a) at (0,0);
        \coordinate (aUp) at (0,.1);
       \coordinate (b) at (-1,-1);
       \coordinate (c) at (1,-1);
        \coordinate (d) at (-2,-2);
        \coordinate (e) at (0,-2);
         \coordinate (eBelow) at (-.1,-2.1);
        \coordinate (f) at (2,-2);
        \coordinate (g) at (-3,-3);
         \coordinate (h) at (-1,-3);
          \coordinate (i) at (1,-3);
               \coordinate (ih) at (2,-3);
           \coordinate (j) at (3,-3);
          \node[above] at (aUp) {$\{\}$};          
          \node[above left] at (b) {$\{\au,\bu,\cu,\du,\eu,\fu\}$};          
               \node[above right] at (c) {$\{\Gu,\Hu,\Iu,\Ju,\Ku,\Lu\}$};                  
            \node[below] at (eBelow) {$\{\au,\bu,\cu,\Gu,\Hu,\Iu\}$};    
                                                
                                                \tikzset{
    every node/.style={
        circle,
        draw,
        fill          = gray!50,
        inner sep     = 0pt,
        minimum width =4 pt
    }   
}
          
\path[draw] 

       
       node[fill=white] at (a) {} 
       node at  (b) {}
       node at  (c) {} 
       node[fill=black] at  (e) {}; 
           

    \draw (a) -- (b) ;
    \draw (a) -- (c);
    \draw (b) -- (e);
    \draw (c) -- (e);

\end{tikzpicture}
\end{center}
\caption{The temporal novelty profile of the history shown is the single-element list (6), for any time function. The topological novelty profile is the same.}
\label{Fig:Novelty_Ex2}
\end{subfigure}
\caption{}
\label{Fig:Novelty_Ex}
\end{center}
\end{figure}
\end{example}

\begin{example}\label{Ex:Two_Recombs}
\cref{Fig:Novelty_Ex3} illustrates a history where two recombinants have the same parents.  For the time function shown, the novelty profile is (5,1).  The small second entry reflects the fact the two recombinant genomes are genetically close to one another.  Exchanging the time values of the bottom-most two vertices yields another time function for this history, for which the temporal novelty profile is (6,1).  If we take the time values of the bottom-most two vertices to both be $3$, then the temporal novelty profile is (6,5).
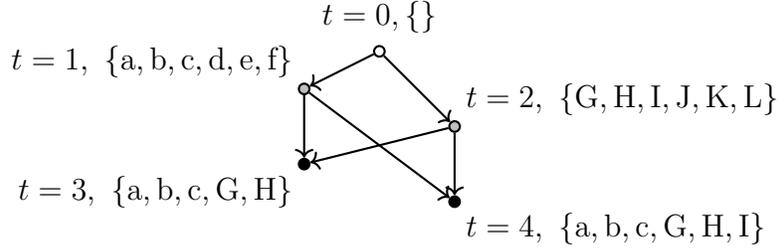
\begin{figure}[h]
\centering
\begin{center}
\begin{tikzpicture}[thick,scale=1.0,->,
                   shorten >=2pt+0.5*\pgflinewidth,
                   shorten <=2pt+0.5*\pgflinewidth] 

       \coordinate (a) at (0,0);
        \coordinate (aUp) at (0,.1);
       \coordinate (b) at (-1,-.5);
       \coordinate (c) at (1,-1);
        \coordinate (d) at (-2,-2);
        \coordinate (e) at (-1,-1.5);
         \coordinate (eLeft) at (-1.1,-1.5);
        \coordinate (f) at (1,-2);
          \coordinate (fRight) at (1.1,-2);
        \coordinate (g) at (-3,-3);
         \coordinate (h) at (-1,-3);
          \coordinate (i) at (1,-3);
               \coordinate (ih) at (2,-3);
           \coordinate (j) at (3,-3);
          \node[above] at (aUp) {$t=0, \{\}$};          
          \node[above left] at (b) {$t=1,\ \{\au,\bu,\cu,\du,\eu,\fu\}$};          
               \node[above right] at (c) {$t=2,\ \{\Gu,\Hu,\Iu,\Ju,\Ku,\Lu\}$};                  
            \node[below left] at (e) {$t=3,\ \{\au,\bu,\cu,\Gu,\Hu\}$};    
             \node[below right] at (f) {$t=4,\ \{\au,\bu,\cu,\Gu,\Hu,\Iu\}$};    
                                                
                                                \tikzset{
    every node/.style={
        circle,
        draw,
        fill          = gray!50,
        inner sep     = 0pt,
        minimum width =4 pt
    }   
}
          
\path[draw] 

       
       node[fill=white] at (a) {} 
       node at  (b) {}
       node at  (c) {} 
       node[fill=black] at  (e) {}
       node[fill=black] at  (f) {}; 
           

    \draw (a) -- (b) ;
    \draw (a) -- (c);
    \draw (b) -- (e);
    \draw (c) -- (e);
        \draw (b) -- (f);
    \draw (c) -- (f);

\end{tikzpicture}
\end{center}
\caption{The temporal novelty profile of the history and time function shown is (5,1); the topological novelty profile is the same.}
\label{Fig:Novelty_Ex3}
\end{figure}
\end{example}

\begin{remark}[Stability]\label{Remark:Temporal_Novelty_Stability}
Suppose we are given histories $\eh$ and $\eh'$ indexed by the same phylogenetic graph $G$ with $|d(\eh_v,\eh_w)-d(\eh'_v,\eh'_w)|\leq \epsilon$ for all vertices $v,w$ of $G$.  Note that this condition holds by the triangle inequality if $d(\eh_v,\eh_{v'})\leq \frac{\epsilon}{2}$ for all $v\in G$.  Let $t$ be any time function on the vertices of $G$.  We then have that \[d_\infty(\mathcal N(\eh,t),\mathcal N(\eh',t))\leq \epsilon,\] where for vectors $A$ and $B$ of the same length, $d_\infty(A,B):=\max_i |A_i-B_i|$. Thus, the temporal novelty profile is stable with respect to genomic perturbations.  

However, the temporal novelty profile is unstable with respect to perturbations of the time function.  For example, consider the history $\eh$ of \cref{Ex:Two_Recombs} (\cref{Fig:Novelty_Ex3}). For $\delta\in (-1,\infty)$, let $t_\delta$ be the time function obtained from the time function $t$ shown in \cref{Fig:Novelty_Ex3} by changing the time value of the vertex on the bottom right from 4 to $3 + \delta$. Then for all $\delta \in (0, 1)$, we have
\[\mathcal N(\eh,t_\delta)=(5,1),\qquad \mathcal N(\eh,t_{-\delta})=(6,1),\]
and therefore
\[d_\infty(\mathcal N(\eh,t_\delta),\mathcal N(\eh,t_{-\delta})) = 1,\]
whereas stability would require that this distance approach $0$ as $\delta$ approaches 0.
\end{remark}

\subsection{The Topological Novelty Profile}

\begin{definition}[Relative Minimum Spanning Tree] Given a weighted graph $G$ and a forest $F\subseteq G$ (i.e., a vertex-disjoint collection of subtrees), we define a \emph{spanning tree of $G$ rel $F$} simply to be a spanning tree $T$ of $G$ containing $F$.  We say $T$ is a \emph{minimum} spanning tree of $G$ rel $F$ if the sum of the edge weights of $T$ is as small as possible, among all spanning trees of $G$ rel $F$.  
\end{definition}

Note that by collapsing each tree in $F$ to a point, the problem of finding a minimum spanning tree rel $F$ is equivalent to the standard problem of finding an ordinary minimum spanning tree on a multigraph.  (A multigraph is a graph which is allowed to have multiple edges between pairs of vertices.)  Thus, all the basic facts about minimum spanning trees have analogues for relative minimum spanning trees.  For example, we have the following:

\begin{proposition}\label{Prop: Rel_Spanning_Tree_Indexwise_Min}
A spanning tree $T \textup{ rel } F$ is minimum if and only if for all $i$, the $i^{\mathrm{th}}$ smallest edge weight is less than or equal to the $i^{\mathrm{th}}$ smallest edge weight in any other spanning tree $\textup{rel }F$.  
\end{proposition}

\begin{proof}
By the remarks above, it suffices to establish the result for ordinary spanning trees, i.e., the case where $F$ is the empty forest.  Let $T$ be a minimum spanning tree, and let $U$ be any other spanning tree.  
To arrive at a contradiction, assume that for some $i$, the $i^{\mathrm{th}}$ smallest edge weight in $T$ is greater than the $i^{\mathrm{th}}$ smallest edge weight in $U$.  
Let $w$ denote the latter weight.  Consider the subforests $T_w\subset T$ and $U_w\subset U$ consisting of all vertices and just those edges of weight at most $w$.  $U_w$ contains more edges than $T_w$, so there exists a pair of vertices $(u,v)$ that lie in the same component of $U_w$ but not in the same component of $T_w$.  
In fact, there must exist some edge $e=(x,y)$ along the path from $u$ to $v$ in $U_w$ such that $x$ and $y$ lie in different components of $T_w$.  Clearly, the path from $x$ to $y$ in $T$ must contain at least one edge $e'$ with weight greater than $w$.  Replacing $e'$ with $e$ in $T$ gives a new spanning tree with strictly smaller weight than $T$, contradicting that $T$ is a minimum spanning tree.
\end{proof}

\begin{remark}\label{Rem:Uniqueness_Of_Rel_Spanning_Tree}
It follows from \cref{Prop: Rel_Spanning_Tree_Indexwise_Min} that the collection of edge weights in a relative minimum spanning tree is independent of the choice of the tree.  
\end{remark}

\begin{definition}[Topological Novelty Profile]\label{Def:Top_Novelty_Profile}
For $\eh$ a history indexed by a phylogenetic graph $G$, let $\forest$ be the forest in $G$ obtained by removing all edges pointing to recombinants.  Let $\bar G$ denote the complete graph with same vertex set as $G$.  Regard $\bar G$ as a weighted graph by taking the weight of edge $(u,v)$ to be $d(\eh_u,\eh_v)$. 

Let $T$ be a minimum spanning tree of $\bar G$ rel $\forest$.  We define $\mathcal T(\eh)$, the \emph{topological novelty profile} of $\eh$, to be the list of distances \[\{d(\eh_u,\eh_v)\mid (u,v)\in T\setminus \forest\},\] counted with multiplicity and sorted in descending order.  By \cref{Rem:Uniqueness_Of_Rel_Spanning_Tree}, $\mathcal T(\eh)$ does not depend on the choice of $T$.
\end{definition}

We will observe in \cref{Sec:Stabilty_Top_Nov_Profile} that the topological novelty profile has an interpretation in terms of persistent homology.

Given two lists of numbers $A$ and $B$, each sorted in decreasing order, we write $A\leq B$ if $|A|\leq |B|$ and for each $i\in \{1,\ldots, |A|\}$, $A_i\leq B_i$.

\begin{proposition}
For any history $\eh$ with time function $t$, \[\mathcal T(\eh)\leq \mathcal N(\eh,t).\]  That is, the topological novelty profile is a lower bound for the temporal novelty profile.
\end{proposition}

\begin{proof}
Suppose $\eh$ is indexed by $G$.  We construct a spanning tree $T$ of $\bar G$ rel $\forest$ such the weights of edges in $T\setminus \forest$ correspond to the temporal novelty profile.  The result then follows from \cref{Prop: Rel_Spanning_Tree_Indexwise_Min}.

To construct $T$, for each recombinant $r\in G$, choose a vertex $v(r)$ in $G$ with $t(v(r))<t(r)$, such that $d(\eh_{v(r)},\eh_r)$ is as small as possible among all such vertices.  We take $T$ to be the graph obtained from $F^G$ by adding in the edge 
$(v(r),r)$ for each recombinant $r$.  It is easy to check that $T$ is in fact a tree.
\end{proof}

\begin{example}
For the histories of \cref{Fig:Novelty_Ex1} and \cref{Fig:Novelty_Ex2}, the topological novelty profile is equal to the temporal one for all time functions.  For the history and time function of  \cref{Fig:Novelty_Ex3}, the topological and temporal novelty profiles are also equal, but one can select a different time function so that the two novelty profiles are not equal.
\end{example}

\begin{example}
\cref{Fig:Top_Vs_Temp_Novelty1} illustrates a history for which the temporal and topological novelty profiles are unequal for any choice of time function.  The topological novelty profile is (1,1), whereas the temporal novelty profile is always (2,1).  This example is degenerate, in the sense that the same genome (the empty one) appears at multiple vertices; \cref{Fig:Top_Vs_Temp_Novelty2} shows a variant of the example without this degeneracy.

\begin{figure}[h]
\begin{center}
\begin{subfigure}[position=b]{.40\linewidth}
\centering
\begin{center}
\begin{tikzpicture}[thick,scale=.75,->,
                   shorten >=2pt+0.5*\pgflinewidth,
                   shorten <=2pt+0.5*\pgflinewidth] 

       \coordinate (b) at (-1,-1);
          \coordinate (bA) at (-1,-.9);
       \coordinate (c) at (1,-1);
        \coordinate (d) at (-2,-2);
        \coordinate (e) at (0,-2);
         \coordinate (eRight) at (0,-2);
        \coordinate (g) at (1,-3);
         \coordinate (h) at (-1,-3);
          \coordinate (i) at (0,-4);
           \coordinate (iBelow) at (0,-4.1);
          \node[above] at (bA) {$\{\}$};          
          \node[above left] at (d) {$\{\au,\bu,\cu,\du\}$};      
            \node[above right] at (e) {$\{\}$};    
              \node[above right] at (g) {$\{\}$};
                \node[below left] at (h) {$\{\au,\bu\}$};        
                \node[below]  at (iBelow) {$\{\au\}$};
                                                
                                                \tikzset{
    every node/.style={
        circle,
        draw,
        fill          = gray!50,
        inner sep     = 0pt,
        minimum width =4 pt
    }   
}
          
\path[draw] 

       
       node[fill=white] at  (b) {}
       node at  (d) {} 
       node at  (e) {} 
        node at  (g) {}
         node[fill=black] at  (h) {}
         node[fill=black] at  (i) {};
           

    \draw (b) -- (d);
    \draw (b) -- (e);
   \draw (e) -- (g);
    \draw (g) -- (i);
    \draw (d) -- (h);
          \draw (e) -- (h);
            \draw (h) -- (i);          

\end{tikzpicture}
\end{center}
\caption{A history for which the topological and temporal novelty profiles are different for any time function.}
\label{Fig:Top_Vs_Temp_Novelty1}
\end{subfigure}
\hskip30pt
\begin{subfigure}[position=b]{.40\linewidth}
\centering
\begin{center}
\begin{tikzpicture}[thick,scale=.75,->,
                   shorten >=2pt+0.5*\pgflinewidth,
                   shorten <=2pt+0.5*\pgflinewidth] 

       \coordinate (b) at (-1,-1);
          \coordinate (bA) at (-1,-.9);
       \coordinate (c) at (1,-1);
        \coordinate (d) at (-2,-2);
        \coordinate (e) at (0,-2);
         \coordinate (eRight) at (0,-2);
        \coordinate (g) at (1,-3);
         \coordinate (h) at (-1,-3);
          \coordinate (i) at (0,-4);
           \coordinate (iBelow) at (0,-4.1);
          \node[above] at (bA) {$\{\}$};          
          \node[above left] at (d) {$\{\Au,\Bu,\Cu,\Du\}$};      
            \node[above right] at (e) {$\{\eu\}$};    
              \node[above right] at (g) {$\{\eu,\fu\}$};
                \node[below left] at (h) {$\{\Au,\Bu\}$};        
                \node[below]  at (iBelow) {$\{\Au\}$};
                                                
                                                \tikzset{
    every node/.style={
        circle,
        draw,
        fill          = gray!50,
        inner sep     = 0pt,
        minimum width =4 pt
    }   
}
          
\path[draw] 

       
       node[fill=white] at  (b) {}
       node at  (d) {} 
       node at  (e) {} 
        node at  (g) {}
         node[fill=black] at  (h) {}
         node[fill=black] at  (i) {};
           

    \draw (b) -- (d);
    \draw (b) -- (e);
   \draw (e) -- (g);
    \draw (g) -- (i);
    \draw (d) -- (h);
          \draw (e) -- (h);
            \draw (h) -- (i);          

\end{tikzpicture}
\end{center}
\caption{A variant of example (a) with no duplicate genomes.  Each boldface letter represents three mutations, while each letter in plain font represents a single mutation. }
\label{Fig:Top_Vs_Temp_Novelty2}
\end{subfigure}
\caption{}
\label{Fig:Top_Vs_Temp_Novelty}
\end{center}
\end{figure}
\end{example}

Like temporal novelty profiles, topological novelty profiles are stable with respect to perturbations of the genomic data; we show this in \cref{Prop:Topological_Novelty_Stability}.  \cref{Equality_Of_Nov_Profiles_For _Galled_Trees} below tells us that when $G$ is a galled tree, the temporal and topological novelty profiles are in fact equal. 

\section{Histories Indexed By Galled Trees}\label{Sec:Galled_Trees}
Our main bounds on novelty profiles concern the special case that our phylogenetic graphs are galled trees.

The definition of galled tree we give is equivalent to the one given in \cite[Definition 6.11.1]{huson2010phylogenetic}.  As noted in \cite{huson2010phylogenetic}, this is slightly more general than the original definition \cite{gusfield2003efficient,wang2001perfect}, which requires the cycles in a galled tree to be node-disjoint.

\begin{definition}[Source-Sink Loop]
We say an undirected graph is a \emph{loop} if its geometric realization is homeomorphic to a circle.  We call a directed graph $G$ a \emph{source-sink loop} if 
\begin{enumerate}
\item[1.] The undirected graph underlying $G$ is a loop.
\item[2.] $G$ has a unique source and unique sink.
\end{enumerate}
\end{definition}

\begin{definition}[Sum of Directed Graphs]
For directed graphs $G$ and $H$, with $v$ a source in $G$ and $w$ any vertex in $H$, we define a directed graph $G\vee_{v,w} H$ by taking the disjoint union of $G$ and $H$ and then identifying $v$ and $w$ (i.e.,``gluing" $v$ to $w$).  We call $G\vee_{v,w} H$ a \emph{sum} of $G$ and $H$.  (We do \emph{not} define the sum $G\vee_{v,w} H$ in the case that neither of the vertices $v$ or $w$ is a source.)
We will sometimes write $G\vee_{v,w} H$ simply as $G\vee H$, suppressing $v$ and $w$.
\end{definition}  

\begin{definition}[Galled Tree]
Let $\A$ be the smallest collection of directed acyclic graphs such that:
\begin{enumerate}
\item[1.] Each rooted tree is in $\A$.
\item[2.] Each source-sink loop is in $\A$.
\item[3.] if $G$ and $H$ are in $\A$, then so is each sum $G\vee H$.
\end{enumerate}
We define a \emph{galled tree} to be a graph isomorphic to one in $\A$.
Thus, informally, a galled tree is a graph obtained by iteratively gluing rooted trees and source-sink loops along single vertices, using the sum operation specified above. 
\end{definition}

We omit the easy proof of the following:
\begin{proposition}
Any galled tree is a phylogenetic graph.
\end{proposition}

Note that the recombinants in a galled tree $G$ are in bijective correspondence with the source-sink loops in $G$.

\cref{Fig:Galled_Tree} gives an example of a galled tree.  It can be checked that the phylogenetic graph of \cref{Fig:Phylo_Graph} is not a galled tree.

\begin{figure}[h]
\begin{center}
\begin{tikzpicture}[thick,scale=.8,->,
                   shorten >=2pt+0.5*\pgflinewidth,
                   shorten <=2pt+0.5*\pgflinewidth,
                   every node/.style={circle,
                                      draw,
                                      fill          = black!50,
                                      inner sep     = 0pt,
                                      minimum width =4 pt}]
                                      \tikzset{
    every node/.style={
        circle,
        draw,
        fill          = gray!50,
        inner sep     = 0pt,
        minimum width =4 pt
    }   
}  

       \coordinate (a) at (0,0);
       \coordinate (b) at (-1,-1);
       \coordinate (c) at (1,-1);
        \coordinate (d) at (-2,-2);
         \coordinate (de) at (-1,-2);
        \coordinate (e) at (0,-2);
         \coordinate (ef) at (1,-2);
        \coordinate (f) at (2,-2);
        \coordinate (g) at (-3,-3);
         \coordinate (h) at (-1,-3);
          \coordinate (i) at (1,-3);
           \coordinate (j) at (3,-3);
            \coordinate (l) at (-2,-4);
             \coordinate (m) at (0,-4);
              \coordinate (n) at (2,-4);
               \coordinate (o) at (4,-4);
           \coordinate (q) at (-3,-5);
            \coordinate (r) at (-1,-5);
             \coordinate (s) at (1,-5);
              \coordinate (t) at (3,-5);
          
\path[draw] 

       
       node at (a) {} 
       node at  (b) {}
       node at  (c) {} 
       node at  (d) {}
       node at  (de) {}  
       node at  (e) {}
        node at  (ef) {}  
       node at  (f) {}
        node at  (g) {}
         node at  (h) {}
          node at  (i) {}
           node at  (j) {}
           node[fill=black] at  (l) {}
           node at  (m) {}
           node[fill=black] at  (n) {}
           node at  (o) {}
           node at  (q) {}
           node at  (r) {}
           node at  (s) {};
           

    \draw[color=black] (a) -- (b) ;
    \draw[color=black] (a) -- (c);
    \draw[color=black] (b) -- (d);
    \draw[color=black] (b) -- (de);
    \draw[color=black] (c) -- (e);
    \draw[dashed] (c) -- (ef);
    \draw[dashed]  (c) -- (f);
    \draw[dashed]  (d) -- (g);
    \draw[dashed]  (d) -- (h);
      \draw[dashed]  (f) -- (j);
       \draw[dashed]  (ef) -- (i);
      \draw[dashed]  (g) -- (l);
      \draw[dashed]  (h) -- (l);
      \draw[color=black] (i) -- (m);
      \draw[dashed]  (i) -- (n);
      \draw[dashed]  (j) -- (n);
       \draw[color=black] (j) -- (o);
         \draw[color=black] (l) -- (q);
          \draw[color=black] (l) -- (r);
           \draw[color=black] (m) -- (s);

\end{tikzpicture}
\end{center}
\caption{A galled tree which can be constructed as the iterated sum of four rooted trees (solid edges) and two source-sink loops (dashed edges).  The two recombinants are shown in black.}
\label{Fig:Galled_Tree}
\end{figure}
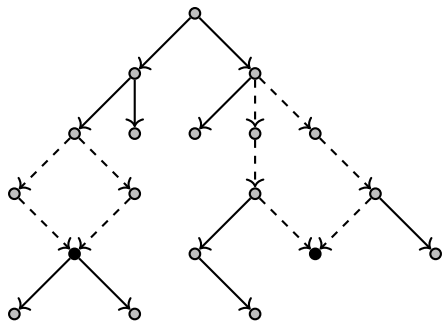

\begin{proposition}[Equality of Temporal and Topological Novelty Profiles on Galled Trees]\label{Equality_Of_Nov_Profiles_For _Galled_Trees}
For any history $\eh$ indexed by a galled tree and time function $t$, \[\mathcal N(\eh,t)=\mathcal T(\eh).\]
\end{proposition}

\begin{proof}
Suppose $\eh$ is indexed by the galled tree $G$.  We use the notation from \cref{Def:Top_Novelty_Profile}.  As in the case of ordinary minimum spanning trees, a minimum spanning tree of $\bar G$ rel $F$ can be constructed greedily, by considering the edges of $\bar G\setminus F$ in order of increasing weight.  In this construction, each edge of $\bar G\setminus F$ added to the relative minimum spanning tree can  be chosen to connect a recombinant $r$ to a vertex $v$ of the source-sink loop in $G$ that has $r$ as its sink.  We then have that $t(v) < t(r)$ and $d(\eh_v,\eh_r)\leq d(\eh_w,\eh_r)$ for any other vertex $w$ with $t(w)<t(r)$.  Clearly, in this construction we never take $r$ to be the same recombinant more than once.  The result follows.
\end{proof}

\begin{remark}[Galled Trees as Models for Evolution in the Low-Recombination Limit]\label{Rem:Galled_Low_Recomb}
Given a probabilistic model generating a phylogenetic graph, one may ask what the probability is of obtaining a galled tree.  This problem has previously been studied by simulation in \cite{Arenas:2008hb}, for a coalescent model of evolution.  In \cref{Sec:Appendix}, we study the same problem analytically.  We show that the problem reduces to the study of a finite-state Markov chain.  A simple analysis of this Markov chain yields, for fixed population size $n$, a system of linear equations $L(\rho)$ depending on a recombination rate parameter $\rho$, whose solution gives the probability $P(n,\rho)$ of obtaining a galled tree.  Solving these linear systems numerically for various values of $\rho$ and $n$, we observe that as $\rho$ tends to 0, $P(n,\rho)$ tends to 1.  

This indicates that histories indexed by galled trees are biologically reasonable models of evolution in low-recombination settings.  While, from a biological standpoint, the specific bounds on $\rho$ needed to obtain a galled tree with high probability are rather stringent in general, we do expect these bounds to hold in some settings of interest; see \cref{Sec:Discussion} for further discussion of this.  

Regardless, from a mathematical perspective, the special case of galled trees seems to be a natural place to begin fleshing out theoretical foundations for the topological study of evolution.
\end{remark}

\section{Topological Preliminaries}\label{Sec:Topological_Preliminaries}
In this section, we briefly review persistent homology and the related topological definitions and results we will need in the remainder of the paper.  As a first application, we observe that the topological novelty profile admits a description in terms of persistent homology, and is therefore stable.  
We also briefly review some ideas from discrete Morse theory.  

We assume that the reader is familiar with some standard concepts from elementary algebraic topology, including simplicial complexes, homology, and homotopy equivalence.  Good introductions can be found in many places, e.g., \cite{hatcher2002algebraic,munkres1984elements}. 

\subsection{Persistent Homology}\label{Sec:Persistent_Homology}
Our treatment of persistent homology will be terse; for a more thorough introduction to these ideas, including a discussion of some of the many applications of persistent homology to data analysis, see the surveys and textbooks \cite{carlsson2014topological,carlsson2009topology,edelsbrunner2010computational,oudot2015persistence}. 

\bparagraph{Filtrations} A \emph{filtration} is a collection of topological spaces $\{\F_r\}_{r\in [0,\infty)}$ such that $\F_r\subseteq \F_s$ whenever $r\leq s$.  A morphism $f:\F\to \G$ of filtrations is a collection of continuous maps $\{f_r:\F_r\to \G_r\}_{r\in [0,\infty)}$ such that the following diagram commutes for all $r\leq s$: 
\[
\begin{tikzcd}
\F_r\ar[hookrightarrow]{r}\ar["f_r",swap]{d} &\F_s\ar["f_s"]{d} \\
\G_r\ar[hookrightarrow]{r} &\G_s
\end{tikzcd}
\]
We say $f$ is an \emph{objectwise homotopy equivalence} if each $f_r$ is a homotopy equivalence.  Intuitively, if two filtrations are connected by an objectwise homotopy equivalence, we should think of them as topologically equivalent; for further discussion of this point in the context of topological data analysis, see \cite{blumberg2017universality}.

\bparagraph{Vietoris--Rips Filtrations}
For $S$ a simplicial complex, we use square brackets to denote simplices of $S$.  Thus, for example, the set of simplices of a triangle with vertex set $\{a,b,c\}$ is $\{[a],[b],[c],[ab],[bc],[ac],[abc]\}$.

For $P$ a finite metric space and $r\in [0,\infty)$, the \emph{Vietoris--Rips complex of $P$} with scale parameter $r$, denoted $\Rips(P)_r$, is the simplicial complex with vertices $P$ that contains simplex $[p_1,p_2,\ldots,p_n]$ if and only if $\diam\{p_1,p_2,\ldots,p_n\}\leq 2r$.   
If $r\leq s$, then $\Rips(P)_r\subseteq \Rips(P)_s$, so $\Rips(P):=\{\Rips(P)_r\}_{r\in [0,\infty)}$ is a filtration; see \cref{Fig:RipsEx}.
\definecolor{aqaqaq}{rgb}{0.6274509803921569,0.6274509803921569,0.6274509803921569}
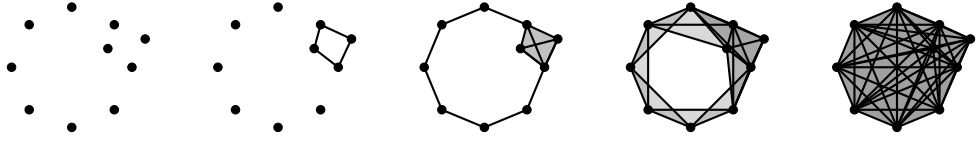
\begin{figure}
\begin{center}
\begin{tikzpicture}[scale=.8]
\coordinate (a) at (1,0);
\coordinate (b) at (0,1);
\coordinate (c) at (-1,0);
\coordinate (d) at (0,-1);
\coordinate (e) at (.707,.707);
\coordinate (f) at (-.707,.707);
\coordinate (g) at (-.707,-.707);
\coordinate (h) at (.707,-.707);
\coordinate (i) at  (1.22,.47);
\coordinate (j) at (.6,.31);
\coordinate (k) at (.6,.9);

\draw [fill] (a) circle [radius=0.07];
\draw [fill] (b) circle [radius=0.07];
\draw [fill] (c) circle [radius=0.07];
\draw [fill] (d) circle [radius=0.07];
\draw [fill] (e) circle [radius=0.07];
\draw [fill] (f) circle [radius=0.07];
\draw [fill] (g) circle [radius=0.07];
\draw [fill] (h) circle [radius=0.07];
\draw [fill] (i) circle [radius=0.07];
\draw [fill] (j) circle [radius=0.07];
\end{tikzpicture}
\hskip20pt
\begin{tikzpicture}[scale=.8]
\coordinate (a) at (1,0);
\coordinate (b) at (0,1);
\coordinate (c) at (-1,0);
\coordinate (d) at (0,-1);
\coordinate (e) at (.707,.707);
\coordinate (f) at (-.707,.707);
\coordinate (g) at (-.707,-.707);
\coordinate (h) at (.707,-.707);
\coordinate (i) at  (1.22,.47);
\coordinate (j) at (.6,.31);
\coordinate (k) at (.6,.9);

\draw[thick] (a) -- (i);
\draw[thick] (a) -- (j);
\draw[thick] (e) -- (i);
\draw[thick] (e) -- (j);

\draw [fill] (a) circle [radius=0.07];
\draw [fill] (b) circle [radius=0.07];
\draw [fill] (c) circle [radius=0.07];
\draw [fill] (d) circle [radius=0.07];
\draw [fill] (e) circle [radius=0.07];
\draw [fill] (f) circle [radius=0.07];
\draw [fill] (g) circle [radius=0.07];
\draw [fill] (h) circle [radius=0.07];
\draw [fill] (i) circle [radius=0.07];
\draw [fill] (j) circle [radius=0.07];
\end{tikzpicture}
\hskip20pt
\begin{tikzpicture}[scale=.8]
\coordinate (a) at (1,0);
\coordinate (b) at (0,1);
\coordinate (c) at (-1,0);
\coordinate (d) at (0,-1);
\coordinate (e) at (.707,.707);
\coordinate (f) at (-.707,.707);
\coordinate (g) at (-.707,-.707);
\coordinate (h) at (.707,-.707);
\coordinate (i) at  (1.22,.47);
\coordinate (j) at (.6,.31);
\coordinate (k) at (.6,.9);

\fill[color=aqaqaq, fill=aqaqaq, fill opacity=0.4] (a) -- (e) -- (i) -- cycle;
\fill[color=aqaqaq, fill=aqaqaq, fill opacity=0.4] (a) -- (e) -- (j) -- cycle;
\fill[color=aqaqaq, fill=aqaqaq, fill opacity=0.4] (a) -- (i) -- (j) -- cycle;
\fill[color=aqaqaq, fill=aqaqaq, fill opacity=0.4] (e) -- (i) -- (j) -- cycle;

\draw[thick] (a) -- (i);
\draw[thick] (a) -- (j);
\draw[thick] (e) -- (i);
\draw[thick] (e) -- (j);
\draw[thick] (a) -- (e);
\draw[thick] (i) -- (j);
\draw[thick] (b) -- (e);
\draw[thick] (b) -- (f);
\draw[thick] (c) -- (f);
\draw[thick] (c) -- (g);
\draw[thick] (d) -- (g);
\draw[thick] (d) -- (h);
\draw[thick] (a) -- (h);
\draw[thick] (a) -- (i);
\draw[thick] (a) -- (j);
\draw [fill] (a) circle [radius=0.07];
\draw [fill] (b) circle [radius=0.07];
\draw [fill] (c) circle [radius=0.07];
\draw [fill] (d) circle [radius=0.07];
\draw [fill] (e) circle [radius=0.07];
\draw [fill] (f) circle [radius=0.07];
\draw [fill] (g) circle [radius=0.07];
\draw [fill] (h) circle [radius=0.07];
\draw [fill] (i) circle [radius=0.07];
\draw [fill] (j) circle [radius=0.07];
\end{tikzpicture}
\hskip20pt
\begin{tikzpicture}[scale=.8]
%
\coordinate (a) at (1,0);
\coordinate (b) at (0,1);
\coordinate (c) at (-1,0);
\coordinate (d) at (0,-1);
\coordinate (e) at (.707,.707);
\coordinate (f) at (-.707,.707);
\coordinate (g) at (-.707,-.707);
\coordinate (h) at (.707,-.707);
\coordinate (i) at  (1.22,.47);
\coordinate (j) at (.6,.31);
\coordinate (k) at (.6,.9);

\fill[color=aqaqaq, fill=aqaqaq, fill opacity=0.4] (a) -- (e) -- (i) -- cycle;
\fill[color=aqaqaq, fill=aqaqaq, fill opacity=0.4] (a) -- (e) -- (j) -- cycle;
\fill[color=aqaqaq, fill=aqaqaq, fill opacity=0.4] (a) -- (i) -- (j) -- cycle;
\fill[color=aqaqaq, fill=aqaqaq, fill opacity=0.4] (e) -- (i) -- (j) -- cycle;

\fill[color=aqaqaq, fill=aqaqaq, fill opacity=0.4] (a) -- (b) -- (e) -- cycle;
\fill[color=aqaqaq, fill=aqaqaq, fill opacity=0.4] (a) -- (b) -- (i) -- cycle;
\fill[color=aqaqaq, fill=aqaqaq, fill opacity=0.4] (a) -- (b) -- (j) -- cycle;
\fill[color=aqaqaq, fill=aqaqaq, fill opacity=0.4] (a) -- (e) -- (h) -- cycle;

\fill[color=aqaqaq, fill=aqaqaq, fill opacity=0.4] (e) -- (b) -- (f) -- cycle;
\fill[color=aqaqaq, fill=aqaqaq, fill opacity=0.4] (b) -- (f) -- (c) -- cycle;
\fill[color=aqaqaq, fill=aqaqaq, fill opacity=0.4] (f) -- (c) -- (g) -- cycle;
\fill[color=aqaqaq, fill=aqaqaq, fill opacity=0.4] (c) -- (g) -- (d) -- cycle;
\fill[color=aqaqaq, fill=aqaqaq, fill opacity=0.4] (g) -- (d) -- (h) -- cycle;
\fill[color=aqaqaq, fill=aqaqaq, fill opacity=0.4] (d) -- (h) -- (a) -- cycle;

\fill[color=aqaqaq, fill=aqaqaq, fill opacity=0.4] (b) -- (e) -- (i) -- cycle;
\fill[color=aqaqaq, fill=aqaqaq, fill opacity=0.4] (b) -- (e) -- (j) -- cycle;
\fill[color=aqaqaq, fill=aqaqaq, fill opacity=0.4] (b) -- (j) -- (i) -- cycle;

\fill[color=aqaqaq, fill=aqaqaq, fill opacity=0.4] (h) -- (e) -- (i) -- cycle;
\fill[color=aqaqaq, fill=aqaqaq, fill opacity=0.4] (h) -- (e) -- (j) -- cycle;
\fill[color=aqaqaq, fill=aqaqaq, fill opacity=0.4] (h) -- (i) -- (j) -- cycle;

\fill[color=aqaqaq, fill=aqaqaq, fill opacity=0.4] (e)-- (f) -- (j) -- cycle;
\draw[thick] (a) -- (i);
\draw[thick] (a) -- (j);
\draw[thick] (e) -- (i);
\draw[thick] (e) -- (j);
\draw[thick] (a) -- (e);
\draw[thick] (i) -- (j);
\draw[thick] (b) -- (e);
\draw[thick] (b) -- (f);
\draw[thick] (c) -- (f);
\draw[thick] (c) -- (g);
\draw[thick] (d) -- (g);
\draw[thick] (d) -- (h);
\draw[thick] (a) -- (h);
\draw[thick] (a) -- (i);
\draw[thick] (a) -- (j);
\draw[thick] (a) -- (b);
\draw[thick] (b) -- (c);
\draw[thick] (c) -- (d);
\draw[thick] (d) -- (a);
\draw[thick] (e) -- (f);
\draw[thick] (f) -- (g);
\draw[thick] (g) -- (h);
\draw[thick] (h) -- (a);
\draw[thick] (e) -- (h);
\draw[thick] (h) -- (i);
\draw[thick] (h) -- (j);
\draw[thick] (i) -- (b);
\draw[thick] (j) -- (b);
\draw[thick] (f) -- (j);

\draw [fill] (a) circle [radius=0.07];
\draw [fill] (b) circle [radius=0.07];
\draw [fill] (c) circle [radius=0.07];
\draw [fill] (d) circle [radius=0.07];
\draw [fill] (e) circle [radius=0.07];
\draw [fill] (f) circle [radius=0.07];
\draw [fill] (g) circle [radius=0.07];
\draw [fill] (h) circle [radius=0.07];
\draw [fill] (i) circle [radius=0.07];
\draw [fill] (j) circle [radius=0.07];
\end{tikzpicture}
\hskip20pt
\begin{tikzpicture}[scale=.8]
%
\coordinate (a) at (1,0);
\coordinate (b) at (0,1);
\coordinate (c) at (-1,0);
\coordinate (d) at (0,-1);
\coordinate (e) at (.707,.707);
\coordinate (f) at (-.707,.707);
\coordinate (g) at (-.707,-.707);
\coordinate (h) at (.707,-.707);
\coordinate (i) at  (1.22,.47);
\coordinate (j) at (.6,.31);
\coordinate (k) at (.6,.9);

\fill[color=aqaqaq, fill=aqaqaq, fill opacity=1] (i) -- (e) -- (b) -- (f) -- (c) -- (g) -- (d) -- (h) -- cycle;
\draw[thick] (a) -- (b);
\draw[thick] (a) -- (c);
\draw[thick] (a) -- (d);
\draw[thick] (a) -- (e);
\draw[thick] (a) -- (f);
\draw[thick] (a) -- (g);
\draw[thick] (a) -- (h);
\draw[thick] (a) -- (i);
\draw[thick] (a) -- (j);

\draw[thick] (b) -- (c);
\draw[thick] (b) -- (d);
\draw[thick] (b) -- (e);
\draw[thick] (b) -- (f);
\draw[thick] (b) -- (g);
\draw[thick] (b) -- (h);
\draw[thick] (b) -- (i);
\draw[thick] (b) -- (j);

\draw[thick] (c) -- (d);
\draw[thick] (c) -- (e);
\draw[thick] (c) -- (f);
\draw[thick] (c) -- (g);
\draw[thick] (c) -- (h);
\draw[thick] (c) -- (i);
\draw[thick] (c) -- (j);

\draw[thick] (d) -- (e);
\draw[thick] (d) -- (f);
\draw[thick] (d) -- (g);
\draw[thick] (d) -- (h);
\draw[thick] (d) -- (i);
\draw[thick] (d) -- (j);

\draw[thick] (e) -- (f);
\draw[thick] (e) -- (g);
\draw[thick] (e) -- (h);
\draw[thick] (e) -- (i);
\draw[thick] (e) -- (j);

\draw[thick] (f) -- (g);
\draw[thick] (f) -- (h);
\draw[thick] (f) -- (i);
\draw[thick] (f) -- (j);

\draw[thick] (g) -- (h);
\draw[thick] (g) -- (i);
\draw[thick] (g) -- (j);

\draw[thick] (h) -- (i);
\draw[thick] (h) -- (j);

\draw[thick] (i) -- (j);

\draw [fill] (a) circle [radius=0.07];
\draw [fill] (b) circle [radius=0.07];
\draw [fill] (c) circle [radius=0.07];
\draw [fill] (d) circle [radius=0.07];
\draw [fill] (e) circle [radius=0.07];
\draw [fill] (f) circle [radius=0.07];
\draw [fill] (g) circle [radius=0.07];
\draw [fill] (h) circle [radius=0.07];
\draw [fill] (i) circle [radius=0.07];
\draw [fill] (j) circle [radius=0.07];
\end{tikzpicture}
\end{center}
\caption{Rips complexes $\Rips(P)_r$ on a simple point cloud $P\subset\R^2$, for several choices of $r$.}\label{Fig:RipsEx}
\end{figure}

\bparagraph{Persistence Modules} A \emph{persistence module} $M$ consists of a collection of vector spaces $\{M_r\}_{r\in [0,\infty)}$, together with a collection of linear maps $\{M_{r,s}:M_r\to M_s\}_{r\leq s}$ such that 
\begin{enumerate}
\item[1.]  for all $r\leq s\leq t$ the following diagram commutes:
\[
\begin{tikzcd}[column sep=7.5ex]
  M_r \arrow{dr}{M_{r,t}}\arrow{d}[swap]{M_{r,s}}
 \\
  M_s \arrow{r}[swap]{M_{s,t}}
& M_t.
\end{tikzcd}
\]
\item[2.] $M_{r,r}=\id_{M_r}$ for all $r$.
\end{enumerate}
We say $M$ is \emph{pointwise finite dimensional (\pfd)} if $\dim M_r<\infty$ for all $r$.  

Similar to the definition for filtrations, a morphism $f:M\to N$ of persistence modules is a collection of linear maps $\{f_r:M_r\to N_r\}_{r\in [0,\infty)}$ such that for all $r\leq s$, the following diagram commutes:
\[
\begin{tikzcd}[column sep=7.5ex]
  M_r \arrow{r}{M_{r,s}}\arrow{d}[swap]{f_r}
& M_s \arrow{d}{f_s} \\
  N_r \arrow{r}[swap]{N_{r,s}}
& N_s.
\end{tikzcd}
\]
We say $f$ is an isomorphism if each of the maps $f_r$ is an isomorphism.

\bparagraph{Direct Sums of Persistence Modules} We assume that the reader is familiar with the definition of the direct sum of vector spaces from linear algebra.  For linear maps $f:V_1\to W_1$ and $g:V_2\to W_2$, we define the direct sum \[f\oplus g:V_1\oplus V_2\to W_1\oplus W_2\] by taking $f\oplus g(v,w)=(f(v),g(w))$.  We then define the sum $M\oplus N$ to be the persistent module given by
\[(M\oplus N)_r=M_r\oplus N_r,\qquad (M\oplus N)_{r,s}=M_{r,s}\oplus N_{r,s}.\]  We can define the direct sum of an arbitrary collection of persistence modules in the same way.

\bparagraph{Reduced Homology} Fix a field $\field$.  (For example, we can take $\field=\Q$, or $\field=\Z_2$, the field with two elements.)  For $i\geq 0$, let $\tilde H_i$ denote the $i^{\mathrm{th}}$ reduced singular homology functor with coefficients in $\field$.  Thus, $\tilde H_i$ maps each topological space $S$ to a $\field$-vector space $\tilde H_i(S)$, and maps each continuous function $f:S\to T$ to a linear map $f_*:\tilde H_i(S)\to \tilde H_i(T)$.  
Applying $\tilde H_i$ to each space and each inclusion map in a filtration $\F$ gives us a persistence module $\tilde H_i(\F)$.  Moreover, a morphism of filtrations $f:\F\to \G$ induces an morphism $f_*: \tilde H_i(\F)\to \tilde H_i(\G)$. 

\begin{lemma}\label{lem:ObjectwiseHE}
If a morphism of filtrations $f:\F\to \G$ is an objectwise homotopy equivalence, then for any $i\geq 0$, $f_*:\tilde H_i(\F)\to \tilde H_i(\G)$ is an isomorphism.
\end{lemma}

\begin{proof}
It is a standard fact that if a continuous map $g$ is a homotopy equivalence, then $\tilde H_i(g)$ is an isomorphism.  This gives the result.
\end{proof}

\bparagraph{Barcodes}
We say $\J\subseteq \R$ is an \emph{interval} if $\J$ is nonempty and connected.  For $\J$ an interval, define the \emph{interval module} $I^\J$ to be the persistence module such that 
\begin{align*}
I^\J_r&=
\begin{cases}
\field &{\textup{if }} r\in \J, \\
0 &{\textup{ otherwise}.}
\end{cases}
& I^\J_{r,s}=
\begin{cases}
\id_\field &{\textup{if }} r\leq s\in \J,\\
0 &{\textup{ otherwise}.}
\end{cases}
\end{align*}

\begin{theorem}[Structure of Persistence Modules \cite{crawley2012decomposition}]\label{Thm:ordinary_stability}
If $M$ is a \pfd persistence module, then there exists a unique collection of intervals $\B_M$ such that 
\[M\cong \oplus_{\J\in \B_M} I^\J.\]  
\end{theorem}
We call $\B_M$ the \emph{barcode of $M$}. For $\F$ a filtration, we write $\B_{\tilde H_i(\F)}$ simply as $\B_i(\F)$.  Similarly, for $P$ a finite metric space, we write $\B_i(\Rips(P))$ simply as $\B_i(P)$.

\begin{remark}
As it is defined using reduced homology, $\B_0(\F)$ differs from the $0^{\mathrm{th}}$ barcode constructed using unreduced homology by the removal of an infinite length interval.
\end{remark}

\begin{definition}[Gromov-Hausdorff Distance] Given two subspaces $P,Q$ of a metric space $Z$, we define the \emph{Hausdorff distance} between $P$ and $Q$, by
\[d_H(P,Q):=\max\{\,\sup_{p \in P} \inf_{q \in Q} d(p,q),\, \sup_{q \in Q} \inf_{p \in P} d(p,q)\}.\]
For $P$ and $Q$ any compact metric spaces, define $d_{GH}(P,Q)$, the \emph{Gromov-Hausdorff distance} between $P$ and $Q$, to be the infimum of $d_H(\gamma(P),\kappa(Q))$ over all isometric embeddings $\gamma:P\to Z$, $\kappa:Q\to Z$ into a third metric space $Z$.
\end{definition}

The following stability result is well known, and plays a central role in topological data analysis.  Let $d_B$ denote the \emph{bottleneck distance} on persistence barcodes, as defined for example in \cite{chazal2014persistence}.

\begin{theorem}[Stability of Persistent Homology \protect\cite{cohen2007stability,chazal2009gromov,chazal2014persistence}]\label{Thm:Stability}
For any finite metric spaces $P$, $Q$ and $i\geq 0$,
\[d_B(\B_i(P),\B_i(Q))\leq d_{GH}(P,Q).\]
\end{theorem}

The following variant of \cref{Thm:Stability}, which appears in slightly different language in \cite{levanger2019comparison}, can be proven by a slight modification of the proof of \cref{Thm:Stability}.
\begin{theorem}[Stability for a Metric Subspace {\cite[Proposition 5.6]{levanger2019comparison}}]\label{Thm:SubspaceStability}
For finite metric spaces $P\subseteq Q$ and $i\geq 0$,
\[d_B(\B_i(P),\B^S_i(Q))\leq \frac{1}{2}\,d_{H}(P,Q),\]
where $\B^S_i(Q)$ is the barcode obtained by shifting each interval of $\B_i(Q)$ to the right by $\frac{1}{2} d_{H}(P,Q)$.  
\end{theorem}

\subsection{Stability of the Topological Novelty Profile}\label{Sec:Stabilty_Top_Nov_Profile}
For $\eh$ a history indexed by $G$ and $\forest\subseteq G$ the forest of \cref{Def:Top_Novelty_Profile}, define a filtration $\F$ by $\F_r:=\forest\,\cup\, \Rips(\eh)_r$.  (In the expression $\Rips(\eh)$, $\eh$ is understood to denote the metric space of \cref{Def:Sym_Diff_Metric}.)  It's easy to check that $\mathcal T(\eh)$, the topological novelty profile of $\eh$, is exactly the list of right endpoints of intervals in $\B_0(\F)$, possibly with some copies of 0 added in.

We can use this description to obtain a simple stability result for topological novelty profiles analogous to the one for temporal novelty profiles mentioned in \cref{Remark:Temporal_Novelty_Stability}:

\begin{proposition}
\label{Prop:Topological_Novelty_Stability}
Given histories $\eh$ and $\eh'$ indexed by the same phylogenetic graph $G$ with $|d(\eh_v,\eh_w)-d(\eh'_v,\eh'_w)|\leq \epsilon$ for all vertices $v,w$ of $G$, we have \[d_\infty(\mathcal T(\eh),\mathcal T(\eh')) \leq \epsilon.\]
\end{proposition}

\begin{proof}
This follows immediately from a generalized version of the stability theorem for persistent homology, as described in \cite{chazal2009proximity}, \cite{chazal2012structure}, and  \cite{bauer2015induced}.
\end{proof}

\subsection{Discrete Morse Theory}\label{Sec:DMT}
The proof of our main results relies on discrete Morse theory (DMT), a well known combinatorial theory concerning topology-preserving collapses of cell complexes.  We will not need the full strength of standard DMT; we review only what we need for our proof.  See \cite{forman2002user} for a detailed introduction to DMT.

Recall that for $G$ a graph with no self-edges $(v,v)$, a \emph{matching} $X$ in $G$ is a subset of the edges of $G$ such that no two edges in $X$ are incident to the same vertex.  For $S$ a simplicial complex, the \emph{Hasse graph $G_S$} of $S$ is the directed graph with vertices the simplices of $S$ and an edge from $s$ to $s'$ if and only if $s'$ is a codimension-1 face of $s$.   A matching $\vfa$ in $G_S$ is said to be \emph{acyclic} if when we modify the graph $G_S$ by reversing the orientation of all edges in $\vfa$, while leaving the orientation of all other edges unchanged, we obtain a  directed acyclic graph.  

A \emph{discrete gradient vector field (DGVF)} $\vfa$ on $S$ is an acyclic matching in $G_S$.  A simplex $\sigma\in S$ is called \emph{critical} in $\vfa$ if $\sigma$ is not matched in $\vfa$. 

The acyclicity condition admits an alternative formulation which is often convenient.  Given a matching $\vfa$ in $G_S$, we define an \emph{$\vfa$-path} to be a sequence of simplices in $S$ \[\sigma_0, \tau_0, \sigma_1, \tau_1,. . . , \sigma_m ,\tau_m, \sigma_{m+1}\] such that for each $j\in \{0,\ldots,m\}$, the following are true:
\begin{itemize}
\item $\sigma_j$ is a face of $\tau_j$ and $\vfa$ matches $\sigma_j$  to $\tau_j$,
\item $\sigma_{j+1}$ is a codimension-1 face of $\tau_j$,
\item $\sigma_j\ne \sigma_{j+1}$.
\end{itemize} 
We say the $\vfa$-path is a \emph{non-trivial} if $m\geq 0$, and \emph{closed} if $\sigma_0=\sigma_{m+1}$. 

\begin{proposition}[{\cite[Theorem 6.2]{forman2002user}}]\label{Prop:K_Paths}
A matching $\vfa$ in $G_S$ is a DGVF if and only if there exists no non-trivial closed $\vfa$-path.
\end{proposition}

The following is one of the basic results of discrete Morse theory:
\begin{proposition}[{\cite[Theorem 11.13]{kozlov2008combinatorial}}]\label{Prop:DMT}
\mbox{}
\begin{enumerate}[(i)]
\item Suppose that $\vfa$ is a DGFV on a finite simplicial complex $S$.  Then $S$ is homotopy equivalent to a CW-complex with exactly one cell of dimension $i$ for each critical $i$-simplex of $\vfa$.
\item If the critical simplices of $\vfa$ form a subcomplex $S'\subseteq S$, then in fact $S$ deformation retracts onto $S'$.
\end{enumerate}
\end{proposition}

\section{Barcodes of Histories Indexed by Galled Trees}\label{Sec:Barcodes_Of_Histories_Indexed_By_Galled_Trees}
The topological novelty profile and $0^{\mathrm{th}}$ persistence barcode of an evolutionary history are closely related by the following result, whose easy verification we leave to the reader:

For $\B{}$ a barcode, let $\lens(\B{})$ denote the list of lengths of intervals of $\B{}$, sorted in descending order.

\begin{proposition}\label{Prop:Novelty_And_0D_Barcodes}
Suppose we are given a history $\eh$ and $\delta>0$ such that $d(\eh_v,\eh_w)< \delta$ whenever $w$ is a clone with parent $v$.  Then the lists obtained from $\lens(\B_0(\eh))$ and $\T(\eh)$ by removing all entries less than $\delta$ are equal.
\end{proposition}

This suggests that in some cases, $0^{\mathrm{th}}$ barcodes may be useful in the study of recombination.  However, in cases where the minimum $\delta$ satisfying the condition of \cref{Prop:Novelty_And_0D_Barcodes} is large, or where we only have a subsample of the history, $0^{\mathrm{th}}$ barcodes may not offer useful information.  This, together with the earlier theoretical result of Chan, Rabadan and Carlsson relating recombination to higher persistence barcodes (\cref{Thm:Chan_Trees} below), motivates us to consider the relationship between the topological novelty profile and the higher barcodes of a history.

  In this section, we present our main result relating barcodes to novelty profiles in the galled tree setting (\cref{Thm:Lower_Bound}).  The technical heart of our proof is a topological description of the Vietoris--Rips filtration of an \emph{almost linear} metric space, which we give in \cref{Sec:ALMS}.  Our arguments rely heavily on discrete Morse theory. 

\subsection{Barcodes of Histories indexed by Trees}\label{Sec:Barcodes_From_Trees}
We first review the key result of Chan et al. on the barcodes of histories indexed by trees.  

\begin{definition}[Tree-Like Metric Space]
We call an undirected tree with a non-negative weight function on its edges a \emph{weighted tree.}  A metric space $P$ is called \emph{tree-like} if it is isometric to a subspace of a metric space arising from the shortest-path metric on a weighted tree.
\end{definition}

\begin{proposition}[{\cite[Supplementary Information, Theorem 2.1]{chan2013topology}}]\label{Prop:Tree_Like_MS}
If $P$ is a tree-like metric space, then for all $r\in [0,\infty)$, each component of $\Rips(P)_r$ is contractible.  Hence, $\B_i(P)=\emptyset$ for $i\geq 1$.
\end{proposition}

In \cite{chan2013topology}, only the part of \cref{Prop:Tree_Like_MS} about triviality of barcodes is stated, and not the stronger contractibility result.  However, the contractibility result follows immediately from the proof given in \cite{chan2013topology}, using the nerve theorem \cite[Chapter 4.G]{hatcher2002algebraic} in place of a Mayer-Vietoris argument.

The following result, due to Chan et al., makes precise the idea that for $i\geq 1$, a non-empty barcode $\B_i(S)$ serves as a certificate that recombination is present in the history from which $S$ was sampled.

\begin{theorem}[\protect\cite{chan2013topology}]
\label{Thm:Chan_Trees}
If $G$ is a tree, $\eh$ is a history indexed by $G$, and $S\subseteq \eh$, then $\B_i(S)=\emptyset$ for $i\geq 1$.  
\end{theorem}

\begin{proof}
If $S$ is a subset of a history indexed by a tree, then $\met{S}$ is easily seen to be tree-like.  Hence, the result follows from \cref{Prop:Tree_Like_MS}.
\end{proof}

\begin{remark}\label{Rem:Homoplasies2}
In the absence of recombination, homoplasies (recurrent mutations that violate the infinite sites assumption) can lead to a metric space that is not tree-like.  However, as indicated in \cref{Rem:Homoplasies}, a small number of homoplasies causes a correspondingly small deviation from tree-likeness (with respect to Gromov-Hausdorff distance).  A single recombination event, on the other hand, can yield a metric space that is arbitrarily far from a tree-like one.
\end{remark}

\subsection{Metric Decomposition of an Evolutionary History}\label{Sec:Met_Decomposition}
Define a \emph{based metric space} simply to be a metric space $P$, together with a choice of basepoint $p\in P$.

\begin{definition}[Sum of Based Metric Spaces]
For based metric spaces $P$ and $Q$ with basepoints $p\in P$, $q\in Q$, we regard the wedge sum $P\vee Q$ as a metric space, with the metric given by 
\[d_{P\vee Q}(x,y)=
\begin{cases}
d_P(x,y) & \textup{if $x,y\in P$},\\
d_Q(x,y) & \textup{if $x,y\in Q$},\\
d_P(x,p)+d_Q(q,y) & \textup{if $x\in P$, $y\in Q$}.
\end{cases}
\]
\end{definition}

For based metric spaces $P$ and $Q$, let $\Rips(P)\vee\Rips(Q)$ denote the \emph{wedge sum filtration}, given by \[(\Rips(P)\vee\Rips(Q))_r:=\Rips(P)_r\vee \Rips(Q)_r.\]

\begin{proposition}\label{Prop:Barcodes_Of_1_Point_Unions}
For finite based metric spaces $P$ and $Q$, the inclusion \[\Rips(P)\vee \Rips(Q)\hookrightarrow \Rips(P\vee Q)\]
is an objectwise homotopy equivalence.  In particular, for any $i\geq 0$, \[\B_i(P\vee Q)=\B_i(P)\cup \B_i(Q).\]
\end{proposition}

\begin{proof}
We give a proof using discrete Morse theory.  For $r\in [0,\infty)$, if $\sigma$ is a simplex in $\Rips(P\vee Q)_r$ containing vertices in both $P$ and $Q$ but not the common vertex $p=q$, then the simplex $\{p=q\}\cup \sigma$ is clearly also in $\Rips(P\vee Q)_r$.  We define a DGVF on $\Rips(P\vee Q)_r$ by matching each such simplex $\sigma$ to $\{p=q\}\cup \sigma$.  It is clear that this matching is acyclic, hence indeed gives a well-defined DGVF whose set of critical simplices is $(\Rips(P)\vee \Rips(Q))_r$.  Thus, by \cref{Prop:DMT}\,(ii), the inclusion \[(\Rips(P)\vee \Rips(Q))_r\to \Rips(P\vee Q)_r\]  is a homotopy equivalence.  

To check that \[\B_i(P\vee Q)=\B_i(P)\cup \B_i(Q),\] note that by \cref{lem:ObjectwiseHE}, \[\B_i(P\vee Q)=\B_i(\Rips(P)\vee\Rips(Q)),\]  so it suffices to check that \[\B_i(\Rips(P)\vee\Rips(Q))=\B_i(P)\cup \B_i(Q).\] 
A standard result on the homology of wedge sums of topological spaces \cite[Corollary 2.25]{hatcher2002algebraic} furnishes isomorphisms of vector spaces \[\tilde H_i(\Rips(P)\vee\Rips(Q))_r\to \tilde H_i(\Rips(P))_r \oplus \tilde H_i(\Rips(Q))_r\] for each $r\in [0,\infty)$, and these isomorphisms are natural, i.e., they  assemble into an isomorphism of persistence modules \[\tilde H_i(\Rips(P)\vee\Rips(Q))\to \tilde H_i(\Rips(P)) \oplus \tilde H_i(\Rips(Q)).\] This  implies that \[\B_i(\Rips(P)\vee\Rips(Q))=\B_i(P)\cup \B_i(Q).\qedhere\]
\end{proof}

\begin{remark}
\cref{Prop:Barcodes_Of_1_Point_Unions} has a category-theoretic interpretation: It says that reduced persistent homology commutes with coproducts in the categories of based metric spaces and persistence modules, where morphisms of metric spaces are 1-Lipschitz maps sending basepoint to basepoint.
\end{remark}

\begin{remark}
\cref{Prop:Barcodes_Of_1_Point_Unions} has also been discovered independently by the authors of \cite{adamaszek2017vietorisGluing}.  Their work also establishes the result for infinite based metric spaces and for \v Cech filtrations.
\end{remark}

We leave the easy verification of the following to the reader:
\begin{proposition}\label{Prop:Decomposing_Evolutionary_Graphs}
Suppose $G$ is a phylogenetic graph with $G=G^1\grSum G^2$ for subgraphs $G^1,G^2\subseteq G$, $\eh$ is a history indexed by $G$, and $\eh^1$ and $\eh^2$ are the respective restrictions of $\eh$ to $G^1$ and $G^2$.
Then \[\met{\eh}\cong \met{\eh^1}\vee\met{\eh^2}.\]
\end{proposition}

\begin{theorem}\label{Prop:Decomposable_Populations}
Suppose a galled tree $G$ is an iterated sum of source-sink loops $G^1,\ldots,G^k\subset G$ and rooted trees $G^{k+1},\ldots,G^{l}\subset G$, and that $\eh$ is a history indexed by $G$.  Let $\eh^j$ denote the restriction of $\eh$ to $G^j$.  
\begin{enumerate}
\item[(i)] There is an objectwise homotopy equivalence from an iterated wedge sum of the filtrations $\Rips(\eh^j)$ to $\Rips(\eh)$. 
\item[(ii)] For $i\geq 1$,
\[\B_i(\eh)=\bigcup_{j=1}^k \B_i(\eh^j).\]
\end{enumerate}
\end{theorem}

\begin{proof}
(i) follows from \cref{Prop:Decomposing_Evolutionary_Graphs,Prop:Barcodes_Of_1_Point_Unions}. (ii) follows from (i) and \cref{Prop:Tree_Like_MS}.
\end{proof}

\subsection{Vietoris--Rips Filtrations of Almost Linear Metric Spaces}\label{Sec:ALMS}
As mentioned in the introduction, we say a non-empty finite metric space $P$ is \emph{almost linear} if there is a point $p\in P$ such that $P\setminus \{p\}$ is isometric to a finite subset of $\R$.  We call any such point $p$ \emph{a distinguished point}.  See \cref{Fig:AL_Metric_Space}.

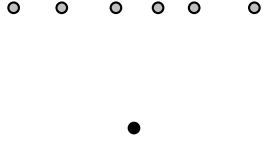
\begin{figure}[h]
\begin{center}
\begin{tikzpicture}[thick,scale=.8,
                   shorten >=2pt+0.5*\pgflinewidth,
                   shorten <=2pt+0.5*\pgflinewidth,
                   every node/.style={circle,
                                      draw,
                                      fill          = black!50,
                                      inner sep     = 0pt,
                                      minimum width =4 pt}]
                                      \tikzset{
    every node/.style={
        circle,
        draw,
        fill          = gray!50,
        inner sep     = 0pt,
        minimum width =4 pt
    }   
}  
      \coordinate (ghost) at (0,.5);
       \coordinate (a) at (0,0);
       \coordinate (b) at (.8,0);
       \coordinate (c) at (1.7,0);
        \coordinate (d) at (2.4,0);
        \coordinate (e) at (3,0);
        \coordinate (f) at (4,0);
        \coordinate (p) at (2,-2);
          
\path[draw] 

       
       node[color=white] at (ghost) {}
       node at (a) {} 
       node at  (b) {}
       node at  (c) {} 
       node at  (d) {}
       node at  (e) {}
       node at  (f) {}
       node[fill=black] at  (p) {};
\end{tikzpicture}
\end{center}
\caption{An almost linear metric space embedded in $\R^2$.  The unique distinguished point is shown in solid black.  Note that not all almost linear metric spaces can be embedded in $\R^2$.}
\label{Fig:AL_Metric_Space}
\end{figure}

\begin{proposition}\label{Lem:Loop_Almost_Linear}
If $\eh$ is a history indexed by a source-sink loop, then $\met{\eh}$ is almost linear.
\end{proposition}

\begin{proof}
Let $p$ be the unique recombinant.  $\met{ \eh\setminus \{\eh_p\}}$ is isometric to a subset of $\R$.  
\end{proof}

In view of \cref{Prop:Decomposable_Populations} and \cref{Lem:Loop_Almost_Linear}, to understand the topology of Vietoris--Rips filtrations of histories indexed by galled trees, it suffices to understand the topology of Vietoris--Rips filtrations of almost linear metric spaces.  We now describe the latter:

\begin{theorem}[Topology of the Vietoris--Rips Filtration of an Almost Linear Metric Space]\label{Thm:ALMS}
Let $P$ be an almost linear metric space with distinguished point $p$. 
\begin{enumerate}[(i)]
\item For each $r\in [0,\infty)$, the connected component $C_r$ of $\Rips(P)_r$ containing $p$ is either contractible or homotopy equivalent to a circle, and each other component of $\Rips(P)_r$ is contractible.  In particular,  $\B_i(P)=0$ for $i\geq 2$.
\item If $C_r$ and $C_{r'}$ are both homotopy equivalent to circles and $r\leq r'$, then the inclusion $C_r\hookrightarrow C_{r'}$ is a homotopy equivalence.  Thus, $\B_1(P)$ has at most one interval.
\item The unique interval of $\B_1(P)$, when it exists, has length at most $d(p,P\setminus \{p\})$ and is contained in the interval \[\left[d(p,P\setminus \{p\}),\mathrm{diameter}(P\setminus\{p\})/2)\right).\]
\end{enumerate}
\end{theorem}

\begin{remark}
Together, \cref{Prop:Decomposable_Populations}\,(i), \cref{Lem:Loop_Almost_Linear}, and \cref{Thm:ALMS}\,(i) tell us that for $\eh$ a history indexed by a galled tree $G$ and $r\in [0,\infty)$, each component of $\Rips(\eh)_r$ is homotopy equivalent to a bouquet of circles.  
\end{remark}

\begin{remark}\label{Rem:AlmostTreeLike}
In analogy with the definition of an almost linear metric space, we can define an \emph{almost tree-like} metric space to be one obtained from a tree-like metric space by adding a single point.  In an earlier version of this paper, we conjectured that \cref{Thm:ALMS} also holds for almost tree-like metric spaces, but Matthew Zaremsky showed us the following simple counterexample: Let $d$ be the metric on $\{a,b,c,z\}$ given by \[d(a,z)=d(b,z)=d(c,z)=1,\quad d(a,b)=d(a,c)=d(b,c)=2.\]  This metric is tree-like; we can take the tree to be the star centered at $z$.  We extend $d$ to a metric on $\{a,b,c,z,p\}$ by taking $d(p,z)=2$ and $d(p,a)=d(p,b)=d(p,c)=1$.  The resulting metric space $P$ is almost tree-like, but $\B_1(P)$ consists of two intervals.   Thus, \cref{Thm:ALMS}\,(ii) does not extend to almost tree-like metric spaces.
\end{remark}

We build up to the proof of \cref{Thm:ALMS} with several definitions and lemmas.  In what follows, let $P$ be an almost linear metric space with distinguished point $p$, and let $r\in [0,\infty)$ be such that $\Rips(P)_r$ is connected.  By choosing an isometric embedding $P\setminus\{p\}\hookrightarrow\R$, we may regard $P\setminus\{p\}$ as a subset of $\R$.  
\begin{definition}
Let $\Pleft\subset P\setminus\{p\}$ denote the set of points $y$ such that
\begin{enumerate}
\item[1.] $[p,y]\not\in\Rips(P)_r$,
\item[2.] there is no $w\in P\setminus\{p\}$ satisfying each of the following conditions:
\begin{itemize}
\item $w<y$, 
\item $w$ and $y$ lie in the same connected component of $\Rips(P\setminus\{p\})_r$,
\item $[p,w]\in\Rips(P)_r$.
\end{itemize}
\end{enumerate}
\end{definition}

See \cref{Fig:The_Set_E} for an illustration of $\Pleft$.

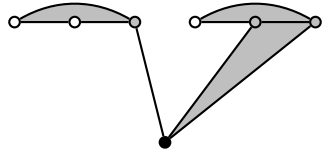
\begin{figure}[h]
\begin{center}
\begin{tikzpicture}[thick,scale=.8,
                   shorten >=2pt+0.5*\pgflinewidth,
                   shorten <=2pt+0.5*\pgflinewidth,
                   every node/.style={circle,
                                      draw,
                                      fill          = gray!50,
                                      inner sep     = 0pt,
                                      minimum width =4 pt}]
                                      \tikzset{
    every node/.style={
        circle,
        draw,
        fill          = gray!50,
        inner sep     = 0pt,
        minimum width =4 pt
    }   
}  

       \coordinate (a) at (1,0);
       \coordinate (b) at (2,0);
       \coordinate (c) at (3,0);
        \coordinate (d) at (4,0);
        \coordinate (e) at (5,0);
        \coordinate (f) at (6,0);
        \coordinate (p) at (3.5,-2);
        
 \fill[fill=gray!50] (a) to [bend left] (c) -- cycle;
 \fill[fill=gray!50] (d) to [bend left] (f) -- cycle;
  \fill[fill=gray!50] (e) -- (f) -- (p) -- cycle;               
        
          
\path[draw] 

       
       node[fill=white] at (a) {} 
       node[fill=white] at  (b) {}
       node at  (c) {} 
       node[fill=white] at  (d) {}
       node at  (e) {}
       node at  (f) {}
       node[fill=black] at  (p) {};
           

   \draw (a) -- (b) ;
    \draw (b) -- (c);
    \draw (a) to [bend left] (c);
     \draw (d) -- (e);
    \draw (e) -- (f);
      \draw (d) to [bend left] (f);
    \draw (c) -- (p);
    \draw (e) -- (p);
     \draw (f) -- (p);


\end{tikzpicture}
\end{center}
\caption{Illustration of the Vietoris--Rips complex $\Rips(P)_r$ for an almost linear metric space $P$, and some choice of scale parameter $r$.  Here, the metric on $P$ is not assumed to be the one given by the shown embedding of the points in the plane.  The distinguished point is solid black, points of $\Pleft$ are white, and the remaining points are gray.}
\label{Fig:The_Set_E}
\end{figure}

\begin{lemma}\label{Lem:Def_Retract}
$\Rips(P)_r$ deformation retracts onto $\Rips(P\setminus \Pleft)_r$.
\end{lemma}

\begin{proof}
We give a simple discrete Morse theory argument.  Define a DGVF $W$ on $\Rips(P)_r$ as follows: For $j\geq 2$ and \[\sigma:=[a_1<a_2<\cdots <a_j]\] a simplex in $\Rips(P)_r$ such that $a_1\in \Pleft$ and $a_2$ is the point in $P$ immediately to the right of $a_1$, $W$ matches $\sigma$ to its face $[a_1,a_3,\ldots,a_j]$.  To see that $W$ is acyclic, note that for any $W$-path 
\[\sigma_0, \tau_0,\ldots, \sigma_m ,\tau_m, \sigma_{m+1},\] 
the $\tau_j$ are strictly increasing with respect to the lexicographical order induced by the vertex ordering.  If $m\geq 0$ and $\sigma_0=\sigma_{m+1}$, then 
\[\sigma_0, \tau_0,\ldots, \sigma_m ,\tau_m, \sigma_0,\tau_0,\sigma_1\] 
is a $W$-path with $\tau_m<\tau_0$, so there cannot exist a non-trivial closed $W$-path.  Therefore $W$ is acyclic.  

Furthermore, $W$ matches every simplex containing a point in $\Pleft$, so the critical simplices of $W$ form the subcomplex  $\Rips(P\setminus \Pleft)_r$.  Hence, $\Rips(P)_r$ deformation retracts onto $\Rips(P\setminus \Pleft)_r$ by \cref{Prop:DMT}\,(ii).  
\end{proof}

Let us now assume that $\Pleft=\emptyset$.  We next define a discrete gradient vector field $\vfb$ on $\C:=\Rips(P)_r$.  We do so in two steps, first giving a simple definition of a DGVF $\vfa$ on $\C:=\Rips(P)_r$, and then extending this by matching more simplices.  To start, we order the vertices in $P$ by taking $\{p\}$ to be the minimum, and ordering $P-\{p\}$ from left to right, via the chosen embedding of $P-\{p\}$ into $\R$.  Henceforth, it will be our convention that the vertices of a simplex in $\C$ are always written in increasing order.

The definition of $\vfa$ is an instance of a general construction due to Matt Kahle \cite[Section 5]{kahle2011random}, which in fact gives a DGVF on any simplicial complex with ordered vertex set: 

\begin{definition}[The Discrete Gradient Vector Field $X$]
If a simplex $\sigma = [a_1,a_2 ,...,a_j]$ of on $\C$ has a coface $a_0\cup \sigma:=[a_0,a_1,a_2,\ldots,a_j]$ with $a_0 < a_1$, then $\vfa$ matches $\sigma$ to $a_0\cup \sigma$ with $a_0$ as small as possible.  $\vfa$ matches no other simplices.  It is easy to check that this in fact gives a well-defined DGVF. 
\end{definition}

See \cref{Fig:Vector_Field_K} for an illustration of the DGVF $X$.  

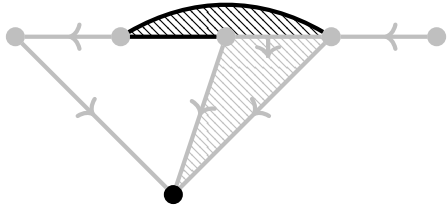
\begin{figure}[h]
\begin{center}
\begin{tikzpicture}[ultra thick,scale=1.4,
                   every node/.style={circle,
                                      draw,
                                      fill          = gray!50,
                                      inner sep     = 0pt,
                                      minimum width =4 pt}]
                                      \tikzset{
    every node/.style={
        circle,
        draw=gray!50,
        fill          = blue,
        inner sep     = 0pt,
        minimum width =6 pt
    }   
}  

        \coordinate (d) at (4,0);
        \coordinate (e) at (5,0);
          \coordinate (eg) at (5.8,.30);
          \coordinate (egminus) at (5.8,.10);
        \coordinate (f) at (6,0);
        \coordinate (fg) at (6.4,0);
        \coordinate (fgminus) at (6.4,-.2);
         \coordinate (g) at (7,0);
         \coordinate (h) at (8,0);
        \coordinate (p) at (5.5,-1.5);
        
 \fill[pattern=north west lines, pattern color=black] (e) to [bend left] (g) -- cycle;
 \fill[pattern=north west lines, pattern color=gray!50] (f) -- (g) -- (p) -- cycle;               
        
          
\path[draw] 

       
       node[fill=gray!50,very thick] at  (d) {}
       node[fill=gray!50,very thick] at  (e) {}
       node[fill=gray!50,very thick] at  (f) {}
       node[fill=gray!50,very thick] at  (g) {}
        node[fill=gray!50,very thick] at  (h) {}
       node[draw=black,fill=black,very thick] at  (p) {};
           

	 \draw[color=gray!50,postaction={decorate,decoration={markings,
        mark=at position 1 with {\arrow[gray!50, line width=.5mm]{>};}}}] (fg) to (fgminus);
%

                     \tikzset{shorten >=2.8pt+0.5*\pgflinewidth,
                   shorten <=2.7pt+0.5*\pgflinewidth}

               	 \draw[color=gray!50] (f) to (g);
                  \draw[color=black] (e) to [bend left] (g);

\tikzset{decoration={
    markings,
    mark=at position 0.5 with {\arrow{>}}}}

    
     \draw[color=gray!50,postaction=decorate] (d) -- (p);
     \draw[color=gray!50,postaction=decorate] (e) -- (d);
    \draw (e) -- (f);

     \draw[gray!50,postaction=decorate] (f) -- (p);
     \draw[gray!50,postaction=decorate] (h) -- (g);
    \draw[gray!50,postaction=decorate] (g) -- (p);


\end{tikzpicture}
\end{center}
\caption{Illustration of the discrete gradient vector field $\vfa$ on a Vietoris--Rips complex of an almost-linear metric space, with the bottom vertex ordered first, and the remaining vertices ordered left-to-right.  Matched simplices are gray, and matched pairs are denoted with an arrow pointing away from the simplex of lower dimension.  Critical simplices are black.  Thus, $\vfa$ has a one critical 0-simplex, two critical 1-simplices, and one critical 2-simplex.}
\label{Fig:Vector_Field_K}
\end{figure}

Clearly, $[p]$ is critical in $\vfa$, and since we assume that $\Pleft=\emptyset$, no other vertex is critical.  The following describes the remaining critical simplices in $\vfa$:

\begin{lemma}\label{lem:Criticality_Characterization}
For $j\geq 2$, a simplex $[a_1,\ldots,a_j]$ is critical in $\vfa$ if and only if the following three conditions are satisfied:
\begin{enumerate}
\item[1.] $a_1\ne p$, 
\item[2.] $[q,a_1,\ldots,a_j]\not\in \C$ for any $q<a_1$,
\item[3.]  $[p,a_2,a_3,\ldots a_j]\in \C$.
\end{enumerate}
\end{lemma}

\begin{proof}
If all three conditions hold, then by condition 2, $\sigma:=[a_1,\ldots,a_j]$ cannot be the simplex of lower dimension in a pair matched by $\vfa$, and by condition 3, $\vfa$ matches $[a_2,a_3,\ldots a_j]$ to $[p,a_2,a_3,\ldots a_j]$, so by condition 1, $\sigma$ cannot be the simplex of higher dimension in a pair matched by $\vfa$.  Thus $\sigma$ is critical in $\vfa$.

Conversely, if $\sigma$ is critical in $\vfa$, then condition 1 holds, for else $\sigma$ would match to $[a_2,a_3,\ldots a_j]$.  Condition 2 holds, for else $\sigma$ would match to a simplex of higher dimension.   Finally, condition 3 holds, for else $[a_2,a_3,\ldots,a_j]$ would match to some simplex $[q,a_2,\ldots,a_j]\in \C$ with $p< q < a_1$, implying that $[q,a_1,\ldots,a_j]\in \C$, and hence contradicting the criticality of $\sigma$.
\end{proof}

\begin{remark}
Note that \cref{lem:Criticality_Characterization} implies in particular that if $[a_1,...,a_j]$ is critical, then $a_1$ is not incident to $p$, since otherwise, in view of condition 3, condition 2 would be violated.
\end{remark}

\cref{lem:Criticality_Characterization} suggests a way to extend $\vfa$ to a DGVF $\vfb$ with the desired properties: 

\begin{definition}[The Discrete Gradient Vector Field $Y$]
For $[a_1,a_2,a_3,\ldots,a_j]$ a critical simplex for $\vfa$ with $j\geq 3$, suppose there exists no vertex $b$ such that $[p,b]\in \C$ and $a_1<b<a_2$.  It follows easily from \cref{lem:Criticality_Characterization} that $[a_1,a_3,\ldots,a_j]$ is also critical in $\vfa$.  We match $[a_1,a_2,a_3,\ldots,a_j]$ to $[a_1,a_3,\ldots,a_j]$ in $\vfb$.  We take all matched pairs in $\vfb\setminus \vfa$ to be of this form.  
\end{definition}

\begin{lemma} The matching $Y$ is acyclic, hence a DGVF.  
\end{lemma}
\begin{proof}
We claim that in any $\vfb$-path 
\[\sigma_0, \tau_0,\ldots, \sigma_m ,\tau_m, \sigma_{m+1},\]  no two distinct $\tau_j$ are equal.  From this, it follows that there does not exist a non-trivial closed $\vfb$-path, so $\vfb$ is indeed acyclic.  
To verify the claim, we make three simple observations:  Letting $\tau_j^1$ denote the minimum vertex in $\tau_j$, we have that for any $j\in \{0,\ldots, m-1\}$,
\begin{enumerate}
\item[1.] If $\tau_{j+1}$ is matched by $\vfa$, then $\tau_{j}^1>\tau_{j+1}^1$.  
\item[2.] If $\tau_{j+1}$ is matched by $\vfb\setminus \vfa$, then $\tau_{j}$ is matched by $\vfa$ and $\tau_{j+1}^1=\tau_j^1$.    
\item[3.] $\tau_j\ne\tau_{j+1}$.
\end{enumerate}
By observations $1$ and $2$, we have that $\tau_{j}^1>\tau_{k}^1$ for all $k\in \{j+2,j+3,\ldots, m\}$. The claim follows from this and observation 3.
\end{proof}

\cref{Fig:Vector_Field_F} illustrates the extension of the DGVF $\vfa$ of example \cref{Fig:Vector_Field_K} to the DGVF $\vfb$.  

\begin{figure}[h]
\begin{center}
\begin{tikzpicture}[ultra thick,scale=1.4,
                   every node/.style={circle,
                                      draw,
                                      fill          = black!50,
                                      inner sep     = 0pt,
                                      minimum width =4 pt}]
                                      \tikzset{
    every node/.style={
        circle,
        draw=gray!50,
        fill          = blue,
        inner sep     = 0pt,
        minimum width =6 pt
    }   
}  

        \coordinate (d) at (4,0);
        \coordinate (e) at (5,0);
          \coordinate (eg) at (5.8,.30);
          \coordinate (egminus) at (5.8,.10);
        \coordinate (f) at (6,0);
        \coordinate (fg) at (6.4,0);
        \coordinate (fgminus) at (6.4,-.2);
         \coordinate (g) at (7,0);
         \coordinate (h) at (8,0);
        \coordinate (p) at (5.5,-1.5);
        
 \fill[pattern=north west lines, pattern color=gray!50] (e) to [bend left] (g) -- cycle;
 \fill[pattern=north west lines, pattern color=gray!50] (f) -- (g) -- (p) -- cycle;               
        
          
\path[draw] 

       
       node[fill=gray!50,very thick] at  (d) {}
       node[fill=gray!50,very thick] at  (e) {}
       node[fill=gray!50,very thick] at  (f) {}
       node[fill=gray!50,very thick] at  (g) {}
       node[fill=gray!50,very thick] at  (h) {}
       node[draw=black,fill=black,very thick] at  (p) {};
           

	 \draw[color=gray!50,postaction={decorate,decoration={markings,
        mark=at position 1 with {\arrow[gray!50, line width=.5mm]{>};}}}] (fg) to (fgminus);
        
   \draw[color=gray!50,postaction={decorate,decoration={markings,
        mark=at position 1 with {\arrow[gray!50, line width=.5mm]{>};}}}] (eg) to (egminus);

                     \tikzset{shorten >=2.8pt+0.5*\pgflinewidth,
                   shorten <=2.7pt+0.5*\pgflinewidth}

               	 \draw[color=gray!50] (f) to (g);
                  \draw[color=gray!50] (e) to [bend left] (g);

\tikzset{decoration={
    markings,
    mark=at position 0.5 with {\arrow{>}}}}

    
     \draw[color=gray!50,postaction=decorate] (d) -- (p);
     \draw[color=gray!50,postaction=decorate] (e) -- (d);
     \draw[color=gray!50,postaction=decorate] (h) -- (g);
    \draw (e) -- (f);

     \draw[gray!50,postaction=decorate] (f) -- (p);
    \draw[gray!50,postaction=decorate] (g) -- (p);


\end{tikzpicture}
\end{center}
\caption{The extension of the DGVF $\vfa$ of example \cref{Fig:Vector_Field_K} to the DGVF $\vfb$.  $\vfb$ contains one pair of matched simplices not in $\vfa$: The curved 1-simplex in the top of the figure now is matched with its coface.  Thus, $\vfb$ has two critical simplices: A critical 0-simplex and a critical 1-simplex.}
\label{Fig:Vector_Field_F}
\end{figure}
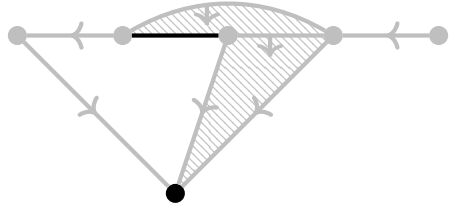

\begin{lemma}\label{Lem:DGVF_Critical_Simplices}
The critical simplices of $\vfb$ are $[p]$ and the 1-simplices $[a_1,a_2]$ such that
\begin{enumerate}
\item[1.] $[a_1,a_2]$ satisfies the conditions of \cref{lem:Criticality_Characterization} and
\item[2.] $[p,b]\not \in \C$ for all $a_1<b<a_2$.
\end{enumerate}
In particular, $\vfb$ has a single critical 0-simplex, and no critical simplices of dimension greater than one.
\end{lemma}
\begin{proof}
Since $\vfb$ is an extension of $\vfa$, any critical simplex of $\vfb$ is a critical simplex of $\vfa$.  It is easy to see that $\vfb$ matches every critical simplex of $\vfa$, except $[p]$ and those 1-simplices satisfying condition 2.   A 1-simplex is critical in $\vfa$ if and only if it satisfies the conditions of \cref{lem:Criticality_Characterization}, so the result follows.
\end{proof}

\begin{lemma}\label{Lem:Wedge_Of_Circles}
For $P$ an almost linear metric space and $r\in [0,\infty)$, each component of $\C=\Rips(P)_r$ is contractible or deformation retracts onto a wedge sum of finitely many circles.
\end{lemma}

\begin{proof}
Any component of $\C$ not containing $p$ is tree-like, and so is contractible by \cref{Prop:Tree_Like_MS}.  Thus, we may assume loss of generality that $\C$ is connected.  Moreover, by \cref{Lem:Def_Retract}, we may assume without loss of generality that $\Pleft=\emptyset$.  The DGVF $Y$ on $\C$ is defined under these assumptions.  The result now follows from \cref{Lem:DGVF_Critical_Simplices} and \cref{Prop:DMT}\,(i).
\end{proof}

\begin{proof}[Proof of \cref{Thm:ALMS}\,(i)]
As in the proof of \cref{Lem:Wedge_Of_Circles}, we may assume without loss of generality that $\C:=\Rips(P)_r$ is connected, and that $\Pleft=\emptyset$.  The fundamental group of a wedge sum of circles is free \cite[Example 1.21]{hatcher2002algebraic}, so by \cref{Lem:Wedge_Of_Circles}, $\pi_1(\C,p)$ is free.  To establish \cref{Thm:ALMS}\,(i), it suffices to show that $\pi_1(\C,p)$ is trivial or cyclic.  

To show this, we first note that the DGVF $\vfb$ provides us with a basis for $\pi_1(\C,p)$, as follows: 
Let $\Cr$ denote the set of critical 1-simplies of $\vfb$, as described by \cref{Lem:DGVF_Critical_Simplices}.  For $\sigma=[b,c]\in \Cr$ with $b<c$, let $a\in P-\{p\}$ denote the maximum vertex such that $a<b$ and $[p,a]\in \C$.  Such $a$ always exists by our assumption that $\Pleft=\emptyset$.  Let us regard $S^1$ as a based topological space, with the basepoint denoted as $1$, and let $\gamma_\sigma:S^1\to [p,a]\cup[a,c]\cup[p,c]$ be a homeomorphism sending $1$ to $p$.  

We now observe that $G:=\{\gamma_\sigma\mid \sigma\in \Cr\}$
 is a basis for $\pi_1(\C,p)$.  For $\sigma\in \Cr$, let $S^1_\sigma$ denote a copy of $S^1$.  The proof of \cref{Prop:DMT}\,(i) presented in \cite{kozlov2008combinatorial} gives a (not necessarily unique) homotopy equivalence $h:\C\to \vee_{\sigma\in \Cr} S^1_\sigma$ mapping the interior of $\sigma$ homeomorphically to $S^1_\sigma\setminus\{1\}$, so that 
$h\circ \gamma_\sigma$ is homotopic either to the inclusion $\iota_\sigma:S^1_\sigma\hookrightarrow \vee_{\sigma\in \Cr} S^1_\sigma$, or to its inverse in $\pi_1(\vee_{\sigma\in \Cr} S^1_\sigma,1)$.  Since $h$ is a homotopy equivalence and $\{\iota_\sigma \mid \sigma\in \Cr\}$ is a basis for $\pi_1(\vee_{\sigma\in \Cr} S^1_\sigma,1)$, we see that $G$ is a basis for $\pi_1(\C,p)$, as desired.

To finish the proof of \cref{Thm:ALMS}\,(i), it remains to show that $|G|\leq 1$.  To do so, we apply the triangle inequality.  Our argument is illustrated in \cref{Fig:Triangle_Inequality}.  
For $[b,c]=\sigma\in \Cr$ with $b<c$, let $a<b$ be as above, and for $[b',c']=\sigma'\in \Cr$ with $b'<c'$, define $a'<b'$ in the same way.  To arrive at a contradiction, suppose $\sigma\ne \sigma'$.  Then either $c\leq a'$ or $c'\leq a$.  Switching the labels of $\sigma$ and $\sigma'$ if necessary, we may assume without loss of generality that $c\leq a'$.  We have $[p,a],[p,c']\in \C$, so $d(a,p)\leq 2r$ and $d(p,c')\leq 2r$.  By the triangle inequality, $d(a,c')\leq 4r$.  Thus, since $P\setminus\{p\}$ is isometric to a subset of $\R$, we have \[d(a,c)+d(a',c')\leq d(a,c') \leq 4r.\]  Therefore either $d(a,c)\leq 2r$ or $d(a',c')\leq 2r$, so either $[a,c]\in \C$ or $[a',c']\in \C$.  But then either $\gamma_\sigma$ or $\gamma_{\sigma'}$ is nullhomotopic in $\C$, contradicting that $G$ is a basis  for $\pi_1(\C,p)$.  \end{proof}

\begin{figure}[h]
\begin{center}
\begin{tikzpicture}[ultra thick,scale=1.2]

                                                         \tikzset{
    every node/.style={
        circle,
        draw=gray!50,
        fill          = blue,
        inner sep     = 0pt,
        minimum width =6 pt
    }   
}

       \coordinate (a) at (1,0);
       \coordinate (b) at (2,0);
       \coordinate (c) at (3,0);
        \coordinate (d) at (4,0);
        \coordinate (e) at (5,0);
         \coordinate (eg) at (5.8,.30);
         \coordinate (egminus) at (5.8,.10);
        \coordinate (f) at (6,0);
        \coordinate (fg) at (6.4,0);
        \coordinate (fgminus) at (6.4,-.2);
         \coordinate (g) at (7,0);
        \coordinate (p) at (4,-2);

        
          
\path[draw] 

           
       node[fill=gray!50,very thick] at (a) {} 
       node[fill=gray!50,very thick] at  (b) {}
       node[fill=gray!50,very thick] at  (c) {} 
       node[fill=gray!50,very thick] at  (d) {}
       node[fill=gray!50,very thick] at  (e) {}
       node[fill=gray!50,very thick] at  (f) {}
       node[draw=black,fill=black,very thick] at  (p) {};
           

%

                                                        \tikzset{
    every node/.style={}}   
\path[draw]
node[above] at (1,.05) {$a$}
node[above] at (2,.05) {$b$}
node[above] at (2.5,0) {$\sigma$}
node[above] at (3,.05) {$c$}
node[above] at (4,.05) {$a'$}
node[above] at (5,.05) {$b'$}
node[above] at (5.5,0) {$\sigma'$}
node[above] at (6,.05) {$c'$};

                     \tikzset{shorten >=2.8pt+0.5*\pgflinewidth,
                   shorten <=2.7pt+0.5*\pgflinewidth}



   \draw[gray!50,postaction=decorate] (b) -- (a) ;
      \draw[gray!50,postaction=decorate] (a) -- (p) ;
    \draw (b) -- (c);
    
     \draw[gray!50,postaction=decorate] (d) -- (p);
     \draw[gray!50,postaction=decorate] (e) -- (d);
    \draw (e) -- (f);

    \draw[gray!50,postaction=decorate] (c) -- (p);
     \draw[gray!50,postaction=decorate] (f) -- (p);
     
     \draw[black!40,thick,decorate,decoration={brace,amplitude=12pt},dashed] (4.0,-2.10) -- (.9,0);
      \draw[black!40,thick,decorate,decoration={brace,amplitude=12pt},dashed] (6.1,0) -- (4.0,-2.10);
      
      \path[draw]
            node at (2.0,-1.6) {$\leq 2t$}
            node at (5.6,-1.5) {$\leq 2t$};
    

\end{tikzpicture}
\end{center}
\caption{Illustration of the argument by contradiction that $|G|\leq 1$ in the proof of \cref{Thm:ALMS}\,(i).  Critical simplices are black and matched simplices are gray.  By the triangle inequality, $d(a,c')\leq 4r$, so since $\{a<b<c\leq a'<b'<c'\}$ is isometric to a subset of $\R$, either $[a,c]\in \C$ or $[a',c']\in \C$.}
\label{Fig:Triangle_Inequality}
\end{figure}
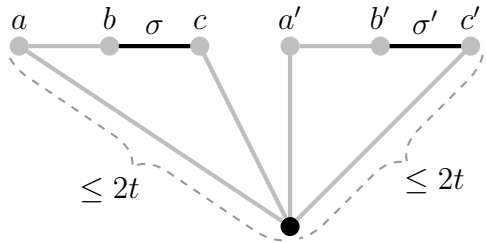

\begin{proof}[Proof of \cref{Thm:ALMS}\,(ii)]

As in the statement of the theorem, let $C_r$ denote the component of $\Rips(P)_r$ containing $\{p\}$.  We need to show that for $r\leq r'\in [0,\infty)$, if $C_r\simeq S^1\simeq C_{r'}$, then the inclusion $C_r\hookrightarrow C_{r'}$ is a homotopy equivalence.  Let $\gamma_\sigma:S^1\to C_r$ and $\gamma_{\sigma'}:S^1\to C_{r'}$ be the generators for $\pi_1(C_{r},p)$ and $\pi_1(C_{r'},p)$ specified in the proof of \cref{Thm:ALMS}\,(i) above.

Given the way $\gamma_{\sigma}$ and $\gamma_{\sigma'}$ are defined, exactly one of the following must be true:
\begin{enumerate}
\item[1.] $c\leq a'$,
\item[2.] $c'\leq a$,
\item[3.] $a\leq a' < c'\leq c$.
\end{enumerate}

We show that we cannot have $c\leq a'$ using essentially the same triangle inequality argument we used in the proof of \cref{Thm:ALMS}\,(i): Suppose otherwise.  Then $d(a,p)<2r$ and $d(c',p)<2r'$.  By the triangle inequality, $d(a,c')\leq 2(r+r')$, so we have \[d(a,c)+d(a',c')\leq d(a,c') \leq 2(r+r').\]  Therefore either $d(a,c)\leq 2r$ or $d(a',c')\leq 2r'$, leading to a contradiction as above.

The same argument shows that we cannot have $c'\leq a$.  Therefore, we must have $a\leq a' < c'\leq c$. 

 We will show that if $a\ne a'$, then $[a,a']\in C_{r'}$: We have $d(a,p)\leq 2r$ and $d(c',p)\leq 2r'$, so by the triangle inequality, $d(a,c')\leq 2(r+r')$.  Therefore either $d(a,a')\leq 2r'$ or $d(a',c')\leq 2r$.  But since $\gamma_{\sigma'}$ is not nullhomotopic by assumption, we must have $d(a',c')>2r'\geq 2r$, so $d(a,a')\leq 2r'$.  Thus $[a,a']\in C_{r'}$, as desired.  It follows that $[p,a,a']\in C_{r'}$.

The symmetric argument shows that if $c'\ne c$, then $[p,c',c]\in C_{r'}$.  Letting \[j:C_{r}\hookrightarrow C_{r'}\] denote the inclusion, we thus have that $j\circ \gamma_\sigma\sim \gamma_{\sigma'}$.  Since $\gamma_\sigma$ and $\gamma_{\sigma'}$ are both homotopy equivalences, $j$ must be a homotopy equivalence as well.  
\end{proof}

\begin{proof}[Proof of \cref{Thm:ALMS}\,(iii)]
Given the form of the set $G$ of generators for $\pi_1(\C(P)_r,p)$ given in the proof of \cref{Thm:ALMS}\,(i), it is clear that if \[r\not \in \left[d(p,P\setminus \{p\}),
\mathrm{diameter}(P\setminus\{p\})/2\right),\] then $\pi_1(\C(P)_r,p)$ is trivial.  By $(i)$ then, each component of $\C(P)_r$ is contractible.  Hence, the unique interval of $\B_1(P)$, if it exists, is contained in 
\[\left[d(p,P\setminus \{p\}),
\mathrm{diameter}(P\setminus\{p\})/2\right).\]

To finish the proof of (iii), we need to show that the unique bar of $\B_1(P)$ is of length at most $d(p,P\setminus \{p\})$.  This follows from the stability of persistent homology.  To see this, note that since $P\setminus \{p\}$ is isometric to a subset of $\R$, it is tree-like, so \cref{Prop:Tree_Like_MS} gives that $\B_1(P\setminus \{p\})=\emptyset$.  Therefore, by \cref{Thm:SubspaceStability},
\[2\, d_B(\B_1(P),\emptyset) =2\, d_B(\B_1(P),\B_1(P\setminus \{p\}))  \leq d_{H}(P,P\setminus \{p\})=d(p,P\setminus \{p\}),\]
where the last equality follows from the definition of $d_H$.  The bottleneck distance of any barcode $\B$ to the empty barcode is half the length of the longest interval of $\B$, so the result follows.
\end{proof}

\subsection{Inference about Recombination from Barcodes}\label{Sec:Inference}
As an immediate corollary of the results of \cref{Sec:Met_Decomposition,Sec:ALMS}, we now obtain our main result relating barcodes to recombination in the galled tree setting.

Recall from \cref{Sec:NoveltyProfiles} that $\T(\eh)$ denotes the topological novelty profile of a history $\eh$, and that the temporal novelty of a recombinant $r$ (with respect to some choice of time function) is denoted as $\mathcal N(r)$.  Recall also from \cref{Equality_Of_Nov_Profiles_For _Galled_Trees} that when $\eh$ is indexed by a galled tree, $\T(\eh)$ is equal to the temporal novelty profile of $\eh$, with respect to any time function.
 
For $G$ a phylogenetic graph, let $\mathcal R^G$ denote the set of recombinants of $G$.  As in the beginning of \cref{Sec:Barcodes_Of_Histories_Indexed_By_Galled_Trees}, for $\B{}$ a barcode, let $\lens(\B{})$ denote the list of lengths of intervals of $\B{}$, sorted in descending order.

\begin{theorem}\label{Thm:Lower_Bound}
Let $\eh$ be a history indexed by a galled tree $G$. 
\begin{enumerate}
\item[(i)] \cref{Prop:Decomposable_Populations}\,(ii) and {\cref{Thm:ALMS}\,(ii)} yield a canonical injection \[\phi:\B_1(\eh)\hookrightarrow\mathcal{R}^G,\] such that $\len(I) \leq \mathcal N(\phi(I))$ for all $I\in \B_1(\eh)$.
In particular, \[\lens(\B_1(\eh))\leq \T(\eh).\]
\item[(ii)] $\B_i(\eh)=\emptyset$ for $i\geq 2$.
\end{enumerate}
\end{theorem}

\begin{proof}
For $G$ a galled tree, each $r\in \mathcal R^G$ corresponds to an entry of $\mathcal T(\eh)$; in fact, this entry is easily seen to be $d(\eh^{L\setminus r},\eh_r)$, where $L$ denotes the source-sink loop corresponding to $R$, and $\eh^{L\setminus r}$ denotes the restriction of $\eh$ to vertices of $L$ other than $r$. (i) now follows from {\cref{Thm:ALMS}\,(iii)}.

(ii) is immediate from \cref{Prop:Decomposable_Populations}\,(ii), \cref{Lem:Loop_Almost_Linear}, and {\cref{Thm:ALMS}\,(i)}.
\end{proof}

\begin{example}\label{Ex:Subsample_Counterexample}
Given the analogy between \cref{Thm:Chan_Trees} (for trees) and \cref{Thm:Lower_Bound} (for galled trees), and the fact that \cref{Thm:Chan_Trees} holds for arbitrary subsamples of a history, it is natural to ask whether \cref{Thm:Lower_Bound} also holds for arbitrary subsamples.  The example shown in \cref{Fig:Subsample_Counterexample} demonstrates that \cref{Thm:Lower_Bound}\,(i) does not hold for arbitrary subsamples; the example, discovered by computer, is a subset $S$ of a history $\eh$ indexed by a galled tree with a single recombinant, for which $\B_1(S)=\{[5,6),[5,6)\}$ and $\B_1(\eh)=\emptyset$.  We conjecture that \cref{Thm:Lower_Bound}\,(ii) also does not hold for arbitrary subsets.  

Nevertheless, it may be the case that for reasonable random models of histories indexed by galled trees, violations of \cref{Thm:Lower_Bound} are relatively rare.  We provide some preliminary numerical evidence for this in \cref{Sec:Violations}, focusing on how often the number of intervals in $B_1(S)$ of a sample $S$ exceeds the number of recombinants in the underlying history.
\begin{figure}[h]
\begin{center}
\begin{tikzpicture}[thick,scale=1.3,->,
                   shorten >=2pt+0.5*\pgflinewidth,
                   shorten <=2pt+0.5*\pgflinewidth]

       \coordinate (a) at (0,0);
          \coordinate (aAbove) at (0,.1);
       \coordinate (b) at (-1,-1);
       \coordinate (bc) at (0,-1);
        \coordinate (bcBelow) at (0,-1.1);
       \coordinate (c) at (1,-1);
        \coordinate (d) at (-2,-2);
         \coordinate (de) at (-1,-2);
          \coordinate (deBelow) at (-.75,-2.1);
        \coordinate (e) at (0,-2);
         \coordinate (ef) at (1,-2);
         \coordinate (efBelow) at (.75,-2.1);
        \coordinate (f) at (2,-2);
        \coordinate (g) at (-3,-3);
         \coordinate (h) at (-1,-3);
          \coordinate (i) at (1,-3);
           \coordinate (j) at (3,-3);
            \coordinate (l) at (-2,-4);
             \coordinate (m) at (0,-4);
               \coordinate (mBelow) at (0,-4.1);
              \coordinate (n) at (2,-4);
               \coordinate (o) at (4,-4);
           \coordinate (q) at (-3,-5);
            \coordinate (r) at (-1,-5);
             \coordinate (s) at (1,-5);
              \coordinate (t) at (3,-5);
          
          \node[above] at (aAbove) {$\{\}$};          
          \node[above left] at (b) {$\{\au,\bu\}$};          
               \node[below] at (bcBelow) {$\{\cu\}$};                  
          \node[above right] at (c) {$\{\du,\eu\}$};  
            \node[above left] at (d) {$\{\au,\bu,\fu,\gu\}$};        
                \node[below] at (deBelow) {$\{\au,\bu,\hu,\iu\}$};     
                \node[below] at (efBelow) {$\{\du,\eu,\ju,\ku\}$};     
             \node[above right] at (f) {$\{\du,\eu,\lu,\muu\}$};        
                 \node[above left] at (g) {$\{\au,\bu,\fu,\gu,\nuu\}$};    
                       \node[above right] at (j) {$\{\du,\eu,\lu,\muu,\ou\}$};     
               \node[below] at (mBelow) {$\{\au,\fu,\du,\lu\}$};   
             
         
             \tikzset{
    every node/.style={
        circle,
        draw,
        fill          = gray!50,
        inner sep     = 0pt,
        minimum width =4 pt
    }   }

\path[draw] 

       
       node[fill=white] at (a) {} 
       node[fill=white] at  (b) {}
       node[fill=black] at  (bc) {}
       node[fill=white]  at  (c) {} 
       node[fill=white]  at  (d) {}
       node[fill=black] at  (de)  {}
        node[fill=black] at  (ef)  {}  
       node[fill=white] at  (f) {}
        node[fill=black] at  (g) {}
          node[fill=black] at  (j) {}
          node[fill=black] at  (m) {};
           

    \draw[color=black] (a) -- (b) ;
    \draw[color=black] (a) -- (c);
     \draw[color=black] (a) -- (bc);
    \draw[color=black] (b) -- (d);
    \draw[color=black] (b) -- (de);
    \draw[color=black] (c) -- (ef);
    \draw[color=black]  (c) -- (f);
    \draw[color=black]  (d) -- (g);
      \draw[color=black]  (f) -- (j);
        \draw[color=black]  (f) -- (m);
          \draw[color=black]  (d) -- (m);

\end{tikzpicture}
\end{center}
\caption{A subset $S$ of a history $\eh$ indexed by a galled tree with one recombinant, for which $|\B_1(S)|=2$ and $B_1(\eh)=\emptyset$.  Nodes corresponding to elements of $S$ are shown in black; the remaining nodes are shown in white.}
\label{Fig:Subsample_Counterexample}
\end{figure}
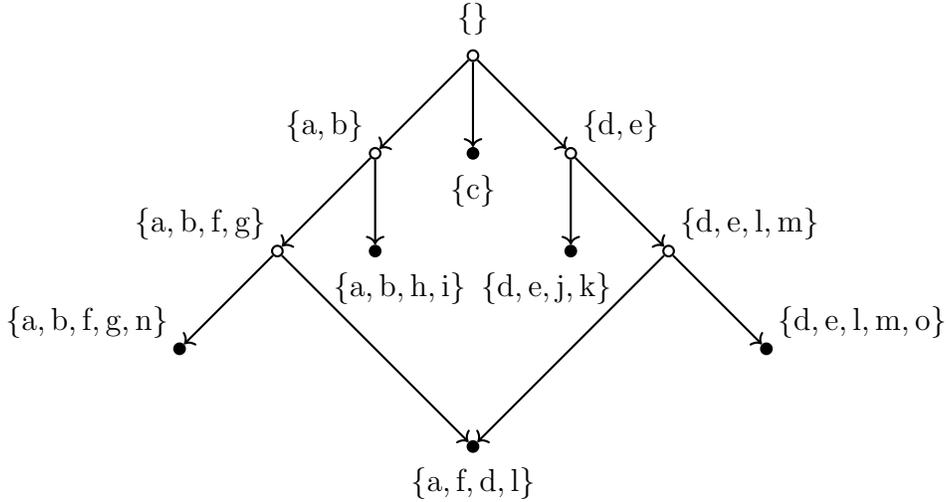
\end{example}

\section{Relaxing the Complete Sampling and Galled Tree Assumptions}\label{Sec:Extensions}
\cref{Thm:Lower_Bound}, the main result of the previous section, holds under the assumption that our evolutionary history is indexed by a galled tree, and that all organisms in the history have been sampled.  In this section, we apply the stability of persistent homology to extend the theorem to the case of an arbitrary (noisy) subsample of a history indexed by an arbitrary phylogenetic graph.

\subsection{Relaxing the Complete Sampling Assumption}\label{Sec:Relaxing_Complete_Sampling}
First, we extend \cref{Thm:Lower_Bound} to the case of a noisy subsample.  Given a list of non-negative numbers $L$, let $\Trim(L,\delta)$ be the list obtained by removing each of the numbers less than or equal to $\delta$ and subtracting $\delta$ from each of the remaining numbers.  

\begin{corollary}
\label{Cor:GalledTreeDetInference_withnoise}
Let $\eh$ be a history indexed by a galled tree and let $S$ be a finite metric space with $d_{GH}(\eh,S)=\delta$.  Then 
\begin{enumerate}[(i)]
\item $\textup{Trim}(\lens(\B_1(S)),2\delta)\leq \T(\eh).$
\item For $i\geq 2$, each interval of $\B_i(S)$ has length at most $2\delta$.
\end{enumerate}
\end{corollary}

\begin{proof}
This follows immediately from \cref{Thm:Lower_Bound,Thm:Stability}.
\end{proof}

\subsection{Relaxing the Galled Tree Assumption}\label{Sec:Relaxing_Galled_Tree}

As an application of \cref{Cor:GalledTreeDetInference_withnoise}\,(i), we next also relax the assumption that $\eh$ is indexed by a galled tree, yielding a further extension of \cref{Thm:Lower_Bound} which applies to any phylogenetic graph.  

For $G$ any phylogenetic graph and $\eh$ a history indexed by $G$, let \[\bigcup \eh:=\bigcup_{v\in V} \eh_v.\]  Thus, $\bigcup \eh$ is set of all mutations appearing in the history $\eh $. 

For $M\subseteq \bigcup\eh $ any subset and $v\in V$, let $\eh^M_v = \eh_v\setminus M$.  Let $G^M$ denote a subgraph of $G$ obtained by removing edges as follows:
Suppose $w$ is a recombinant of $G$ with parents $u$, $v$.  If $\eh^M_{w}=\eh^M_{u}\ne \eh^M_{v}$, we remove the edge $(v,w)$ from $G$.  If $\eh^M_{w}=\eh^M_{u}= \eh^M_{v}$ we remove exactly one of the edges $(u,w)$ and $(v,w)$, choosing arbitrarily.  It is easy to check that the sets $\eh^M_v$ then give a well-defined evolutionary history $\eh^M$ indexed by $G^M$.  

\begin{definition}
We let \[\Gall(\eh):=\min\,\left\{ |M|\ \middle|\  M\subseteq \bigcup \eh \textup{ such that } G^M\textup{ can be be chosen to be a galled tree} \right\}.\]
Informally, $\Gall(\eh)$ is the number of mutations in $\eh$ which must be ignored to obtain a history indexed by a galled tree by pruning edges in $G$.
\end{definition}

The following is our most general result relating barcodes and recombination:

\begin{corollary}
\label{Cor:Inference_Arbitrary_G}
Let $\eh$ be a history indexed by an arbitrary phylogenetic graph $G$, and let $S$ be a finite metric space with $d_{GH}(\eh,S)= \delta$.  Then
\begin{enumerate}[(i)]
\item $\textup{Trim}(\lens(\B_1(S)),3 \Gall(\eh)+2\delta)\leq \mathcal \T(\eh).$
\item For $i\geq 2$, each interval of $\B_i(S)$ has length at most $2(\Gall(\eh)+\delta)$.
\end{enumerate}
\end{corollary}

\begin{proof}
Choose $M\subseteq \bigcup \eh$ and a galled tree $G^M$ as above, such that $|M|=\Gall(\eh)$.  Note that $d_{GH}(\eh^M,\eh) \leq |M|=\Gall(\eh)$, so by the triangle inequality, $d_{GH}(\eh^M,S)\leq \Gall(\eh)+\delta$.  (ii) then follows from \cref{Cor:GalledTreeDetInference_withnoise}\,(ii).  By \cref{Cor:GalledTreeDetInference_withnoise}\,(i),  
\begin{equation}
\textup{Trim}(\lens(\B_1(S)),2(\Gall(\eh)+\delta))\leq \T(\eh^M).
\label{Eq:Relax_Galled_Assumption}
\end{equation}
Letting  $\bar\T(\eh^M)$ be the vector of length $| \T(\eh)|$ obtained by adding some 0's to the end of  $\T(\eh^M)$, we have by \cref{Prop:Topological_Novelty_Stability} that $d_\infty(\bar \T(\eh^M),\T(\eh))\leq |M|$.  Together with \eqref{Eq:Relax_Galled_Assumption}, this implies that 
\[\textup{Trim}(\lens(\B_1(S)),3\Gall(\eh)+2\delta)\leq \T(\eh),\]
which gives (i).
\end{proof}

\begin{remark}\label{Rmk:Interpretation_of_general_result}
Clearly, for \cref{Cor:Inference_Arbitrary_G} to yield a strong bound, $\Gall(\eh)$ must be small.   One might expect $\Gall(\eh)$ to be small but non-zero when recombination events typically affect short genome tracts (e.g., when they are gene conversion events).
\end{remark}

\section{Random Histories Indexed by Galled Trees}\label{Sec:Probability}
The results we have presented so far have been deterministic.  In \cref{Sec:Independence} below, we observe that in a wide class of probabilistic models of genetic sequence evolution on galled trees, the intervals of the first persistence barcode are independent random variables.  Thus, to understand the statistical properties of these barcodes, it suffices to understand the special case that the galled tree is a source-sink loop.  In \cref{Sec:Sensitivity_Numerical}, we study this special case numerically, for one choice of probabilistic model.

\subsection{Independence of Intervals in the First Barcode}\label{Sec:Independence}
In this section, we assume the reader is familiar with basic elements of the measure-theoretic formulation of probability theory \cite{durrett2010probability} and with the definition of conditional independence given a random variable \cite[Chapter 5]{kallenberg2006foundations}.

\bparagraph{Notation}
Suppose $X$, $Y$, and $Z$ are random variables on the same probability space.  If $X$ is independent of $Y$, we write $X\Perp Y$.  If $X$ is independent of $Y$ given $Z$, we write $X\Perp Y \mid Z$.

For $P$ a poset and $p\in P$, let 
\[\nde(p):=\{q\in P\mid p\not\leq q \}.\]
(Here, $\nde$ stands for \emph{non-descendants}.)  In what follows, $P$ will often be the vertex set of a directed acylic graph, with the partial order induced by the graph.

If $\eh$ is an evolutionary history indexed by $G$, $V$ is the vertex set of $G$, and $S\subseteq V$, we write $\eh_S:=\{\eh_v\mid v\in S\}$.  Similarly, for $v,w\in V$, let 
$\eh_{v\setminus w}:=\eh_v\setminus \eh_w$, and $\eh_{v\cap w}:=\eh_v\cap \eh_w$.  We will also use these notation conventions in combination with one another, so that e.g., $\eh_{r\setminus (p\cap q)}$ is understood to denote $\eh_r\setminus (\eh_p\cap \eh_q).$  

\begin{definition}[Random History]
For $G$ a fixed phylogenetic graph with vertices $V$, a \emph{random (evolutionary) history} $\eh$ indexed by $G$ consists of the following data:
\begin{itemize}
\item A probability space $\Omega$.
\item A countable set $X_v$ for each $v\in V$, such that each element of $X_v$ is itself a set.  We equip $X_v$ with the discrete $\sigma$-algebra.  
\item For each $v\in V$, a random variable $\eh_v:\Omega\to X_v$ such that for each $\omega\in \Omega$, $\{\eh_v(\omega)\}_{v\in V}$ is an evolutionary history.
\end{itemize}
\end{definition}

\begin{definition}[Locally Markov History]
Suppose that $\eh$ is a random history indexed by a phylogenetic graph $G$ with vertices $V$.  
$\eh$ is said to be \emph{locally Markov} if for each $v\in V$, 
\begin{equation}\label{Eq:Local_Markov}
\eh_v \Perp \eh_{\nde(v)}  \mid \{\eh_p\mid p \textup{ a parent of } v\}.
\end{equation}
\end{definition}

A locally Markov history is a special case of a \emph{Bayesian network}, a widely used probabilistic model \cite{lauritzen1996graphical}.

The assumption that a random history is locally Markov is quite natural; informally, this says that the genome of each organism depends only on the genomes of its parents.  However, the next example shows that for $\eh$ a locally Markov history, it is not necessarily the case that the intervals in $\B_1(\eh)$ are independent.

\begin{example}\label{Ex:Cycles_Not_Independent}
In the locally Markov history $\eh$ of \cref{Fig:Markov_Counterexample}, the mutations from the top source-sink loop are passed down to the bottom source-sink loop, where they serve as ``instructions" for how clonal mutations occur in the bottom loop.  Thus, the intervals in $\B_1(\eh)$ associated to the two recombinants are not independent.

For each vertex $v\ne w$, $\eh_v$ is completely determined by its parents.  The top recombinant corresponds to an interval $[\frac{1}{2},1)$ in the first barcode with probability $\frac{1}{2}$, and to an empty interval with probability $\frac{1}{2}$.  If the top recombinant corresponds to $[\frac{1}{2},1)$, then the bottom recombinant corresponds to $[1,2)$; otherwise, the bottom recombinant corresponds to $[\frac{1}{2},1)$.
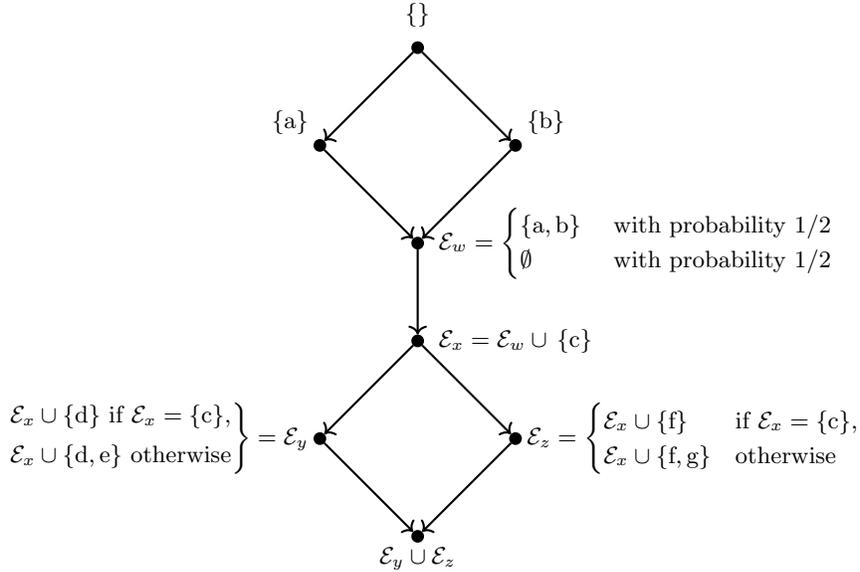
\begin{figure}[h]
\begin{center}
\begin{tikzpicture}[thick,scale=1.3,->,
                   shorten >=2pt+0.5*\pgflinewidth,
                   shorten <=2pt+0.5*\pgflinewidth]

       \coordinate (a) at (0,0);
          \coordinate (aAbove) at (0,.1);
       \coordinate (b) at (-1,-1);
        \coordinate (c) at (1,-1);
       \coordinate (d) at (0,-2);
        \coordinate (dRight) at (.1,-2);
       \coordinate (e) at (0,-3);
       \coordinate (eRight) at (.1,-3);
       \coordinate (f) at (-1,-4);
       \coordinate (g) at (1,-4);
        \coordinate (h) at (0,-5);
         \coordinate (below) at (-.1,-5);

          \node[above] at (aAbove) {\scriptsize$\{\}$};          
          \node[above left] at (b) {\scriptsize$\{\au\}$};                          
          \node[above right] at (c) {\scriptsize$\{\bu\}$};  
          \node[right] at (dRight) {\scriptsize $\eh_w=\begin{cases}\{\au,\bu\}&\textup{ with probability } 1/2\\ \emptyset&\textup{ with probability } 1/2\end{cases}$};  
  	 \node[right] at (eRight) {\scriptsize $\eh_x=\eh_w \cup\, \{\cu\}$};
	   \newenvironment{rcases}
  {\left.\begin{aligned}}
  {\end{aligned}\right\rbrace}
	     \node[left] at (f) {\scriptsize$\begin{rcases}&\eh_x\cup \{\du\}\textup{ if } \eh_x=\{\cu\},\\& \eh_x\cup \{\du,\eu\}\textup{ otherwise}\end{rcases}=\eh_y$};  
	      \node[right] at (g) {\scriptsize$\eh_z=\begin{cases}\eh_x\cup \{\fu\}&\textup{if } \eh_x=\{\cu\},\\ \eh_x\cup \{\fu,\gu\}&\textup{otherwise}\end{cases}$};  
	       \node[below] at (h) {\scriptsize $\eh_y\cup \eh_z$};
             \tikzset{
    every node/.style={
        circle,
        draw,
        fill          = gray!50,
        inner sep     = 0pt,
        minimum width =4 pt
    }   }

\path[draw] 

       
       node[fill=black] at (a) {} 
       node[fill=black] at  (b) {}
       node[fill=black] at  (c) {}
       node[fill=black]  at  (d) {} 
       node[fill=black]  at  (e) {}
       node[fill=black] at  (f) {}
        node[fill=black] at  (g) {}
        node[fill=black] at  (h) {};

    \draw[color=black] (a) -- (b) ;
    \draw[color=black] (a) -- (c);
     \draw[color=black] (b) -- (d);
     \draw[color=black] (c) -- (d);
     \draw[color=black] (d) -- (e);
      \draw[color=black] (e) -- (f);
       \draw[color=black] (e) -- (g);
       \draw[color=black] (f) -- (h);
       \draw[color=black] (g) -- (h);

\end{tikzpicture}
\end{center}
\caption{A locally Markov history $\eh$ for which the intervals in the $1^{\mathrm{st}}$ persistence barcode corresponding to the two recombinants are not independent.}
\label{Fig:Markov_Counterexample}
\end{figure}
\end{example}

Motivated by the above, we introduce the following subclass of locally Markov histories:

\begin{definition}[Phylogenetically Markov History]\label{Def:Phylo_Markov}
Suppose that $\eh$ is a random history indexed by a fixed phylogenetic graph $G$ with vertices $V$.  We say $\eh$ is \emph{phylogenetically Markov} if
\begin{enumerate}
\item[1.] $\eh$ is locally Markov.
\item[2.] $\eh_{v \setminus p} \Perp \eh_p$ for all clones $v\in V$ with parent $p$.
\item[3.] For all recombinants $r$ with parents $p$ and $q$, \[\eh_{r \setminus (p \cap q)} \Perp \eh_{\nde(r)} \mid \eh_{\{p\setminus (p\cap q),\,q\setminus (p\cap q)\}}.\]
\end{enumerate}

To parse condition 3, recall that a recombinant $r$'s genome necessarily inherits what is common to both parents $p$ and $q$; condition 3 states that the rest of $r$'s genome is independent of the genomes of all non-descendants of $r$, given the rest of the genomes of each parent.
\end{definition}

\begin{remark}  \cref{Def:Phylo_Markov} is slightly redundant, in the sense that condition 3 implies the local Markov property for each recombinant; this follows from \cref{Lem:Cond_Exp} below.  One might hope that one could obtain an equivalent definition by replacing condition 3 in \cref{Def:Phylo_Markov} with the simpler condition that $\eh_{r \setminus (p \cap q)} \Perp \eh_{p \cap q}$, but in fact this is strictly weaker.  It can be shown that our independence result (\cref{Prop:Markov_Independent} below) does not hold for this weaker condition.

%
%
%
%
%
%
%

\end{remark}

\begin{example}\label{Ex:Poisson}
Assume that $G$ is endowed with a time function $t:V\to \R$, as defined in \cref{Sec:NoveltyProfiles}.  We specify (up to choice of labels for mutations) a phylogenetically Markov history $\eh$, the \emph{Poisson history} indexed by $G$:
\begin{itemize}
\item $\eh_{r}=\emptyset$, for $r$ the root of $G$.
\item If $w$ is a clone with ancestor $v$, $|\eh_w\setminus \eh_v|$ is Poisson distributed with parameter $t(w)-t(v)$.
\item If $w$ is a recombinant with parents $u$ and $v$, then for each $m \in \eh_u \setminus \eh_v$, $P(m \in \eh_w)=p_w$, and for each $m \in \eh_v \setminus \eh_u$, $P(m \in \eh_w)=1-p_w$.  Here, we may either take $p_w=1/2$ for all $w$, or take the $p_w$ to be i.i.d. random variables with the uniform distribution on $[0,1]$.
\end{itemize}
\end{example}

Let \[\mathcal I:= \left\{[a,b)\mid a<b\in \{0,1,2,\ldots\}\right\} \cup \{\emptyset\}.\]  Thus, $\mathcal I$ is a collection of intervals with integer endpoints, together with the empty interval.
For $\eh$ a history indexed by a galled tree $G$ and $r$ a recombinant in $G$, let $\I^\eh(r)\in \I$ denote the unique interval in $\B_1(\eh)$ corresponding to $r$, if such an interval exists (see \cref{Thm:Lower_Bound}), and let $\I^\eh(r)=\emptyset$  otherwise.   As in \cref{Sec:Inference}, we let $\mathcal R^G$ denote the set of recombinants of $G$.

Here is the main result of this section:

\begin{theorem}\label{Prop:Markov_Independent}
For $\eh$ a phylogenetically Markov history indexed by a galled tree $G$, the random variables $\{\I^\eh(r)\}_{r\in \mathcal R^G}$ are independent.
\end{theorem}

The proof of the theorem will use several standard facts about conditional independence, which we record in the following lemma.
\begin{lemma}\label{Lem:Cond_Exp}
Assume $h$ is a measurable function whose domain is the codomain of the random variable $X$.
\begin{enumerate}[(i)]
\item If $X\Perp Y\mid Z$, then $Y\Perp X \mid Z$.
\item If $X\Perp Y\mid Z$, then $h(X)\Perp Y\mid Z$.
\item If $X\Perp Y\mid Z$, then $X\Perp Y\mid (Z,h(X))$.
\item  If $X\Perp Y\mid Z$ and $W\Perp Y\mid (X,Z)$, then $(W,X) \Perp Y\mid Z$.
\item If $X\Perp Y\mid Z$, then $(Z,X) \Perp Y\mid Z$.
\end{enumerate}
\end{lemma}
Note that by taking $Z$ to be the identity random variable, we also obtain unconditional versions of (i)-(iv) above.

\begin{proof}
Properties (i)-(iv) appear in many places; see e.g. \cite[Chapter 3]{lauritzen1996graphical}.  We prove (v).  Taking $W=Z$ in (iv), it suffices to show that $Z\Perp Y \mid (X,Z)$.  By (ii), for this it is enough to show that $(X,Z) \Perp Y\mid (X,Z)$.  But it is easy to check that in general, $A\Perp B \mid A$.
\end{proof}  

\begin{lemma}\label{Lem:Posets_And_Independence}
If $\{X_a\}_{a\in P}$ is a collection of random variables indexed by a finite poset $P$ and $X_a \Perp X_{\nde(a)}$ for each $a\in P$, then the $\{X_a\}_{a\in P}$ are independent.
\end{lemma}

\begin{proof}
Choose a total order compatible with the partial order on $P$, and relabel the random variables with respect to this order as $X_1,\ldots,X_{|P|}$.  We show by induction that
$X_1,\ldots, X_m$ are independent  for each $m\in \{1,\ldots, |P|\}$.  The base case is trivial.  Now suppose $X_1,\ldots, X_{m-1}$ are independent.  The elements of $P$ corresponding to the indices $1,\ldots, m-1$ are in $\nde(m)$, so $X_m$ is independent of $X_1,\ldots, X_{m-1}$.  By this and the induction hypothesis, $X_1,\ldots, X_m$ are independent.
\end{proof}

\begin{proof}[Proof of \cref{Prop:Markov_Independent}]
Order the vertices of $G$ arbitrarily.  For $r\in \mathcal R^G$, let $D^r$ denote the distance matrix obtained by restricting $\eh$ to the source-sink loop $L^r$ of $G$ with sink $r$.  The images of independent random variables under measurable functions remain independent, so in view of the results of \cref{Sec:Barcodes_Of_Histories_Indexed_By_Galled_Trees}, it suffices to show that the $\{D^r\}_{r\in \mathcal R^G}$ are independent.  Now for each recombinant $r$, let $V^r$ denote the vertices of $L^r$, and let $q_r$ denote the unique source of $L^r$.  Since $\eh_{q_r}\subseteq \eh_v$ for all $v\in V^r$, clearly $D^r$ is determined by $A^r:=\{\eh_{v\setminus q_r}\}_{v\in V^r\setminus \{q_r\}}$.  Therefore, it in fact suffices to show that the random sets $\{A^r\}_{r\in \mathcal R^G}$ are independent.  

We define a partial order on $\mathcal R^G$ by writing $r\leq r'$ if for some $v\in V^r\setminus \{q_r\}$, there is a directed path from $v$ to $q_r'$ in $G$; it is easy to check that this is in fact a partial order.  This partial order induces a partial order on $\{A^r\}_{r\in \mathcal R^G}$.   We establish the independence of the $\{A^r\}_{r\in \mathcal R^G}$ by applying \cref{Lem:Posets_And_Independence}, using this partial order.  Let \[\nde(L^r):=\bigcap_{v\in V^r\setminus \{q_r\}} \nde(v).\]  If $r\not \leq r'$ then $V^{r'}\subseteq \nde(L^r)$, so $\{A^{r'} \mid A^r \not \leq A^{r'}\}$ is completely determined by $\eh_{\nde(L^r)}$.  Thus, it suffices to show that $A^r \Perp \eh_{\nde(L^r)}$ for each $r\in \mathcal R^G$.   

Let us fix $r\in \mathcal R^G$ and write $q=q^r$.  Choose a total order on $V^r\setminus \{r,q\}$ compatible with the partial order on $V$, and write the elements in increasing order as $\{c_1,\ldots,c_m\}$.  For $j\in \{1,\ldots,m\}$, let $B_j:=\{\eh_{c_i\setminus q}\}_{1\leq i\leq j}\subset A^r$.  We show by induction on $j$ that $B_j \Perp \eh_{\nde(L^r)}$ for each $j$.  

First, consider the base case $j=1$.   In the remainder of the proof, the five statements of \cref{Lem:Cond_Exp} will be denoted simply as (i)-(v).  By the definition of a phylogenetically Markov history, we have 
$\eh_{c_1} \Perp \eh_{\nde(c_1)} \mid \eh_q$, so by (ii), we have $\eh_{c_1} \Perp \eh_{\nde(L^r)} \mid \eh_q$.  By (v) then, $\eh_{\{c_1,q\}} \Perp \eh_{\nde(L^r)} \mid \eh_q$, so by (ii), $\eh_{c_1\setminus q} \Perp \eh_{\nde(L^r)} \mid \eh_q$.  The definition of a phylogenetically Markov history also gives that $\eh_{c_1 \setminus q} \Perp \eh_q$.  Applying (iv) and (ii), we find that $\eh_{c_1\setminus q} \Perp \eh_{\nde(L^r)}$.  This shows that $B_1 \Perp \eh_{\nde(L^r)}$.

The induction step  is similar to the above.  Let $p$ denote the parent of $c_j$.  $\eh_{c_j} \Perp \eh_{\nde(c_j)} \mid \eh_p$, so $\eh_{c_j\setminus p} \Perp \eh_{\nde(c_j)} \mid \eh_p$.  Moreover, $\eh_{c_j \setminus p} \Perp \eh_p$, so $\eh_{c_j\setminus p} \Perp \eh_{\nde(c_j)}$.  Then by (ii) and (iii), $\eh_{c_j\setminus p} \Perp \eh_{\nde(L^r)} \mid B_{j-1}$.  By (iv) and the induction hypothesis, we thus have that $B_{j} \Perp \eh_{\nde(L^r)}$, as desired.

Finally, we show that $A^r \Perp \eh_{\nde(L^r)}$.  Let $p_1$ and $p_2$ denote the parents of $r$.  By the third condition in the definition of a phylogenetically Markov history, we have 
$\eh_{r\setminus q} \Perp \eh_{\nde(r)}\mid \eh_{\{p_1\setminus q,\ p_2\setminus q\}}$.  By (iii), $\eh_{r\setminus q} \Perp \eh_{\nde(r)}\mid (\eh_{\{p_1\setminus q,\ p_2\setminus q\}},B_m)$.  $\sigma(\eh_{\{p_1\setminus q,\ p_2\setminus q\}},B_m)=\sigma(B_m)$ since $\eh_{\{p_1\setminus q,\ p_2\setminus q\}}= h(B_m)$ for some measurable function $h$, so $\eh_{r\setminus q} \Perp \eh_{\nde(r)}\mid B_m$ by the definition of conditional independence.  By (ii), $\eh_{r\setminus q} \Perp \eh_{\nde(L^r)}\mid B_m$.  We have also shown that $B_m \Perp \eh_{\nde(L^r)}$, so by (iv), $A^r \Perp \eh_{\nde(L^r)}$.
\end{proof}

\subsection{The Barcode of a Random History on a Source-Sink Loop: Numerical Results}\label{Sec:Sensitivity_Numerical}
\cref{Prop:Markov_Independent} tells us that for a phylogenetically Markov history indexed by a galled tree, to understand the distribution of the $1^{\mathrm{st}}$ barcode, it suffices to understand this for each source-sink loop in the galled tree.  Working with a simple random model of a history $\eh$ indexed by a source-sink loop, we now study the distribution of $\B_1(\eh)$ numerically.  Recall that by \cref{Thm:ALMS}\,(ii), $\B_1(\eh)$ has at most one interval.  
We consider here the probability that $\B_1(\eh)$ is nontrivial, as well as the average length of the interval.  

In our simulations, we find that in the limit of high novelty, persistent homology captures between 14\% and 35\% of recombination events, the exact value depending on mutational parameters. In typical simulations where a recombination event is detected, the bar length is well below the theoretical maximum provided by \cref{Thm:Lower_Bound}\,(i), scaling roughly as the square root of novelty.

\bparagraph{Details of the Computations}
We now specify the random model of a history indexed by a source-sink loop that we use in our simulations.  The model depends on parameters $m$ and $k$.  Each random history generated by this model consists of: a left parent with no mutations; a right parent with mutations $\{1, \dots, m\}$; a recombinant with some subset of these mutations; and $k$ ``intermediate sequences,'' each randomly sampled (with replacement) from the set \[\left\{\{1\}, \{1,2\}, \{1,2,3\}, \dots, \{1, \dots, m-1\} \right\}.\] In our simulations, we consider values of $m$ between 2 and 200, and values of $k$ between $1$ and $50$. In addition, we consider a ``maximal sampling'' scenario, in which all possible intermediate sequences were included in the sample; for the purpose of visualization, this scenario is assigned parameter value $k=51$.  To construct the recombinant, we select each of the $m$ mutations with probability $\alpha$.  In one set of our simulations, we set $\alpha=0.5$ (simulating a recombination breakpoint at the midpoint of the genome); in a second set of simulations, we choose $\alpha$  randomly from the uniform distribution on $[0,1]$. For each sampled history, we compute both the novelty of the recombinant and the persistent homology of the sample.

\bparagraph{Results}  The results of our simulations are given in figure \cref{Fig:simulations}.  To obtain each subfigure, we aggregated the data for the various values of the parameter $m$.  We see that the rate of detection of a recombinant increases with novelty, up to about 37\% (midpoint recombination breakpoint, \cref{Fig:simulations}\,(a)) or 20\% (uniform recombination breakpoint, \cref{Fig:simulations}\,(b)). For high novelty, increasing $k$ improves detection only up to about $k=7$, after which detection falls to about 28\% (midpoint recombination breakpoint) or 16\% (uniform recombination breakpoint) for high $k$.

Bar length typically falls well below the upper bound given by the novelty of the recombinant. In particular, for simulations where the recombination event was detected, bar length scales roughly as the square root of novelty (\cref{Fig:simulations}\,(c,d)). For cases with high novelty, median bar length ranges from about 25\% of the square root of novelty (if many intermediate sequences are sampled, upper right corner of \cref{Fig:simulations}\,(e,f)) to 35\% of the square root of novelty (if few intermediate sequences are sampled, upper left corner of \cref{Fig:simulations}\,(e,f)).

\begin{figure}
\centering
\begin{tabular}[c]{cc}
	\begin{subfigure}{0.42\textwidth}
		\includegraphics[width=1\linewidth]{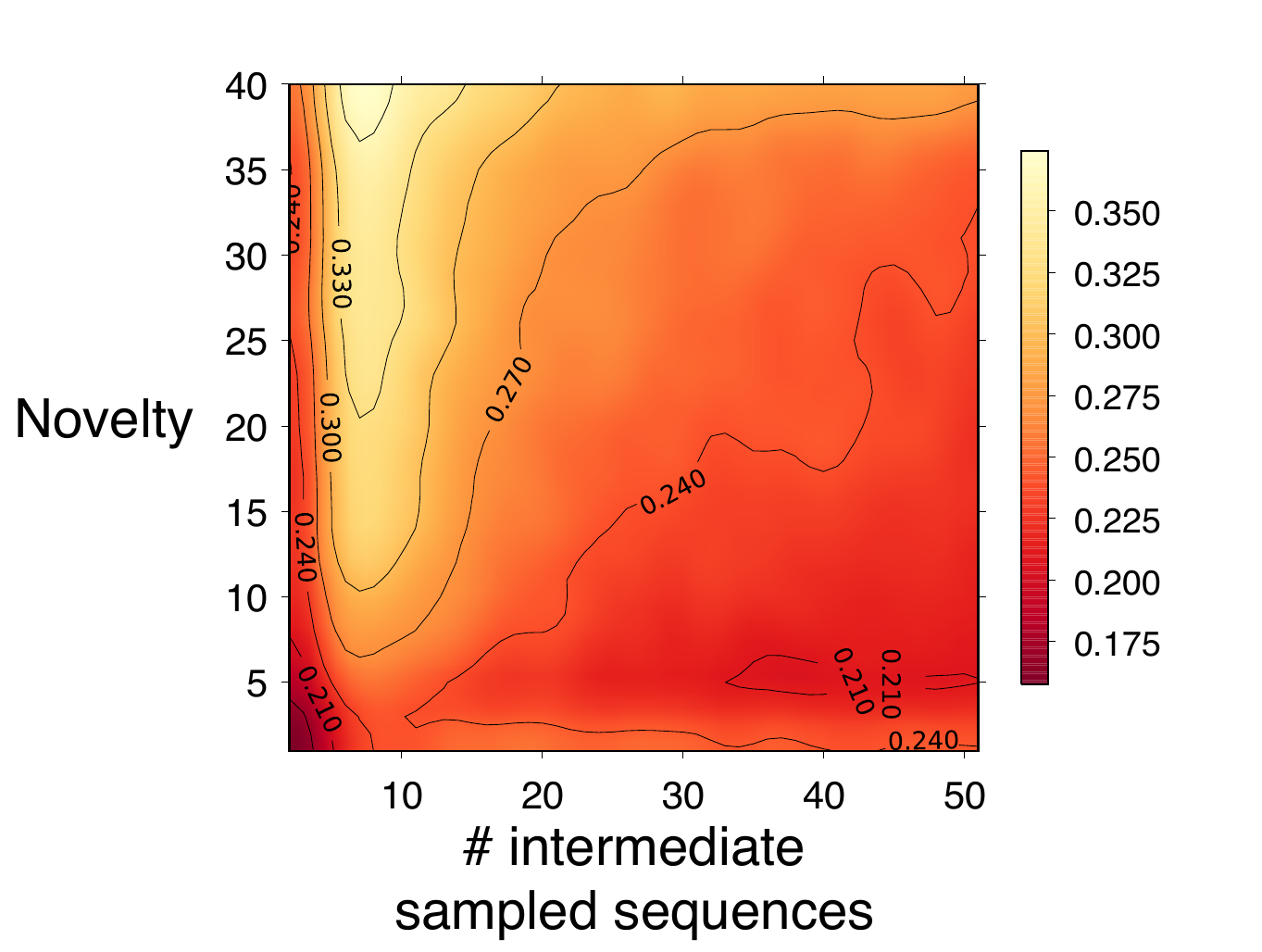}
	\end{subfigure} &
	\begin{subfigure}{0.42\textwidth}
		\includegraphics[width=1\linewidth]{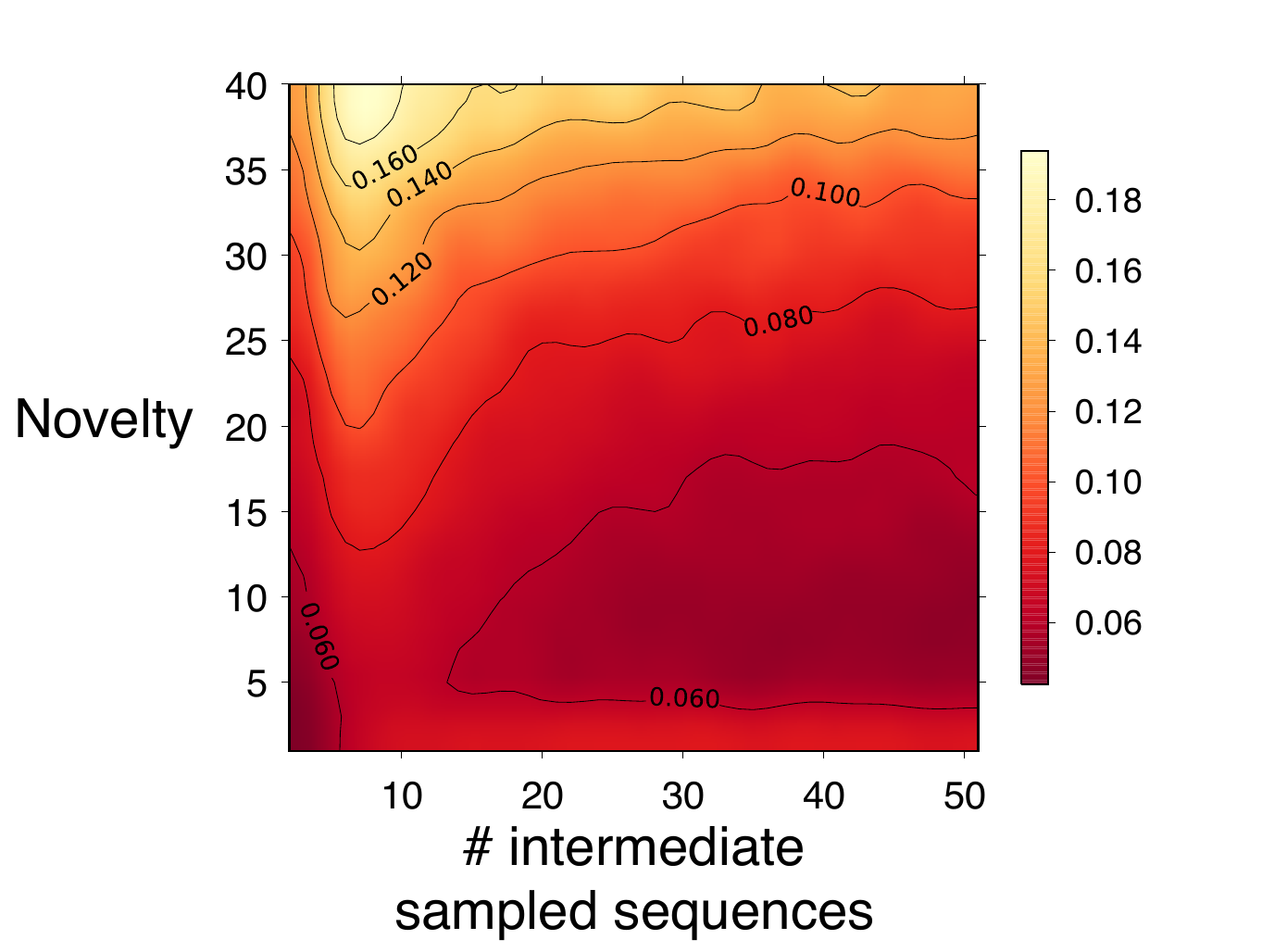}
	\end{subfigure} \\
	(a) & (b) \\
	
	\begin{subfigure}{0.42\textwidth}
		\includegraphics[width=1\linewidth]{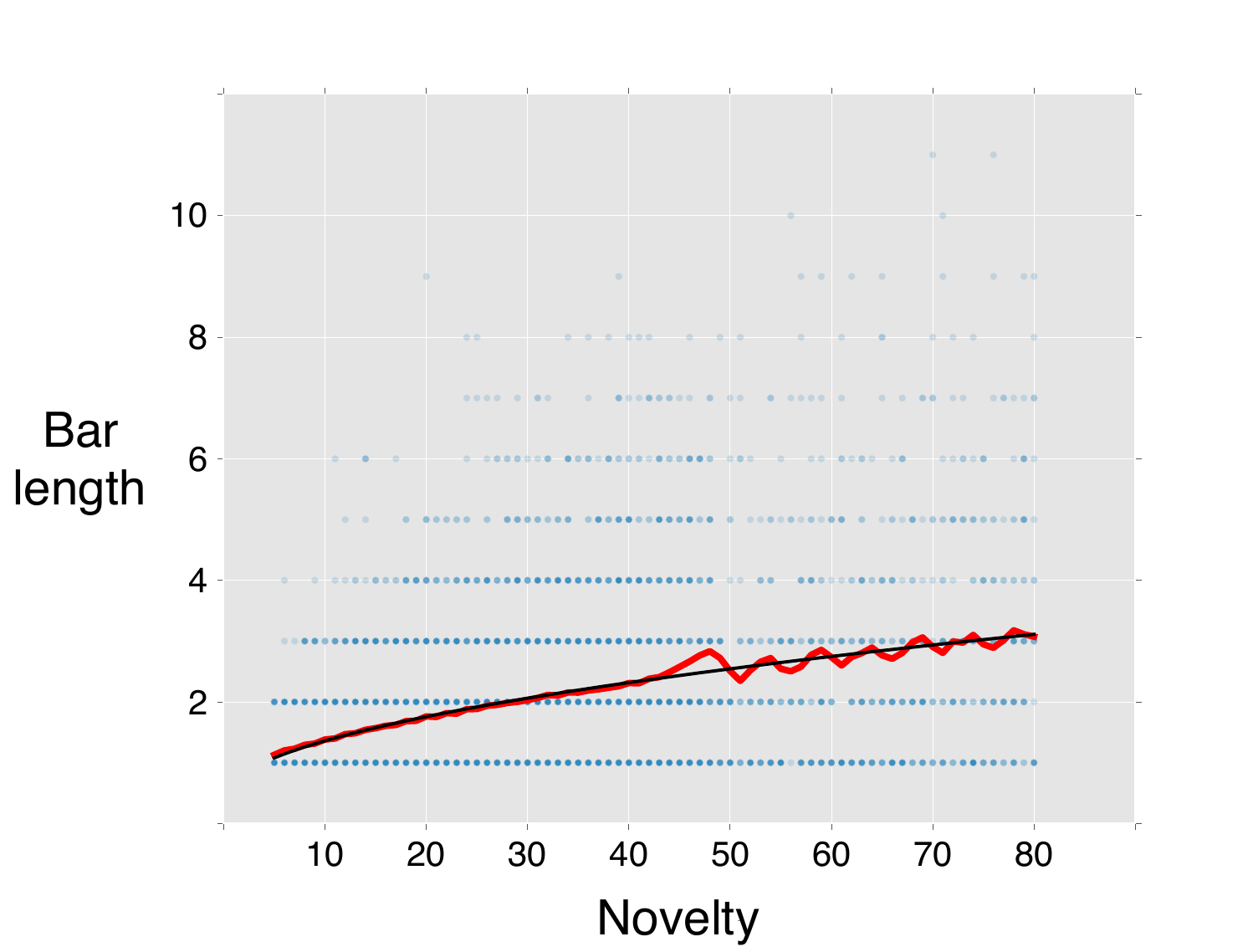}
	\end{subfigure} &
	\begin{subfigure}{0.42\textwidth}
		\includegraphics[width=1\linewidth]{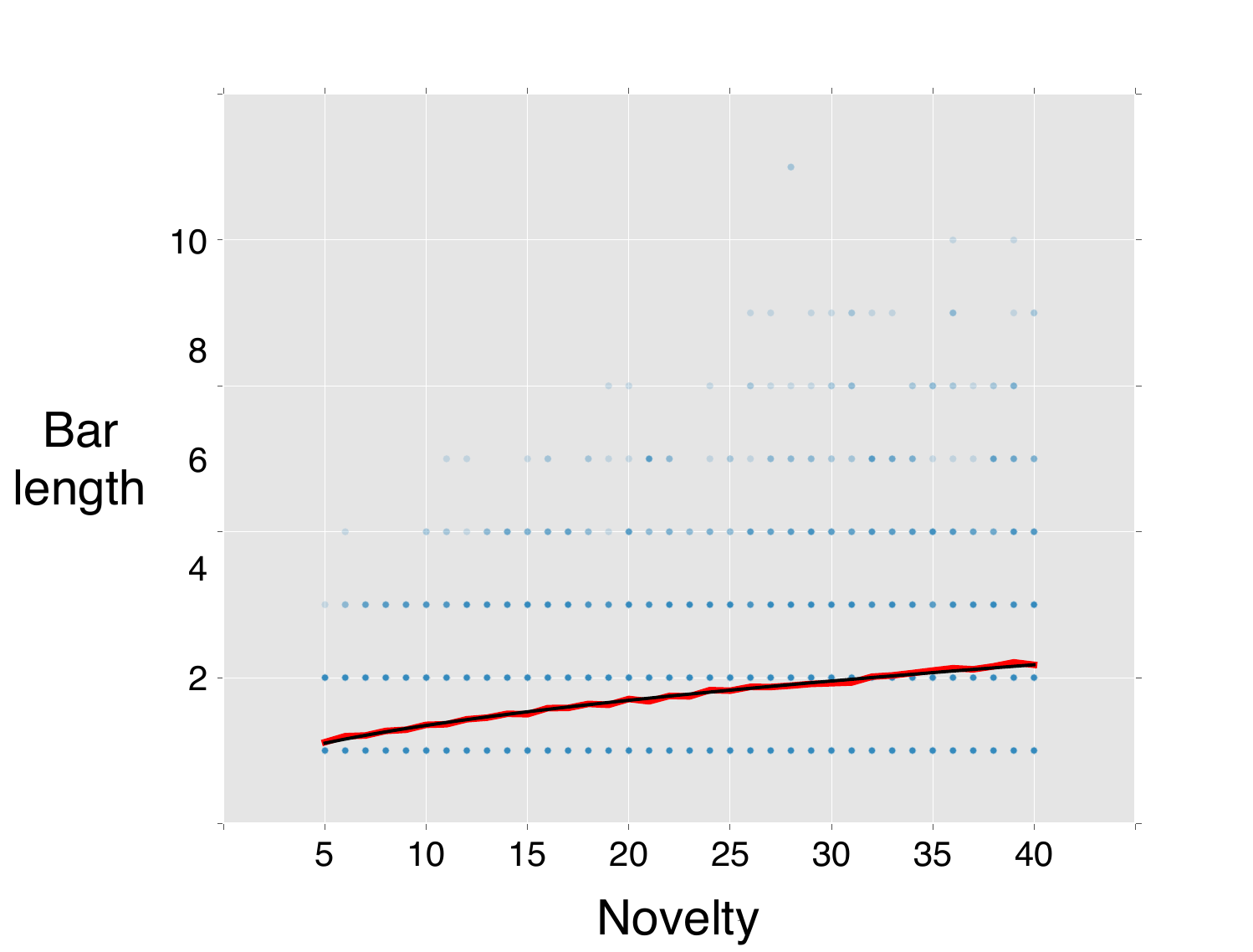}
	\end{subfigure} \\
	(c) & (d) \\
	
	\begin{subfigure}{0.42\textwidth}
		\includegraphics[width=1\linewidth]{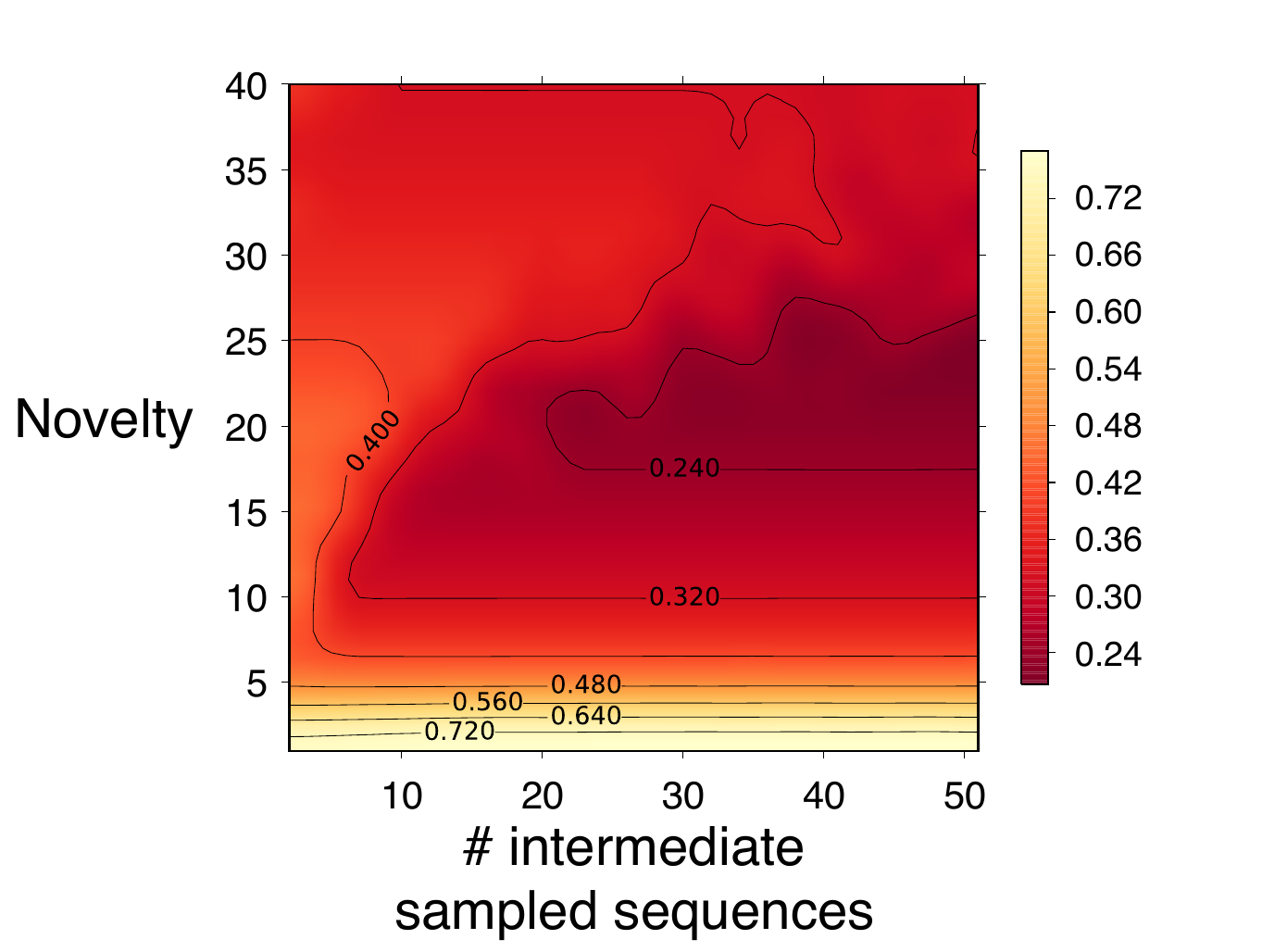}
	\end{subfigure} &
	\begin{subfigure}{0.42\textwidth}
		\includegraphics[width=1\linewidth]{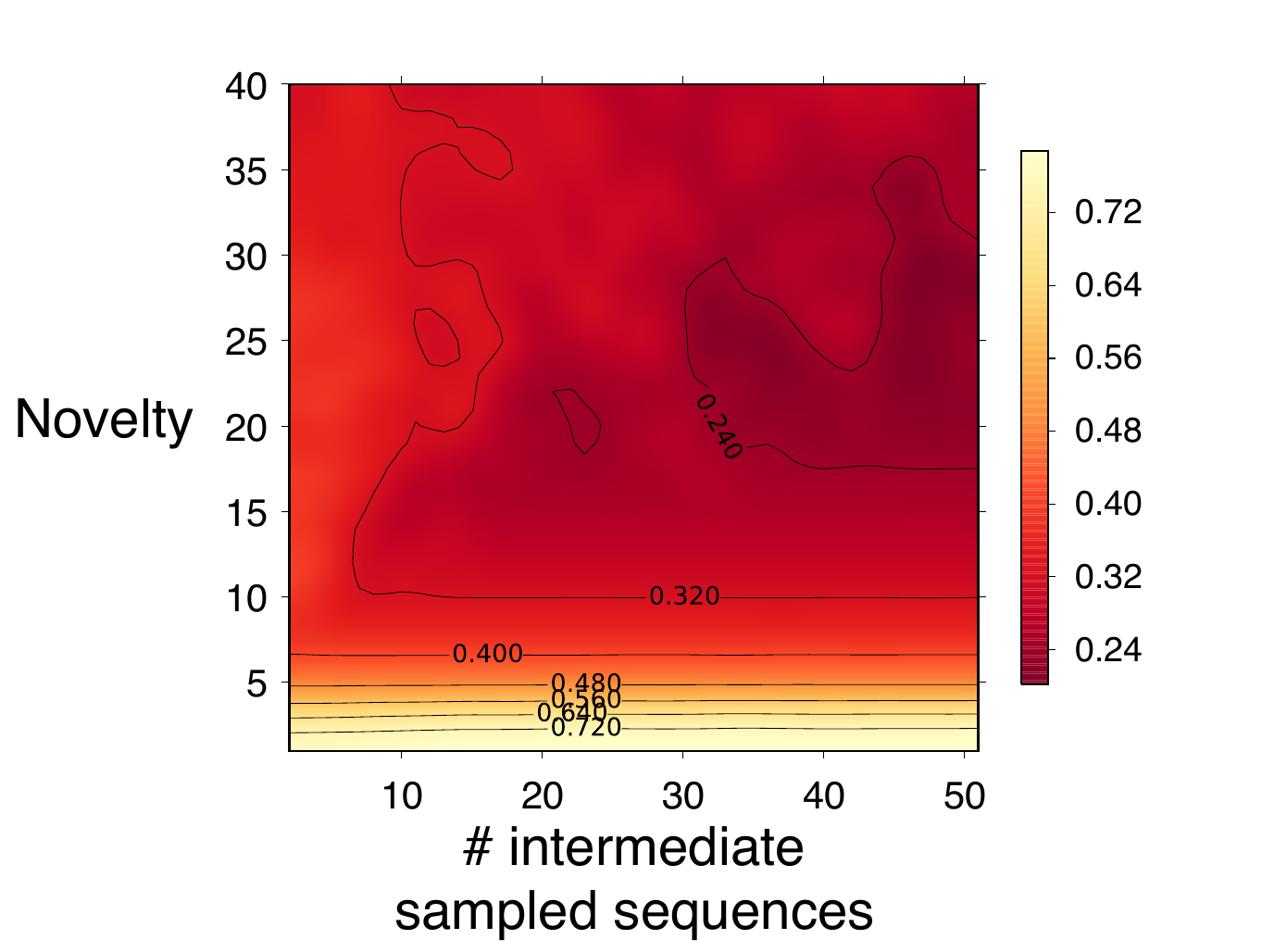}
	\end{subfigure} \\
	(e) & (f)
\end{tabular}
\caption{\scriptsize{Sensitivity of persistent homology in simulations of single recombination events. Panels (a,b): Fraction of simulations in which the recombination event was detected. Recombination breakpoint is either at the midpoint of (panel a) or uniformly distributed along (panel b) the genome. Panels (c,d): Each translucent point marks the result of a single simulation in which the recombination event was detected, the red line tracks the average bar length among all simulations with equal novelty, and the black line shows the least-squares fit of parameters $a$ and $b$ among functions $y = a \sqrt{x} + b$. Panel (c): Recombination breakpoint at midpoint of the genome. The fit of all 684{,}026 cases is $y = 0.30 \times \sqrt{x} + 0.40$. Panel (d): Recombination breakpoint uniformly distributed along the genome. The fit of all 134{,}830 cases is $y = 0.26 \times \sqrt{x} + 0.51$. Panels (e,f): Median of ratio of bar length to square root of novelty, conditional on the recombination event being detected (i.e., bar length $\geq 1$). Recombination breakpoint is either at the midpoint of (panel e) or uniformly distributed along (panel f) the genome.}}
\label{Fig:simulations}
\end{figure}

\section{Discussion}\label{Sec:Discussion}

In this paper, we have introduced \emph{novelty profiles}, simple statistics of an evolutionary history which not only count the number of recombination events in the history, but also quantify the contribution recombination makes to genetic diversity.  We have studied the problem of inferring information about a novelty profile from the persistent homology of sampled data, and have shown that under certain conditions, persistence barcodes of genomic data can be interpreted as lower bounds on novelty profiles.  Our results provide mathematical foundations for several earlier works which have used persistent homology to study recombination.    

\subsection{Potential Applications of the Novelty Profile}

Understanding the precise mechanisms by which recombination contributes to evolution is a long-standing problem in evolutionary biology.  Progress on this problem depends critically on the availability of suitable quantitative descriptors of recombination.  As a measure of the contribution of recombination to genetic diversity, the novelty profile captures information about recombination not captured by standard measures such as the recombination rate.  This information may be helpful for understanding how recombination drives evolution.

We describe one class of potential applications in this direction.  It is well known that reticulate evolution plays an important role both in the spread of infection and in the development of drug resistance.  Examples include the emergence of a norovirus pandemic \cite{eden2013recombination}, outbreaks of influenza (e.g., the Swine flu pandemic of 2009) \cite{ito1998molecular,trifonov2009geographic,solovyov2009cluster}, the emergence of resistance to anti-viral medication in HIV \cite{nora2007contribution}, and the spread of antibiotic resistance in bacteria (e.g., \textit{E. coli}) \cite{davies2010origins,rohde2011open}.

The novelty profile could be useful for developing a fuller quantitative understanding of the role of recombination in such epidemiological events.  We hypothesize that, compared to a count of recombination events alone, the novelty profile of a pathogen better predicts both future outbreaks of infection and the proliferation of drug resistance.   It could be interesting to test this hypothesis in simulation.  To test the hypothesis on real biological data, one needs well-behaved estimators of (statistics of) the novelty profile; our main bounds represent progress in this direction, and in \cref{Sec:Alternative_Strategies} below, we discuss alternative estimation approaches which may be more practical.  If the novelty profile is indeed predictive of outbreaks or of the proliferation of drug resistance, statistics derived from estimates of novelty profiles could potentially inform public health responses to infectious disease.

\subsection{On the Assumptions Underlying our Main Results}\label{Sec:Assumptions}
Our main results relating barcodes to novelty profiles depend on strong assumptions about the evolving population and genomic sampling.  In their simplest form, our results assume that the evolutionary history $\eh$ is indexed by a galled tree and that all genomes in the history are included in our sample.  Using the stability of persistent homology, we have extended these results to hold for an arbitrary sample $S$ of an arbitrary evolutionary history $\eh$.  The strength of the bounds provided by these extended results, relative to the ideal case of a galled tree with every genome sampled, is controlled by $d_{GH}(S,\eh)$, the Gromov-Hausdorff distance between $S$ and $\eh$, and $\Gall(\eh)$, the number of mutations in $\eh$ which must be ignored to obtain a history indexed by a galled tree.  In cases where $d_{GH}(S,\eh)$ and $\Gall(\eh)$ can be assumed to be small relative to the lengths of intervals in the barcodes $\B_i(S)$, our results provide an informative lower bound on the novelty profile, though the numerical results of \cref{Sec:Sensitivity_Numerical} suggest that this bound is typically far from tight.
 
These results raise three key questions about applications of our work: First, under what circumstances can real-world genomic samples be expected to exhibit small enough values of $d_{GH}(S,\eh)$ and $\Gall(\eh)$ for our bounds on the novelty profile to be useful?  Second, can the theory in this paper be extended to yield a useful topological bound on the novelty profile even when $d_{GH}(S,\eh)$ are $\Gall(\eh)$ not necessarily small? And third, can more sensitive bounds on the novelty profile be obtained? We discuss each of these questions below.

\bparagraph{The Small $\Gall(\eh)$ Condition}

While restrictive, the assumption that $\Gall(\eh)$ is small is biologically plausible in some settings.  As shown in \cref{Sec:Appendix} in the context of the coalescent model, the phylogenetic graph $G$ indexing a history $\eh$ will be a galled tree with high probability if and only if a relatively strong condition on the rareness of recombination is satisfied. It is important to note that this condition depends not only on the species studied, but also the sample size and how the genomic data is analyzed.  For example, in the study of human recombination, a key methodological choice is the size of the genomic window used for the analysis; while a larger window offers more accurate estimates of recombination rate, a smaller window better localizes recombination breakpoints \cite{Camara:2016eh}. Since recombination is also rarer in a smaller window, a sample's ancestry is more likely to be represented by a galled tree when using a smaller window. In a sample of 125 humans, for instance, the ancestry in an average 275 bp window is predicted to be a galled tree with 90\% probability (\cref{Sec:Appendix}). As recombination rates vary dramatically across segments of the human genome, the actual probability likewise varies. 

\bparagraph{The Small Gromov-Hausdorff Distance Condition}

For a typical genomic sample $S$, $d_{GH}(S,\eh)$ can be large.  Indeed, regardless of whether individuals are sampled simultaneously or longitudinally, the most recent common ancestor of two individuals in $S$ may be genetically distant from all individuals in $S$.

There are applications, however, where sampling is so dense, and so frequent, that we do expect common ancestors of sampled individuals to be genetically close to individuals in the sample. One such application is dense epidemiological sampling of HIV, for which entire countries have established long-term viral genomic surveillance \cite{Bezemer:2010hu, theSwissHIVCohortStudy:2010up}. In such cases, standard phylogentic methods, which assume that ancestors are absent from the sample, may produce misleading results \cite{Stadler:2011hn}. As genetic sequencing continues to decline in cost, it is likely for dense longitudinal genomic samples to become more common. For such data sets, particularly for pathogens where evolutionary time scales are short compared to sampling duration, the assumption that $d_{GH}(S,\eh)$ is small may be more reasonable.

Even for samples $S$ for which $d_{GH}(S,\eh)$ is large, our simulation results using the coalescent model in \cref{Sec:Violations} suggest that violations of the exact bound on the number of recombinations given by \cref{Thm:Lower_Bound}\,(i) are relatively rare, in part because of the limited sensitivity of persistence barcodes in detecting recombination (\cref{Sec:Sensitivity_Numerical}).   This, together with the empirical results from extensive simulations described in previous literature \cite{chan2013topology, Emmett:2014us, Camara:2016eh,Camara:2016kl}, give us hope that our main theoretical results may be extended to  probabilistic ones that yield useful bounds even when $d_{GH}(S,\eh)$ is large.

\subsection{Directions for Further Theoretical Work}
\mbox{}

\bparagraph{Extending our Results to More Complex Phylogenetic Graphs} It may be possible to extend our bounds to histories indexed by iterated sums of phylogenetic graphs with at most $k$ recombination events, at least for small $k$. (The galled tree setting is the case $k=1$.)  Such an extension would yield lower bounds on the novelty profile for a larger class of evolutionary histories.
The next logical step would be to study the case $k=2$.  In analogy with our main results, several questions arise about a history $\eh$ indexed by a phylogenetic graph with two recombinants:
\begin{itemize}
\item Does $B_1(\eh)$ have at most two bars?
\item For which degrees $i$ is $B_i(\eh)$ necessarily trivial?   
\item What can be said about the lengths of the intervals of $\B_i(\eh)$?
\end{itemize}

\bparagraph{Tightening our Lower Bounds on the Novelty Profile}

As shown numerically in \cref{Sec:Sensitivity_Numerical}, one limitation of barcodes as lower bounds on novelty profiles is their relatively low sensitivity to individual recombination events. A natural goal is to devise a more sensitive variant of our bounds, with similar theoretical guarantees.  Is it possible to develop a \emph{consistent} barcode estimator for the novelty profile, for a reasonable class of probabilistic models?  

As a step in this direction, it would be interesting to apply our independence result, \cref{Prop:Markov_Independent}, to obtain analytic results about the probability distribution on $\B_1(\eh)$, for $\eh$ a phylogenetically Markov random evolutionary history (e.g., Poisson) indexed by a galled tree.  Ideally, such results would explain the relationships between novelty profiles and barcodes observed empirically in \cref{Sec:Sensitivity_Numerical}.

\subsection{Practical Strategies for Estimating (Statistics of) Novelty Profiles}\label{Sec:Alternative_Strategies}
While this paper has focused primarily on bounding the novelty profile in terms of standard persistence barcodes, other approaches to estimation of the novelty profile may be more practical.  
There is a large literature on direct estimation of evolutionary histories (which, as noted earlier, are usually called \emph{ancestral recombination graphs}); see for example \cite{gusfield2014recombinatorics} and the references therein.  Though direct inference of histories is computationally difficult on larger data sets, recent approaches such as ARGWeaver \cite{Rasmussen:2014cq} are powerful enough to yield biological insights from some real data sets consisting of dozens or hundreds of genomes.  For such data sets, it may be feasible to use direct inference of a history to estimate the novelty profile.  Indeed, once one has the history, computing the novelty profile is straightforward.

For larger data sets, where estimation of a full evolutionary history is not feasible, a machine learning approach may be effective.  For context, a recent paper of Humphreys, McGuirl, Miyagi, and Blumberg \cite{humphreys2019fast} trains a regression model to predict recombination rates from Vietoris--Rips barcodes of genomic data.  The training data is obtained from simulations.  The authors demonstrate that this approach performs well, offering a good tradeoff between accuracy and scalability compared to a popular alternative approach called LDhelmet \cite{chan2012genome}.  To estimate statistics of the novelty profile, it may be worthwhile to develop an approach analogous to that of \cite{humphreys2019fast}.
\appendix

\section{Probability that the Coalescent with Recombination Generates a Galled Tree}
\label{Sec:Appendix}

\subsection{Overview of the Coalescent with Recombination}
The coalescent with recombination is a commonly used model of the evolutionary process generating a population genetic sample. For a detailed introduction to the coalescent with recombination, see  \cite{Wakeley:2007ua}.  Here, we give only a brief,  informal description.

	Instead of tracking an entire population, which may include millions or billions of reproducing organisms, a coalescent model tracks only the sampled individuals and their direct ancestors, up to their most recent common ancestor.  We can think of the coalescent with recombination as a dynamical model that generates a phylogenetic graph, together with a time function on it, by proceeding backward in time.  We start with an initial set of $n$ vertices at some fixed final time, corresponding to $n$ distinct lineages.  As we proceed backwards in time, we can merge two lineages by adding a vertex of in-degree one and out-degree two, representing a common ancestor.  We can also split a lineage into two distinct ones by adding a vertex of out-degree one and in-degree two, representing a recombinant.  We require that the phylogenetic graph we create is rooted, so any split must eventually resolve itself by a merge further back in time.  Once all lineages merge into the common ancestor, the graph-generating process stops; one can then generate the mutations at each vertex, using, e.g., a Poisson-type model as in \cref{Ex:Poisson}.  However, in this section, we will be concerned only with the underlying phylogenetic graph.

Two parameters are needed to specify the coalescent with recombination's graph generation process: the number of leaves $n$ and a recombination rate parameter $\rho$.    
The rate parameter equals twice the expected number of recombination events occurring in the entire population, per generation.   

\subsection{Probability of Generating a Galled Tree as the Solution of a Linear System}
Let $P(n,\rho)$ denote the probability that the coalescent with recombination generates a galled tree, given the parameters $n$ and $\rho$.  For fixed $n$, we derive a system of $\mathcal O(n^2)$ linear equations, depending on $\rho$, whose solution gives an analytic expression for $P(n,\rho)$ as a function of $\rho$.  As $n$ grows large, this expression becomes very complicated.  But each linear system is sparse, so it is easy to solve for $P(n,\rho)$ numerically for fixed values of $n$ and $\rho$, provided $n$ is not too large; see \cref{fig:prob_approx}.

In the coalescent model, there are two types of \emph{disallowed interactions} whose occurrence prevents the resulting graph from being a galled tree. First, after a split occurs, one of the two resulting branches may again split, prior to the resolution of the first split (\cref{fig:coalescentfigure}a).  Second, two unrelated splits may occur (resulting in four parental branches), after which a branch from one split joins with a branch from the other split (\cref{fig:coalescentfigure}b). 

To compute the probability $P(n,\rho)$ that the coalescent generates a galled tree, we track the number of lineages ($k$) and unresolved splits ($s$) at each step of the process.  The evolution of $k$ and $s$ is described by a discrete-time Markov chain whose state space is the finite set $T \cup \{X\},$ where
\[T:=\left\{(k,s)\, \middle|\, 1\leq k \leq 2n,\  \max(0,k-n)\leq s\leq \frac{k}{2}\right\},\]
and $X$ is an absorbing ``failure state" that we enter into when a disallowed interaction occurs.  $P(n,\rho)$ is the probability that we eventually reach state $(1,0)$ in this Markov Chain, starting from $(n,0)$.  

To complete the description of the Markov chain, we first specify the transitions that can occur, and then specify the probabilities of each these.  No self-transitions occur, with the exception that $X$ and $(1,0)$ are both absorbing states.  No two split or join events occur simultaneously in the coalescent; each transition thus corresponds to a single split or merge.

There are two types of splits: We may have a disallowed split, as described above and illustrated in (\cref{fig:coalescentfigure}a); or the split may be allowed, in which case the state transition is $(k,s) \mapsto (k+1,s+1)$ (\cref{fig:coalescentfigure}c).  
There are three types of joins:  We may have a disallowed join, as described above and illustrated in (\cref{fig:coalescentfigure}b); 
a split may resolve itself (transition $(k,s) \mapsto (k-1,s-1)$) (\cref{fig:coalescentfigure}d), 
or two branches may join without altering any split (transition $(k,s) \mapsto (k-1,s)$) (\cref{fig:coalescentfigure}e).  

The transition probabilities are obtained as ratios of rates: In the coalescent, for $k\geq 2$ the rate $r_{\spl}$ at which a split occurs is defined to be $\rho k/2$, and the rate $r_{\jn}$ at which two branches join together is defined to be $k(k-1)/2$ \cite{Wakeley:2007ua}.  The total rate of a split or merge is then $r_{\total}:=r_{\spl}+r_{\jn}$, i.e., $r_{\total}=k(k+\rho -1)/2$.

The rate $r_{\mathrm{\asplit}}$ of an allowed split is the product of $r_{\spl}$ with the fraction of branches for which splits are allowed, i.e., $r_{\mathrm{\asplit}}=\rho \left(k/2 - s\right)$.  For $k\geq 2$, let $p_S(k,s)$ denote the probability of an allowed split (i.e., a transition $(k,s)\mapsto (k,s+1)$).  Then \[p_S(k,s)= r_{\mathrm{\asplit}}/r_{\total}= \frac{\rho (k-2s)}{k(k+\rho-1)}.\]  The other transition probabilities are obtained analogously.  We denote them as follows: $p_{R}(k,s)$ is the probability that a split is resolved; $p_{J}(k,s)$ is the probability of an allowed join that does not resolve a split; and $p_{X}(k,s)$ is the probability of a  disallowed interaction of either type.  The formulas for these are given in \cref{fig:coalescentfigure}.  

\begin{figure}
\includegraphics[width=1.0\textwidth]{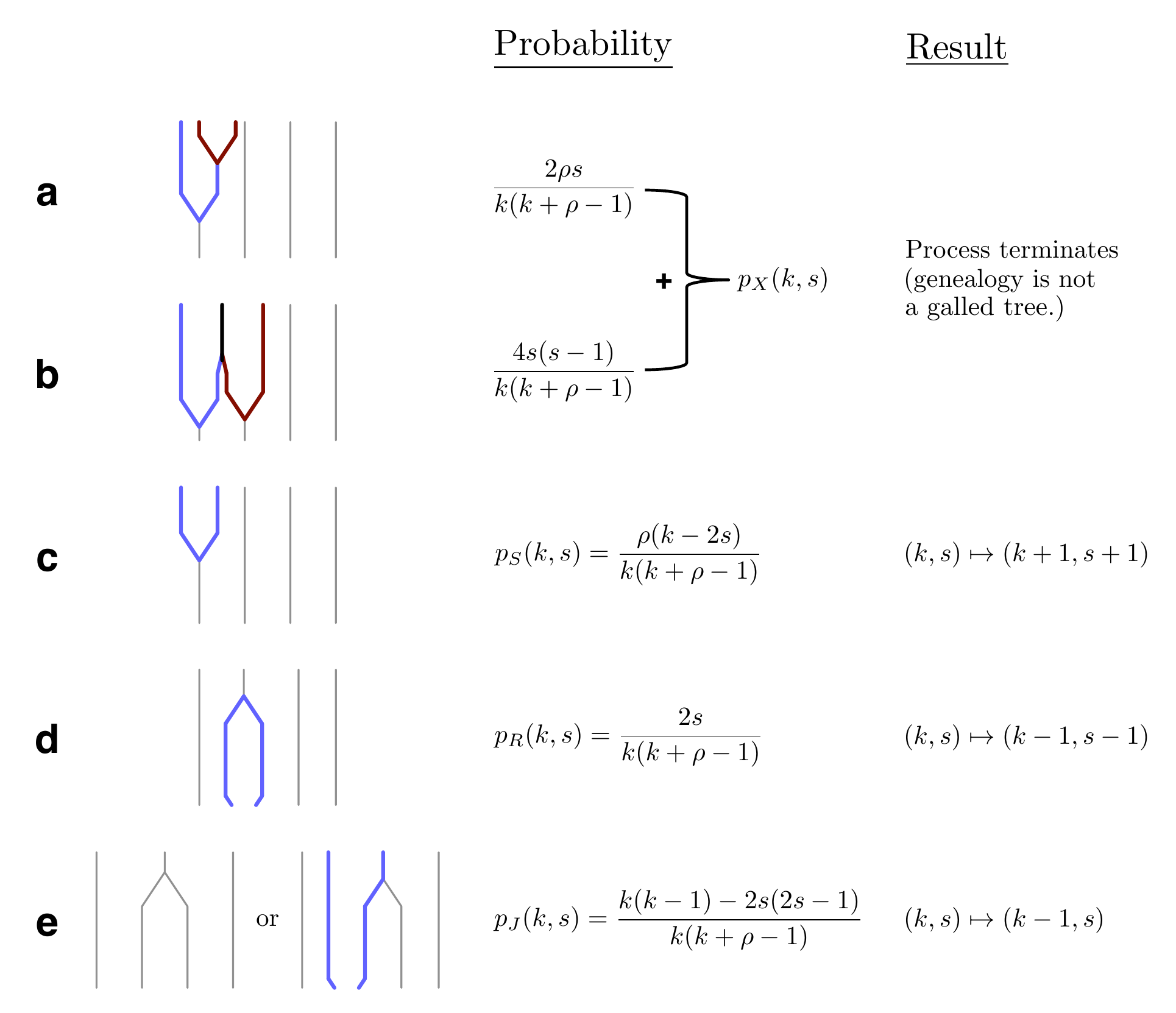}
\caption{Possible events in the coalescent process described in the text. Light gray lines indicate ordinary lineages, medium blue and dark red lines indicate pairs of parental lineages that split from their recombinant child lineages, and the black line indicates a join between the red and blue lineages. Each diagram is read upwards (back in time). To the right of each diagram are the corresponding transition probability and the resulting change in state of the Markov chain, in terms of population-scaled recombination rate $\rho$, number of lineages $k$, and number of unresolved splits $s$. Where multiple events are depicted, the probability is given only for the topmost (most ancient) event. (a) A disallowed split occurs when a lineage that has already split (light blue) splits again (dark red). (b) A disallowed join occurs when parental lineages from two separate splitting events (blue, red) join together (black).  (c) An allowed split.  (d) A join that resolves a split. (e)  An allowed join that does not resolve a split.}
\label{fig:coalescentfigure}
\end{figure}

For $(k,s)\in T$, let $f(k,s)$ be the probability that we eventually generate a galled tree, given that the current state is $(k,s)$, and for $(k,s)\in \Z^2\setminus T$, let $f(k,s)=0$.  Thus, $P(n,\rho)=f(n,0)$.  The $f(k,s)$ satisfy the linear system
\begin{align}
\begin{split}
\label{eq:recurrence_full}
f(k,s) &= p_{S}(k,s) f(k+1,s+1) \\
         &+p_{R}(k,s) f(k-1,s-1)\\
         &+ p_{J}(k,s) f(k-1,s) \quad \textup{ for } (k,s)\in T\setminus \{(1,0)\},\\
         f(1,0) &= 1.
\end{split}
\end{align}

\cref{fig:prob_approx} gives values of $P(n,\rho)$ for several choices of $\rho$ and $n$, obtained by solving this linear system numerically.  We observe that for fixed $n$, $P(n,\rho)$ tends to 1 as $\rho$ tends to 0.  Varying both parameters, $1-P(n,\rho)$ appears to decrease on the order $\rho^2 \left(\text{log } n\right)^2$ for large $n$ and small $\rho$.

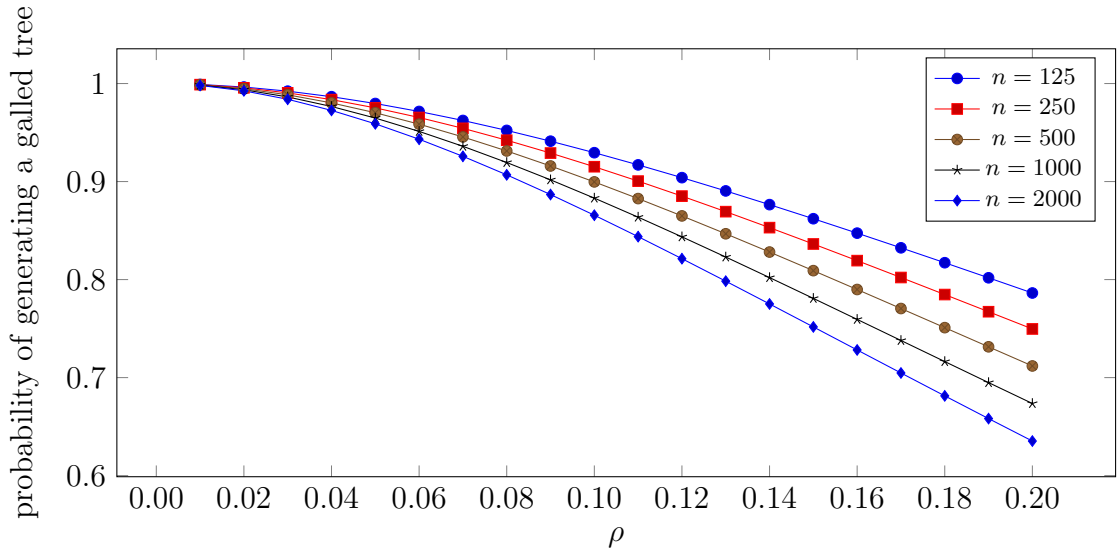
\begin{figure}
\centering
\begin{tikzpicture}
\begin{axis}[
  xlabel=$\rho$,
  ylabel=probability of generating a galled tree,width=.9\textwidth,height=\axisdefaultheight, x tick label style={
        /pgf/number format/.cd,
            fixed,
            fixed zerofill,
            precision=2,
        /tikz/.cd
    }]
\addplot table [x index=5, y index=0,col sep=comma]{Galled_Tree_Probabilities.csv};
\addlegendentry{\scriptsize $n=125$}
\addplot table [x index=5, y index=1,col sep=comma]{Galled_Tree_Probabilities.csv};
\addlegendentry{\scriptsize $n=250$}
\addplot table [x index=5, y index=2,col sep=comma]{Galled_Tree_Probabilities.csv};
\addlegendentry{\scriptsize $n=500$}
\addplot table [x index=5, y index=3,col sep=comma]{Galled_Tree_Probabilities.csv};
\addlegendentry{\scriptsize $n=1000$}
\addplot table [x index=5, y index=4,col sep=comma]{Galled_Tree_Probabilities.csv};
\addlegendentry{\scriptsize $n=2000$}
\end{axis}
\end{tikzpicture}
\caption{Probability that the coalescent with recombination yields a galled tree, for several values of the recombination rate parameter $\rho$ and the number of sampled genomes $n$.  We see that for fixed $n$, the probability of obtaining a galled tree tends to 1 as $\rho$ tends to 0.}
\label{fig:prob_approx}
\end{figure}

To give a sense of scale for human population genetics (species effective population size $N_e \approx 10^4$, recombination rate $c \approx 10^-8~\text{bp}^{-1}$), in a sample of 125 individuals ($n=250$ haploid genomes), to ensure that disallowed interactions occur with probability less than $0.1$ ($P(n,\rho) > 0.9$), the requirement shown in \cref{fig:prob_approx} is $\rho < 0.11$. Using $\rho = 4 N_e c L$, where $L$ is the length of the genome segment analyzed, this requirement becomes $L < 275$ bp. In a sample of one thousand ($n=2000$ in \cref{fig:prob_approx}), the requirement is instead $\rho < 0.08$, or $L < 200$ bp. Extrapolating from the apparent $\rho^2 \left(\text{log } n\right)^2$ scaling, in a sample of one million ($n=2\times 10^6$), the requirement is $\rho < 0.04$, or $L < 100$ bp.

\section{Subsamples Rarely Violate our Theoretical Bounds for Complete Samples}\label{Sec:Violations}
\cref{Ex:Subsample_Counterexample} makes clear that \cref{Thm:Lower_Bound}\,(i), our persistent homology lower bound on the novelty profile of an evolutionary history, does not hold for arbitrary samples $S$ of the history $\eh$.  Nevertheless, one might hope that violations of \cref{Thm:Lower_Bound} are relatively rare.  Here we use simulations to explore this question in the coalescent model with recombination.  

Our computations focus on how often the number of intervals in $\B_1(S)$ exceeds the number of recombinants in $\eh$.  We find that in our simulations, this happens quite rarely,  though it does occur.  We did not consider the frequency of other kinds of violations 
of \cref{Thm:Lower_Bound} for subsamples, though it would be interesting to do so.

We simulated over 42{,}000 evolutionary histories, assuming a constant population size and using parameters $n=10,$ 15, or 20; $\rho=1,$ 2, 3, or 4; and $\theta=5,$ 10, or 30, where $\theta$ is the mutation rate parameter for the coalescent \cite{Wakeley:2007ua}.  We used rejection sampling, retaining only those histories that were indexed by galled trees.  Consistent with \cref{Sec:Phylogenetic_graphs_and_populations}, we used an infinite-sites model of mutation.  Each genetic site in a recombinant offspring inherited the state of a parent with probability one-half. For each simulated history, we counted the number of detectable recombination events that took place -- defined as the number of events giving rise to a recombinant that generates an incompatibility according to the four-gamete test \cite{Hudson:1985wh}. We then computed the maximum $|\B_1(S)|$ among 2500 random subsamples $S$ of the history, with 5 to 30 genomes in each subsample. Among all simulations, nine had a maximum $|\B_1(S)|$ greater than the true number of detectable recombination events (\cref{Tbl:b1_versus_numRC_in_coalsims}). The rarity of this violation suggests that counterexamples such as \cref{Ex:Subsample_Counterexample} may be uncommon in actual population-genetic data.

\begin{table}[!h]
\centering
\caption{Counts of coalescent simulations, by number of detectable recombination events and maximum number of intervals in $\B_1(S)$ among all subsamples $S$.  For nine of the simulations (shown in red), the maximum number of intervals in $\B_1(S)$ exceeds the number of detectable recombinations.}
\label{Tbl:b1_versus_numRC_in_coalsims}
\begin{tabular}{cccccccc}
                                                                                                                 &   & \multicolumn{6}{c}{Max. $|\B_1(S)|$ among all subsamples $S$} \\                                                                                                                                                                                                                       
                                                                                                                 &   & \multicolumn{1}{|c|}{0}     & \multicolumn{1}{c|}{1}                         & \multicolumn{1}{c|}{2}                                                & \multicolumn{1}{c|}{3}                         & \multicolumn{1}{c|}{4}                         & 5                         \\                                                                                                                                                              
                                                                                                                 \cline{3-8} 
                                                                                                                 \hhline{~|*{5}{-}}
                                                                                                                 & 0 & \multicolumn{1}{|c|}{17953} & \multicolumn{1}{c|}{0} & \multicolumn{1}{c|}{0}                        & \multicolumn{1}{c|}{0} & \multicolumn{1}{c|}{0} &  \multicolumn{1}{c}{0}   \\
                                                                                                                 & 1 & \multicolumn{1}{|c|}{9346}  & \multicolumn{1}{c|}{10049}                     & \multicolumn{1}{c|}{\textcolor{red}{9}} & \multicolumn{1}{c|}{0} & \multicolumn{1}{c|}{0} &  \multicolumn{1}{c}{0}  \\
                                                                                                                 & 2 & \multicolumn{1}{|c|}{1159}  & \multicolumn{1}{c|}{2632}                      & \multicolumn{1}{c|}{714}                                              & \multicolumn{1}{c|}{0} & \multicolumn{1}{c|}{0} &   \multicolumn{1}{c}{0} \\
                                                                                                                 & 3 & \multicolumn{1}{|c|}{126}   & \multicolumn{1}{c|}{395}                       & \multicolumn{1}{c|}{182}                                              & \multicolumn{1}{c|}{14}                        & \multicolumn{1}{c|}{0} & \multicolumn{1}{c}{0}   \\
                                                                                                                 & 4 & \multicolumn{1}{|c|}{16}    & \multicolumn{1}{c|}{48}                        & \multicolumn{1}{c|}{39}                                               & \multicolumn{1}{c|}{3}                         & \multicolumn{1}{c|}{0}                         &  \multicolumn{1}{c}{0}  \\
\multirow{-6}{*}{\begin{tabular}[c]{@{}c@{}}Number of\\ detectable\\ recombination\\ events in history\end{tabular}} & 5 & \multicolumn{1}{|c|}{0}     & \multicolumn{1}{c|}{5}                         & \multicolumn{1}{c|}{2}                                                & \multicolumn{1}{c|}{0}                         & \multicolumn{1}{c|}{0}                         & 0                        
\end{tabular}
\end{table}

\bibliographystyle{abbrv}
{\small \bibliography{Recombination_Refs} }

\begin{thebibliography}{10}

\bibitem{adamaszek2017vietoris}
{\sc M.~Adamaszek and H.~Adams}, {\em The {V}ietoris-{R}ips complexes of a
  circle}, Pacific Journal of Mathematics, 290 (2017), pp.~1--40.

\bibitem{adamaszek2017vietorisGluing}
{\sc M.~Adamaszek, H.~Adams, E.~Gasparovic, M.~Gommel, E.~Purvine,
  R.~Sazdanovic, B.~Wang, Y.~Wang, and L.~Ziegelmeier}, {\em Vietoris-rips and
  cech complexes of metric gluings}, in 34th International Symposium on
  Computational Geometry (SoCG 2018), Schloss Dagstuhl-Leibniz-Zentrum fuer
  Informatik, 2018.

\bibitem{adamaszek2017vietorisB}
{\sc M.~Adamaszek, H.~Adams, and S.~Reddy}, {\em On {V}ietoris-{R}ips complexes
  of ellipses}, Journal of Topology and Analysis,  (2017), pp.~1--30.

\bibitem{Arenas:2008hb}
{\sc M.~Arenas, G.~Valiente, and D.~Posada}, {\em {Characterization of
  reticulate networks based on the coalescent with recombination.}}, Molecular
  biology and evolution, 25 (2008), pp.~2517--2520.

\bibitem{bauer2015induced}
{\sc U.~Bauer and M.~Lesnick}, {\em Induced matchings and the algebraic
  stability of persistence barcodes}, Journal of Computational Geometry, 6
  (2015), pp.~162--191.

\bibitem{Bezemer:2010hu}
{\sc D.~Bezemer, A.~van Sighem, V.~V. Lukashov, L.~van~der Hoek, N.~Back,
  R.~Schuurman, C.~A.~B. Boucher, E.~C.~J. Claas, M.~C. Boerlijst, R.~A.
  Coutinho, F.~de~Wolf, and {ATHENA observational cohort}}, {\em {Transmission
  networks of HIV-1 among men having sex with men in the Netherlands.}}, AIDS
  (London, England), 24 (2010), pp.~271--282.

\bibitem{blumberg2017universality}
{\sc A.~J. Blumberg and M.~Lesnick}, {\em Universality of the homotopy
  interleaving distance}, arXiv preprint arXiv:1705.01690,  (2017).

\bibitem{Camara:2016kl}
{\sc P.~G. C{\'a}mara, A.~J. Levine, and R.~Rabad{\'a}n}, {\em Inference of
  ancestral recombination graphs through topological data analysis}, PLoS
  Comput Biol, 12 (2016), p.~e1005071.

\bibitem{Camara:2016eh}
{\sc P.~G. C\'amara, D.~I.~S. Rosenbloom, K.~J. Emmett, A.~J. Levine, and
  R.~Rabad\'an}, {\em {Topological Data Analysis Generates High-Resolution,
  Genome-wide Maps of Human Recombination.}}, Cell systems, 3 (2016),
  pp.~83--94.

\bibitem{carlsson2009topology}
{\sc G.~Carlsson}, {\em Topology and data}, Bulletin of the American
  Mathematical Society, 46 (2009), pp.~255--308.

\bibitem{carlsson2014topological}
{\sc G.~Carlsson}, {\em Topological pattern recognition for point cloud data},
  Acta Numerica, 23 (2014), pp.~289--368,
  \url{https://doi.org/10.1017/s0962492914000051}.

\bibitem{chan2012genome}
{\sc A.~H. Chan, P.~A. Jenkins, and Y.~S. Song}, {\em Genome-wide fine-scale
  recombination rate variation in drosophila melanogaster}, PLoS genetics, 8
  (2012), p.~e1003090.

\bibitem{chan2013topology}
{\sc J.~Chan, G.~Carlsson, and R.~Rabad\'an}, {\em Topology of viral
  evolution}, Proceedings of the National Academy of Sciences, 110 (2013).

\bibitem{chazal2009proximity}
{\sc F.~Chazal, D.~Cohen-Steiner, M.~Glisse, L.~Guibas, and S.~Oudot}, {\em
  {Proximity of persistence modules and their diagrams}}, in Proceedings of the
  25\textsuperscript{th} annual symposium on Computational geometry, ACM, 2009,
  pp.~237--246.

\bibitem{chazal2009gromov}
{\sc F.~Chazal, D.~Cohen-Steiner, L.~Guibas, F.~M{\'e}moli, and S.~Oudot}, {\em
  {Gromov-Hausdorff stable signatures for shapes using persistence}}, in
  Proceedings of the Symposium on Geometry Processing, Eurographics
  Association, 2009, pp.~1393--1403.

\bibitem{chazal2012structure}
{\sc F.~Chazal, V.~de~Silva, M.~Glisse, and S.~Oudot}, {\em The Structure and
  Stability of Persistence Modules}, Springer International Publishing, 2016,
  \url{https://doi.org/10.1007/978-3-319-42545-0},
  \url{https://doi.org/10.1007%2F978-3-319-42545-0}.

\bibitem{chazal2014persistence}
{\sc F.~Chazal, V.~De~Silva, and S.~Oudot}, {\em Persistence stability for
  geometric complexes}, Geometriae Dedicata, 173 (2014), pp.~193--214.

\bibitem{cohen2007stability}
{\sc D.~Cohen-Steiner, H.~Edelsbrunner, and J.~Harer}, {\em {Stability of
  persistence diagrams}}, Discrete and Computational Geometry, 37 (2007),
  pp.~103--120.

\bibitem{crawley2012decomposition}
{\sc W.~Crawley-Boevey}, {\em Decomposition of pointwise finite-dimensional
  persistence modules}, Journal of Algebra and Its Applications, 14 (2015),
  p.~1550066.

\bibitem{davies2010origins}
{\sc J.~Davies and D.~Davies}, {\em Origins and evolution of antibiotic
  resistance}, Microbiology and Molecular Biology Reviews, 74 (2010),
  pp.~417--433.

\bibitem{durrett2010probability}
{\sc R.~Durrett}, {\em Probability: theory and examples}, Cambridge University
  Press, 2010.

\bibitem{edelsbrunner2010computational}
{\sc H.~Edelsbrunner and J.~Harer}, {\em Computational topology: an
  introduction}, American Mathematical Society, 2010.

\bibitem{eden2013recombination}
{\sc J.~S. Eden, M.~M. Tanaka, M.~F. Boni, W.~D. Rawlinson, and P.~A. White},
  {\em Recombination within the pandemic norovirus gii. 4 lineage}, Journal of
  Virology, 87 (2013), pp.~6270--6282.

\bibitem{Emmett:2014us}
{\sc K.~Emmett, D.~Rosenbloom, P.~C\'amara, and R.~Rabad\'an}, {\em {Parametric
  inference using persistence diagrams: A case study in population genetics}},
  Proc. 31st Intl. Conf. Machine Learning,  (2014),
  \url{https://arxiv.org/abs/C21B8DD4-D6F5-4057-ADCC-BAF5DCC28A7B}.

\bibitem{Emmett:2014um}
{\sc K.~J. Emmett and R.~Rabad\'an}, {\em {Characterizing Scales of Genetic
  Recombination and Antibiotic Resistance in Pathogenic Bacteria Using
  Topological Data Analysis}}, Lecture Notes in Computer Science, 8609 (2014),
  pp.~540--551.

\bibitem{forman2002user}
{\sc R.~Forman}, {\em A user's guide to discrete {M}orse theory}, S\'em.
  Lothar. Combin., 48 (2002), pp.~Art.\ B48c, 35.

\bibitem{gusfield2014recombinatorics}
{\sc D.~Gusfield}, {\em ReCombinatorics: the algorithmics of ancestral
  recombination graphs and explicit phylogenetic networks}, MIT Press, 2014.

\bibitem{gusfield2003efficient}
{\sc D.~Gusfield, S.~Eddhu, and C.~Langley}, {\em Efficient reconstruction of
  phylogenetic networks with constrained recombination}, in Bioinformatics
  Conference, 2003. CSB 2003. Proceedings of the 2003 IEEE, IEEE, 2003,
  pp.~363--374.

\bibitem{levanger2019comparison}
{\sc S.~Harker, M.~Kram{\'a}r, R.~Levanger, and K.~Mischaikow}, {\em A
  comparison framework for interleaved persistence modules}, Journal of Applied
  and Computational Topology, 3 (2019), pp.~85--118,
  \url{https://doi.org/10.1007/s41468-019-00026-x},
  \url{https://doi.org/10.1007/s41468-019-00026-x}.

\bibitem{hatcher2002algebraic}
{\sc A.~Hatcher}, {\em {Algebraic topology}}, Cambridge University Press, 2002.

\bibitem{Hudson:1985wh}
{\sc R.~R. Hudson and N.~L. Kaplan}, {\em {Statistical properties of the number
  of recombination events in the history of a sample of DNA sequences.}},
  Genetics, 111 (1985), pp.~147--164.

\bibitem{humphreys2019fast}
{\sc D.~P. Humphreys, M.~R. McGuirl, M.~Miyagi, and A.~J. Blumberg}, {\em Fast
  estimation of recombination rates using topological data analysis}, Genetics,
   (2019), pp.~genetics--301565.

\bibitem{huson2010phylogenetic}
{\sc D.~H. Huson, R.~Rupp, and C.~Scornavacca}, {\em Phylogenetic networks:
  concepts, algorithms and applications}, Cambridge University Press, 2010.

\bibitem{ito1998molecular}
{\sc T.~Ito, J.~N. S.~S. Couceiro, S.~Kelm, L.~G. Baum, S.~Krauss, M.~R.
  Castrucci, I.~Donatelli, H.~Kida, J.~C. Paulson, R.~G. Webster, and
  Y.~Kawaoka}, {\em Molecular basis for the generation in pigs of influenza {A}
  viruses with pandemic potential}, Journal of Virology, 72 (1998),
  pp.~7367--7373.

\bibitem{kahle2011random}
{\sc M.~Kahle}, {\em Random geometric complexes}, Discrete \& Computational
  Geometry, 45 (2011), pp.~553--573.

\bibitem{kallenberg2006foundations}
{\sc O.~Kallenberg}, {\em Foundations of modern probability}, Springer Science
  \& Business Media, 2006.

\bibitem{kozlov2008combinatorial}
{\sc D.~Kozlov}, {\em Combinatorial algebraic topology}, vol.~21 of Algorithms
  and Computation in Mathematics, Springer, Berlin, 2008,
  \url{https://doi.org/10.1007/978-3-540-71962-5},
  \url{http://dx.doi.org/10.1007/978-3-540-71962-5}.

\bibitem{lauritzen1996graphical}
{\sc S.~L. Lauritzen}, {\em Graphical Models}, Oxford University Press, 1996.

\bibitem{munkres1984elements}
{\sc J.~Munkres}, {\em {Elements of Algebraic Topology}}, Prentice Hall, 1984,
  \url{http://www.worldcat.org/isbn/0131816292}.

\bibitem{Myers:2003tv}
{\sc S.~R. Myers and R.~C. Griffiths}, {\em {Bounds on the minimum number of
  recombination events in a sample history.}}, Genetics, 163 (2003),
  pp.~375--394.

\bibitem{nora2007contribution}
{\sc T.~Nora, C.~Charpentier, O.~Tenaillon, C.~Hoede, F.~Clavel, and A.~J.
  Hance}, {\em Contribution of recombination to the evolution of human
  immunodeficiency viruses expressing resistance to antiretroviral treatment},
  Journal of virology, 81 (2007), pp.~7620--7628.

\bibitem{oudot2015persistence}
{\sc S.~Y. Oudot}, {\em Persistence Theory: From Quiver Representations to Data
  Analysis}, no.~209 in AMS Mathematical Surveys and Monographs, American
  Mathematical Society, 2015.

\bibitem{parida2015topological}
{\sc L.~Parida, F.~Utro, D.~Yorukoglu, A.~P. Carrieri, D.~Kuhn, and S.~Basu},
  {\em Topological signatures for population admixture}, in International
  Conference on Research in Computational Molecular Biology, Springer, 2015,
  pp.~261--275.

\bibitem{rabadan2019topological}
{\sc R.~Rabad{\'a}n and A.~J. Blumberg}, {\em Topological Data Analysis for
  Genomics and Evolution: Topology in Biology}, Cambridge University Press,
  2019.

\bibitem{Rasmussen:2014cq}
{\sc M.~D. Rasmussen, M.~J. Hubisz, I.~Gronau, and A.~Siepel}, {\em
  {Genome-Wide Inference of Ancestral Recombination Graphs}}, PLoS Genetics, 10
  (2014), p.~e1004342.

\bibitem{rohde2011open}
{\sc H.~Rohde, J.~Qin, Y.~Cui, D.~Li, N.~J. Loman, M.~Hentschke, W.~Chen,
  F.~Pu, Y.~Peng, J.~Li, et~al.}, {\em Open-source genomic analysis of
  shiga-toxin-producing e. coli o104: H4}, New England Journal of Medicine, 365
  (2011), pp.~718--724.

\bibitem{solovyov2009cluster}
{\sc A.~Solovyov, G.~Palacios, T.~Briese, W.~I. Lipkin, and R.~Rabad\'an}, {\em
  Cluster analysis of the origins of the new influenza {A} ({H1N1}) virus},
  Euro surveillance: bulletin Europeen sur les maladies transmissibles =
  European communicable disease bulletin, 14 (2009).

\bibitem{Song:2005cg}
{\sc Y.~S. Song, Y.~Wu, and D.~Gusfield}, {\em {Efficient computation of close
  lower and upper bounds on the minimum number of recombinations in biological
  sequence evolution.}}, Bioinformatics (Oxford, England), 21 Suppl 1 (2005),
  pp.~i413--22.

\bibitem{Stadler:2011hn}
{\sc T.~Stadler, R.~Kouyos, V.~von Wyl, S.~Yerly, J.~Boni, P.~Burgisser,
  T.~Klimkait, B.~Joos, P.~Rieder, D.~Xie, H.~F. Gunthard, A.~J. Drummond,
  S.~Bonhoeffer, and {the Swiss HIV Cohort Study}}, {\em {Estimating the Basic
  Reproductive Number from Viral Sequence Data}}, Molecular biology and
  evolution, 29 (2011), pp.~347--357.

\bibitem{theSwissHIVCohortStudy:2010up}
{\sc {The Swiss HIV Cohort Study}}, {\em {Cohort profile: the Swiss HIV Cohort
  study}}, International Journal of Epidemiology, 39 (2010), pp.~1179--1189.

\bibitem{trifonov2009geographic}
{\sc V.~Trifonov, H.~Khiabanian, R.~Rabad\'an, et~al.}, {\em Geographic
  dependence, surveillance, and origins of the 2009 influenza a (h1n1) virus.},
  New England Journal of Medicine, 361 (2009), pp.~115--119.

\bibitem{Wakeley:2007ua}
{\sc J.~Wakeley}, {\em {Coalescent Theory}}, An Introduction, Roberts {\&} Co.,
  2007.

\bibitem{wang2001perfect}
{\sc L.~Wang, K.~Zhang, and L.~Zhang}, {\em Perfect phylogenetic networks with
  recombination}, Journal of Computational Biology, 8 (2001), pp.~69--78.

\end{thebibliography}

\end{document}